\documentclass[12pt]{amsart}
\usepackage[margin= 1 in]{geometry}

\usepackage{pdfsync}

 \usepackage[plainpages=false]{hyperref}
 \usepackage{amsfonts,amsmath,amssymb,amsthm,accents,mathtools}
 \usepackage{latexsym,lscape,rawfonts,mathrsfs}

 \usepackage[all]{xy}
 \usepackage{eufrak}
\usepackage{psfrag}
 \usepackage{graphicx} 
 \usepackage{pstool}
 \usepackage{epstopdf}

 \usepackage{array,tabularx}

 \usepackage{appendix}

\usepackage{txfonts}

 \newcommand{\ba}{\begin{align}}
 \newcommand{\ea}{\end{align}}
 \newcommand{\bal}{\begin{align*}}
 \newcommand{\eal}{\end{align*}}

 \DeclareMathOperator{\diam}{diam}

 \newcommand{\Rm}{\mathbf{Rm}}
 \newcommand{\Rc}{\mathbf{Rc}}

 \newcommand{\dvol}{\text{d}V}

\renewcommand{\epsilon}{\varepsilon}

\newcommand{\rank}{\textbf{rank}}

 \makeatletter
 \def\ExtendSymbol#1#2#3#4#5{\ext@arrow 0099{\arrowfill@#1#2#3}{#4}{#5}}
 
 \makeatother

 \makeatletter
 \def\ExtendSymbol#1#2#3#4#5{\ext@arrow 0099{\arrowfill@#1#2#3}{#4}{#5}}
 
 \makeatother

 \definecolor{hao}{rgb}{1,0.5,0}
 \definecolor{miao}{cmyk}{0.5,0,0.2,0.2}
 \definecolor{qiao}{gray}{0.96}

\newtheorem{prop}{Proposition}[section]

\newtheorem{proposition}[prop]{Proposition}

\newtheorem{theorem}[prop]{Theorem}

\newtheorem{lemma}[prop]{Lemma}

\newtheorem{corollary}[prop]{Corollary}
\newtheorem{remarkin}[prop]{Remark}

\newtheorem*{theorem*}{Theorem}
\theoremstyle{remark}
\newtheorem{remark}{Remark}

\numberwithin{equation}{section}

\keywords{Betti number, collapsing, torus fibration, nilpotency rank, Ricci
flow.}

\address{Shaosai Huang, Two Morneau Shepell Centre Suite 915, 895 Don Mills Road, Toronto, ON, M3C 1W3, Canada}
\email{arthur.ut.ca@gmail.com}

\address{Bing Wang, Institute of Geometry and Physics, and School of
Mathematical Sciences, University of Science and Technology of China, 96
Jinzhai Road, Hefei, Anhui Province, 230026, China}
\email{topspin@ustc.edu.cn}

\title[Rigidity of the first Betti number]{Rigidity of the first
Betti number via Ricci flow smoothing}
\author{Shaosai Huang}
\author{Bing Wang}
\date{\today}

\begin{document}
\maketitle

\begin{abstract}
The Colding-Gromov gap theorem asserts that an almost non-negatively Ricci 
curved manifold with unit diameter and maximal first Betti number is
homeomorphic to the flat torus. In this paper, we prove a parametrized version
of this theorem, in the context of collapsing Riemannian manifolds with Ricci
curvature bounded below: if a closed manifold with Ricci curvature uniformly 
bounded below is Gromov-Hausdorff close to a (lower dimensional) manifold
with bounded geometry, and has the difference of their first Betti numbers 
equal to the dimensional difference, then it is diffeomorphic to a torus bundle
over the one with bounded geometry. We rely on two novel technical tools: the
first is an effective control of the spreading of minimal geodesics with initial
data parallel transported along a short geodesic segment, and the second is a
Ricci flow smoothing result for certain collapsing initial data with Ricci
curvature bounded below.
\end{abstract}

\tableofcontents

\section{Introduction}
\addtocontents{toc}{\protect\setcounter{tocdepth}{1}}

The classical Bochner technique (see \cite{Bochner, BY}) implies that a
closed Riemannian manifold with non-negative Ricci curvature has its first
Betti number bounded above by its dimension, with the equality case only
achieved by the flat torus. This is a rarely found topological rigidity theorem 
for Riemannian manifolds with Ricci curvature bounded below.

After more than three decades since the birth of Bochner's technique, Gromov
\cite[Page 75]{GLP} conjectured a quantitative gap phenomenon, expecting the
existence of a small dimensional constant $\delta_{G}>0$, such that if the
Ricci curvature of a closed Riemannian manifold (with unit diameter) has its
lowest eigenvalue bounded below by $-\delta_{G}$, then the first Betti number
does not exceed the dimension, while equality warrants the toral structure of
the manifold. This conjecture was later proven by Colding \cite{Colding97}
based on his renowned volume continuity theorem; see also \cite{ChCoI}.

The Colding-Gromov gap theorem is akin to Gromov's almost flat manifold theorem
(see \cite{Gromov78b, Ruh}), the latter asserting the existence of a
dimensional gap $\delta_{AF}>0$ such that if a closed Riemannian manifold
(with unit diameter) has sectional curvature bounded by $\delta_{AF}$ in
absolute value, then it is diffeomorphic to an infranil manifold.

While the almost flat manifold theorem is beautiful it only detects a  
special class of manifolds; it is the study of the collapsing geometry with
bounded curvature by Cheeger, Fukaya and Gromov (see \cite{CGI, CGII,
Fukaya87ld, Fukaya88, CFG92}) that fits this theorem in a much broader context
--- as suggested by \cite[Main Theorem]{Fukaya87ld}, if a Riemannian manifold
with uniformly bounded sectional curvature is Gromov-Hausdorff close
(\emph{collapsing}) to a lower dimensional one with bounded geometry, then it is
a fiber bundle over the lower dimensional manifold, with fibers being infranil
manifolds. Putting it another way, one could think of the collapsing manifold
as a collection of infranil manifolds smoothly parametrized by the collapsing
limit space (assumed to be a lower dimensional manifold).

The purpose of the current paper is then to report a parametrized version of
the Colding-Gromov gap theorem, describing the collapsing behavior of
certain Riemannian manifolds with Ricci curvature bounded below. 
 Before stating our theorem, let us fix some notations. We let
 $\mathcal{M}_{Rc}(m)$ denote the collection of $m$-dimensional complete
 Riemannian manifolds $(M,g)$ with $\Rc_g\ge-(m-1)g$. We also let
 $\mathcal{M}_{Rm}(k,D,v)$ denote the collection of $k$-dimensional closed
 Riemannian manifolds with sectional curvature at any point not exceeding $1$
 in absolute value, diameter bounded above by $D\ge 1$ and volume bounded below
 by $v>0$.  

With these notations, our main result states as
\begin{theorem}[Rigidity of the first Betti number]\label{thm: main1}
Given the data $m\in \mathbb{N}$, $D\ge 1$ and $v>0$, there is a
constant $\delta_{B}(m,D,v)\in (0,1)$ such that if for some $(M,g)\in
\mathcal{M}_{Rc}(m)$ and some $(N,h)\in \mathcal{M}_{Rm}(k,D,v)$
(with $k\le m$) it holds $d_{GH}(M,N)<\delta_B$, then
\begin{enumerate}
  \item $b_1(M)-b_1(N)\le m-k$; and 
  \item if the equality holds, then $M$ is diffeomorphic to an $(m-k)$-torus
  bundle over $N$.
\end{enumerate}
\end{theorem}
\begin{remark}
Of course, $(N,h)\in \mathcal{M}_{Rm}(k,D,v)$ is just one way to describe
that $(N,h)$ has ``bounded geometry''. Alternative descriptions include
assuming that $\diam (N,h)\le D$, $\Rc_h\ge -(k-1)h$ and the $C^{1,\frac{1}{2}}$
harmonic radii at all points of $N$ are bounded below by $\iota\in (0,1)$ ---
in fact, $\delta_B(m,D,v)$ directly depends on the $C^{1,\frac{1}{2}}$ harmonic
radii lower bound, obtained in \cite{JK82} for manifolds in
$\mathcal{M}_{Rm}(k,D,v)$.
\end{remark}

While the collapsing phenomena of sequences of Riemannian manifolds with
bounded sectional curvature is well-understood thanks to the works of Cheeger,
Fukaya, Gromov and Rong \cite{CGI, CGII, Fukaya87ld, Fukaya88, CFG92, Rong93,
CR95, CR96}, the behavior of metrics when collapsing with only Ricci curvature
lower bound is much more complicated and much less understood. For instance,
even when a sequence of Ricci flat manifolds collapse to very regular limit
spaces, there may be no uniform curvature bound for the collapsing sequence, as
shown by examples in \cite{GW00, HSVZ18, LiYang17}. Beyond these recently
discovered examples, Theorem~\ref{thm: main1} provides a definite result that
helps us better understand the collapsing geometry with only Ricci curvature
bounded below. The strength of Theorem~\ref{thm: main1} lies in the fact that
 while the assumption on the first Betti numbers is only numerical, the
outcome provides a much more detailed structural description.

The torus fiber bundle structure predicted by Theorem~\ref{thm: main1} is even
simpler than the infranil fibration structure expected from the general theory
of collapsing geometry with bounded sectional curvature (see \cite{Fukaya87ld,
Fukaya88, CFG92}) --- it is the assumption on the first Betti numbers that
drastically reduces the topological complexity. We believe that the methods in
proving Theorem~\ref{thm: main1}, when further localized, should shed some
light on our understanding of the collapsing geometry of Ricci flat K\"ahler
manifolds, especially the SYZ conjecture \cite{SYZ96}.

We notice that the equality case in Claim (2) of Theorem~\ref{thm: main1} does
not apply to Berger's sphere, as $b_1(\mathbb{S}^3)=b_1(\mathbb{S}^2)=0$ --- in
fact, when $M$ is almost non-negatively Ricci curved, we expect that $M\cong
N\times \mathbb{T}^{m-k}$ in the equality case of Theorem~\ref{thm: main1}; compare Theorem~\ref{thm: rigidity_Rc_ge0}. On the other hand, examples of manifolds satisfying Claim (2) of Theorem~\ref{thm: main1} include, but are not limited to, non-positively curved compact manifolds that collapse to lower dimensional manifolds, as discussed by Cao, Cheeger and Rong in \cite{CCR}.

When $b_1(M)-b_1(N)=\dim M-\dim N$ in Theorem~\ref{thm: main1}, since $N$ is a smooth manifold, $M$ admits a pure $Cr$-structure of rank $m-k$, which is a special type of $F$-structure \emph{a la} Cheeger and Gromov \cite{CGI, CGII}; see \cite{Buyalo1, Buyalo2, Buyalo3} and \cite[Section 4]{CCR} for the definition of $Cr$-structure. It is easily seen that we
could construct an invariant metric with respect to such structure; see also
\cite{Molino88}:
\begin{corollary}
In the equality case of Theorem~\ref{thm: main1}, on $M$ there is a Riemannian
metric $g'$ which defines a distance function close to the original one induced
by $g$, and a regular Riemannian foliation on $(M,g')$ with leaves generated by
$m-k$ commuting Killing vector fields. Moreover, by shrinking $g'$ on the leaf
directions, there admits a family of Riemannian metrics on $M$ that (volume)
collapse with uniformly bounded diameter and sectional curvature. 
\end{corollary}

A basic concept in studying the local geometry of Riemannian manifolds in
$\mathcal{M}_{Rc}(m)$, as discussed in \cite{FY92, KW11, NaberZhang}, is the
fibered fundamental group, which takes into consideration those very short
loops based at a given point, and allowed to be deformed in a definite geodesic
ball centered at that point. More precisely, given $(M,g)\in
\mathcal{M}_{Rc}(m)$, for any $\delta\in (0,1)$ and any $p\in M$, the
\emph{fibered fundamental group} at $p$ is defined as
\begin{align*}
\Gamma_{\delta}(p)\ :=\ Image[\pi_1(B_g(p,\delta),p)\to \pi_1(B_g(p,2),p)].
\end{align*}
For suitably small $\delta$, it is known, through the work of Kapovitch and
Wilking (\cite[Theorem 1]{KW11}), that $\Gamma_{\delta}(p)$ is an almost
nilpotent group with nilpotency rank bounded above by $m=\dim M$.
In the setting of Theorem~\ref{thm: main1},  $M$ is
$\delta$-Gromov-Hausdorff close to some $(N,h)\in \mathcal{M}_{Rm}(k,D,v)$,
and the work of Naber and Zhang \cite{NaberZhang} provides more information: 
by \cite[Theorem 2.27]{NaberZhang} we know that $\rank\ \Gamma_{\delta}(p)\le
\dim M-\dim N$ when $\delta>0$ is sufficiently small, and \cite[Proposition
5.9]{NaberZhang} tells that when the equality holds, the universal covering
space of $B_g(p,2)$ is uniformly non-collapsing.

From this point of view, Theorem~\ref{thm: main1} could also be seen as a global
version of the above mentioned results on the fibered fundamental groups, in a
more natural situation --- notice that the conditions on the fibered fundamental
groups are purely local, and could hardly be checked at each and every
single point, whereas our considerations on the first Betti numbers in
Theorem~\ref{thm: main1} are global and topological. In fact, much of our effort
is devoted to ``localizing'' the information encoded in the first Betti numbers
to control the nilpotency rank of the fibered fundamental groups.

This ``localization'' is carried out by first locating those very short loops in
$M$. We collect all the first homology classes
that could be generated by loops of lengths not exceeding $10\delta$ in
$H_1^{\delta}(M;\mathbb{Z})$, which clearly is a subgroup of the abelian group
$H_1(M;\mathbb{Z})$, and we will show that $b_1(M)-b_1(N)=\rank\
H_1^{\delta}(M;\mathbb{Z})$ in Proposition~\ref{prop: rank}. Notice that if
$\gamma'$ is a geodesic loop based at some $p_0\in M$, representing a
torsion-free generator of $H_1^{\delta}(M;\mathbb{Z})$ with $|\gamma'|\le
10\delta$, we could perturb it in its \emph{free homotopy class} to find a
shortest representative $\gamma:[0,1]\to M$ --- this does not alter the
homology class of $\gamma'$ although in general $\gamma$ may no longer be a
loop passing through $p_0\in M$. The advantage of $\gamma$ is that it is a
\emph{closed geodesic}, rather than just being a geodesic loop. In the second
step, we will basically show that for $\delta$ sufficiently small, if we slide
$\gamma$ along a minimal geodesic initially perpendicular to $\gamma$, it will then 
end up with being a geodesic loop of length comparable to $\delta$. In this
way, if $\gamma$ generates a torsion-free class in $H_1^{\delta}(M;\mathbb{Z})$, then
sliding it to another point $p\in M\backslash \gamma([0,1])$ will produce a
loop contained in $B_g(p,\bar{C}(m,D)\delta)$, with $D\ge \diam M$ and $m=\dim
M$. Making $\delta$ sufficiently small, we could make sure that any torsion-free
class in $H_1^{\delta}(M;\mathbb{Z})$ defines a torsion-free homotopy class in
$\tilde{\Gamma}_{\bar{C}\delta}(p)$, therefore bounding $\rank\
\tilde{\Gamma}_{\bar{C}\delta}(p)$ from below by $\rank\
H_1^{\delta}(M;\mathbb{Z})$, which is shown to be equal to $b_1(M)-b_1(N)$ ---
here $\tilde{\Gamma}_{\delta}(p)$ denotes the \emph{pseudo-local fundamental
group}, which is defined for any $\delta\in (0,1)$ and $p\in M$ as
 \begin{align*}
 \tilde{\Gamma}_{\delta}(p)\ :=\ Image[\pi_1(B_g(p,\delta))\to \pi_1(M,p)].
 \end{align*}
Roughly speaking, this group considers those very short loops based at the
given point, but are allowed to deform within the entire manifold. The
pseudo-local fundamental group is an intermediate concept that interpolates
between the $\delta$-small first homology classes $H_1^{\delta}(M;\mathbb{Z})$,
which is entirely global, and the purely local concept $\Gamma_{\delta}(p)$. In
Lemma~\ref{lem: nil_rank}, we will check that under the assumption of
Theorem~\ref{thm: main1}, each $\tilde{\Gamma}_{\delta}(p)$ is almost nilpotent
with $\rank\ \tilde{\Gamma}_{\delta}(p)\le \dim M-\dim N$, as long as $\delta$
is sufficiently small. This will lead to the first claim in Theorem~\ref{thm:
main1}.

We now present our first major technical input, which is an effective control
of the geodesic spreading. To set up the context, for $(M,g)\in
\mathcal{M}_{Rc}(m)$ and $\Sigma\subset M$ a closed embedded submanifold, we
let $r_{\Sigma}:M\to \mathbb{R}$ denote the distance function to $\Sigma$. This
is a Lipschitz function and is almost everywhere smooth (see \S 4.1). It
defines a smooth vector field $\nabla r_{\Sigma}$ almost everywhere on $M$.
Notice that any minimal geodesic realizing the distance between a point and
$\Sigma$ is an integral curve of $\nabla r_{\Sigma}$ with initial value in
$T^{\perp}\Sigma$, the normal bundle of $\Sigma\subset M$. We now state
\begin{theorem}\label{thm: main3}
For any positive numbers $D\ge 1$, $\beta<10^{-2}$ and $m\in \mathbb{N}$, there
are constants  $\bar{r}\in (0,1)$ and $\bar{C}>1$ solely determined by
$m$, $D$, and $\beta$, to the following effect: let $(M,g)\in
\mathcal{M}_{Rc}(m)$ and let $\Sigma\subset M$ be a closed embedded submanifold,
and let $\sigma_0,\sigma_1:[0,l]\to M$ ($\frac{1}{4}\le l\le D$) be two minimal
geodesics of unit speed that are also integral curves of $\nabla r_{\Sigma}$
with $\sigma_0(0),\sigma_1(0)\in \Sigma$, if $d_g(\sigma_0(\beta
l),\sigma_1(\beta l))\le\bar{r}$, then we have
\begin{align}
\forall t\in [\beta l,(1-\beta)l],\quad
d_g\left(\sigma_0(t), \sigma_1(t)\right)\ \le\ \bar{C} d_g\left(\sigma_0(\beta
l),\sigma_1(\beta l)\right).
\end{align} 
\end{theorem}

In the application, if $\gamma: [0,1]\to M$ is a closed geodesic generating a
torsion-free class in $H_1^{\delta}(M;\mathbb{Z})$ with $|\gamma|\le 10\delta$,
then it lifts to the universal covering $\widetilde{M}$ (equipped with the
covering metric) and becomes a complete geodesic $\tilde{\gamma}:\mathbb{R}\to
\widetilde{M}$. Regarding $\gamma$ as an isometric action on $\widetilde{M}$, we
understand that bounding the size of $\gamma$ slided along a minimal geodesic
$\sigma$ realizing $d_g(p,\gamma([0,1]))$ amounts to estimating the distance
between the two lifted minimal geodesics $\sigma_0=\tilde{\sigma}$ and
$\sigma_1=\gamma.\tilde{\sigma}$ in $\widetilde{M}$ --- here we notice that
$\tilde{\gamma}(\mathbb{R})\subset \widetilde{M}$ is a closed embedded
smooth submanifold and that both $\sigma_0$ and $\sigma_1$ are integral curves of
$\nabla r_{\tilde{\gamma}(\mathbb{R})}$ --- Theorem~\ref{thm: main3}
applies to the pair $(\widetilde{M}, \tilde{\gamma}(\mathbb{R}))$; see
Figure~\ref{fig:threemoduli}.

\begin{figure}
 \begin{center}
 \psfrag{A}[c][c]{$\widetilde{M}$}
 \psfrag{B}[c][c]{$M$}
 \psfrag{C}[c][c]{$\pi$}
 \psfrag{E1}[c][c]{$\textcolor{red}{\gamma}$}
 \psfrag{E2}[c][c]{$\textcolor{red}{\tilde{\gamma}}$}
 \psfrag{p}[c][c]{$p$}
 \psfrag{p1}[c][c]{$\tilde{p}$}
 \psfrag{p2}[c][c]{$\gamma . \tilde{p}$}
 \psfrag{q}[c][c]{$q$}
 \psfrag{q1}[c][c]{$\tilde{q}$}
 \psfrag{q2}[c][c]{$\gamma . \tilde{q}$}
  \psfrag{s}[c][c]{$\textcolor{blue}{\sigma}$}
 \psfrag{s1}[c][c]{$\textcolor{blue}{\tilde{\sigma}}$}
 \psfrag{s2}[c][c]{$\textcolor{blue}{\gamma . \tilde{\sigma}}$}
 \includegraphics[width=0.5 \columnwidth]{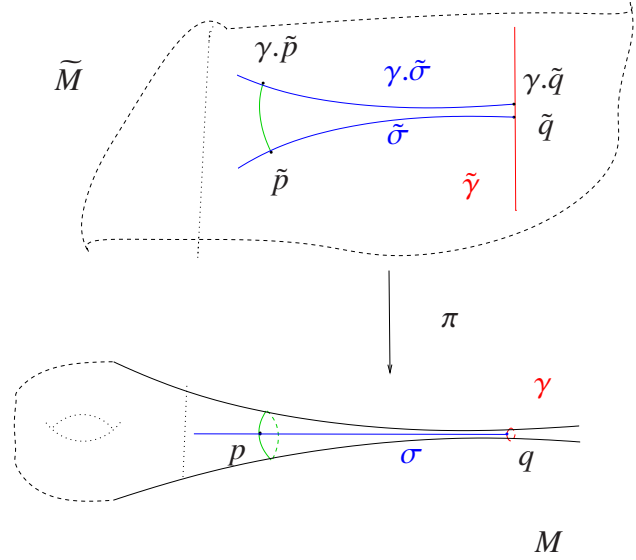}
 \caption{Pseudo-local actions from small first homology classes}
 \label{fig:threemoduli}
 \end{center}
 \end{figure}

We remark that the proof of Theorem~\ref{thm: main3} is inspired by Colding and
Naber's original work \cite{ColdingNaber}, where the H\"older continuity (in
the Gromov-Hausdorff sense) of geodesic balls centered along the middle of a
minimal geodesic is proven. Colding and Naber \cite{ColdingNaber} developed
ingenious and powerful arguments that enable us to pass the metric properties
along the middle of a minimal geodesic beyond the local scale, and we expect
applications in many other settings. For instance, in \cite{HLW18} their
arguments are adapted to show that any pointed Gromov-Hausdorff limit of a
sequence of Ricci shrinkers with a uniform $\boldsymbol{\mu}$-entropy lower
bound is a conifold Ricci shrinker; see also \cite{LLW18}.

The rigidity case in Theorem~\ref{thm: main1}, i.e. when $b_1(M)-b_1(N)=\dim
M-\dim N$, is then a relatively straightforward consequence of our second major
technical tool, a Ricci flow smoothing result:
\begin{theorem}\label{thm: main2}
Given positive constants $D\ge 1$, $m\in \mathbb{N}$, $\alpha<10^{-2}m^{-1}$
and $\iota<\min\{1,10^{-2}D\}$, there are positive constants
$\delta_{RF}(m,D,\alpha,\iota)<1$ and $\varepsilon_{RF}(m,D,\alpha,\iota)<1$ to
the following effect: if $(M,g)\in \mathcal{M}_{Rc}(m)$ and $(N,h)\in
\mathcal{M}_{Rm}(k,D,v)$ (with $k\le m$) satisfy
\begin{enumerate}
  \item $d_{GH}(M,N)\ <\ \delta$ for some $\delta\le \delta_{RF}$, and 
  \item $b_1(M)-b_1(N)\ =\ \dim M-\dim N$,
\end{enumerate}
then there is a Ricci flow solution $g(t)$ defined on $M$ with $g(0)=g$,
existing for a period no shorter than $\varepsilon_{RF}^2$, such that 
\begin{align}\label{eqn: main2}
\forall\ t\ \in\ (0,\varepsilon_{RF}^2],\quad
\sup_{M}\left|\Rm_{g(t)}\right|_{g(t)}\ \le\ \alpha
t^{-1}+\varepsilon_{RF}^{-2}.
\end{align}
\end{theorem}

This theorem grows out of a program initiated by the first named author in
\cite{Foxy1808} to investigate the behavior of Ricci flows with possibly
collapsing initial data. While in the setting of Theorem~\ref{thm: main2}, one
could always start a Ricci flow as $M$ is a closed manifold (see
\cite{Hamilton}), the emphasis here is the uniform lower bound on the existence
time, a crucial aspect when applying Ricci flows as means of smoothing. In
dimensions at least three, all known results on the short-time existence of
Ricci flows with initial Ricci curvature lower bound (see \cite{Simon12,
Hochard, He16, ST17, SHuang19}) rely on the initial uniform
\emph{non-collapsing} assumption to bound from below the existence time. In
contrast, Theorem~\ref{thm: main2} (when $k<m$) enables one to start the Ricci flow from collapsing initial data for a definite period of time. A localized version of Theorem~\ref{thm: main2} with singular collapsing limit could be found in \cite{HW20b}.

In fact, by (\ref{eqn: main2}) it is not hard to check that the smoothing
metric $g(\varepsilon^2_{RF})$, obtained from Theorem~\ref{thm: main2}, defines
a distance function that is equivalent to the original distance specified by
$g=g(0)$, and thus $(M,g(\varepsilon^2_{RF}))$ is also sufficiently
Gromov-Hausdorff close to the (lower dimensional) manifold $(N,h)$ at a fixed
scale. By the fibration theorems in \cite{CFG92, HKRX18}, we know that $M$ is
an infranil fiber bundle over the smooth manifold $N$. Relatively simple
arguments involving the Hurewicz theorem and the first Betti number then show
that the fibers must be tori.

After discussing the background and pointing out the technical difficulties in
proving Theorem~\ref{thm: main1} in \S 2, we will quantify the first Betti number difference by short loop homology (cf. Proposition \ref{prop: rank})
in \S 3. We will then prove
Theorem~\ref{thm: main3} in \S 4, and consequently Claim (1) of
Theorem~\ref{thm: main1} in \S 5. The proof of Claim (2) in Theorem~\ref{thm:
main1} will follow once Theorem~\ref{thm: main2} is established in \S 6, and
some further remarks will be left in the final section.

\section{Background and preliminary discussion}
In this section we explain the rationale and the technical difficulties in the 
proof of Theorem~\ref{thm: main1}. We begin our discussion with a much simpler
case.

\subsection{A precursor for non-negative Ricci curvature}
For closed manifolds with non-negative Ricci curvature, the Bochner technique tells that $b_1(M)\le \dim M$, and the following theorem, which is due to Yau \cite[Theorem 3]{Yau72}, reveals the structural information encoded in $b_1(M)$:
\begin{theorem}(Yau, 1972)\label{thm: rigidity_Rc_ge0}
If $(M,g)$ is a close oriented Riemannian manifold with non-negative Ricci curvature and
dimension at least three, then there is a closed Riemannian manifold $(F,h)$ with non-negative Ricci curvature, such that $M$ is isometric to an $F$-bundle over $\mathbb{T}^{b_1(M)}$.
\end{theorem}

Just as the approach of Bochner's original theorem (see \cite{Bochner, BY}), Yau's proof relied on the existence of harmonic vector fields. The assumed non-negative Ricci curvature plays a key role in the proof: firstly, it implies that the harmonic vector fields are parallel, therefore integrating to \emph{lines} in the universal covering space; and secondly, it allows the application of the de Rham (or the Cheeger-Gromoll) splitting theorem, to isometricaly produce an $\mathbb{R}^{b_1(M)}$-factor of the universal covering space. Moving from the case of non-negative Ricci curvature to the more general setting of Ricci curvature bounded (negatively) from below usually involves non-trivial localization and quantification, as examplified by the Colding-Gromov gap theorem \cite{Colding97} and the Cheeger-Colding almost splitting theorem \cite{ChCo0}. In order to facilitate such localization and quantification, we will realize the torsion-free first homology classes by closed geodesics.

 By the Hurewicz theorem, for each torsion-free generator in 
 $H_1(M;\mathbb{Z})$, we could find a continuous loop $\gamma':[0,1]\to M$ such that the
 homotopy class $[\gamma']\in \pi_1(M,\gamma'(0))$ is also torsion-free.
 Minimizing the length functional within the \emph{free homotopy class} of
 $\gamma'$, we could find a loop $\gamma:[0,1]\to M$ of minimal length in the class. Clearly $[\![\gamma]\!]=[\![\gamma']\!]\in
 H_1(M;\mathbb{Z})$, and  $\gamma$ is in fact a \emph{closed geodesic} (\cite[\S 12.2]{doCarmo}), i.e.
 $\gamma$ is a smooth geodesic and $(\gamma(0),\dot{\gamma}(0))
 =(\gamma(1),\dot{\gamma}(1))\in TM$.  
 
 Let $\pi:\widetilde{M}\to M$ denote the universal covering map, and 
equip $\widetilde{M}$ with the pull-back metric $\pi^{\ast}g$. 
 Now lifting $\gamma$ to $\widetilde{M}$, since $\gamma$ is a closed geodesic,
 we know that the lifted curve extends over both ends as a \emph{smooth
 geodesic}. Moreover, since $[\![\gamma]\!]$ is torsion-free, it is of infinite
 order --- we could therefore extend the lifted curve infinitely towards both
 directions and obtain a smooth geodesic $\tilde{\gamma}: \mathbb{R} \to
 \widetilde{M}$. In particular, $\tilde{\gamma}(\mathbb{R})\subset \widetilde{M}$ is a closed embedded submanifold.

In general, $\tilde{\gamma}$ is not a line; even if it were a line its existence in a complete manifold with a negative Ricci curvature lower bound does not guarantee the desired isometric splitting. We therefore needs to rely on Theorem~\ref{thm: main3} to quantitatively and uniformly control the size of the action induced by a small isometric action along the curve $\tilde{\gamma}$. On the other hand, the existence of such small isometric actions is a consequence of the ``collapsing'' assumption in Theorem~\ref{thm: main1} --- the Gromov-Hausdorff closeness of $M$ to (the lower dimensional) $N$ enables us to find the short closed geodesics that generate small isometric actions along their lifts in the universal covering space $\widetilde{M}$.

\subsection{Outlining the proof of Theorem~\ref{thm: main1}}
As previously mentioned, we will need tol ``localize'' and ``quantify'' the information encoded in the first Betti number.

Given $(M,g)\in \mathcal{M}_{Rc}(m)$, we recall that the pseudo-local fundamental group $\tilde{\Gamma}_{\delta}(p)$ is defined for any $\delta\in (0,1)$ and $p\in M$ as
\begin{align*}
\tilde{\Gamma}_{\delta}(p)\ =\ Image[\pi_1(B_g(p,\delta),p)\to \pi_1(M,p)].
\end{align*}
Notice that $\tilde{\Gamma}_{\delta}(p)$ could be generated by geodesic loops
$\gamma:(\mathbb{S}^1,1)\to (B_g(p,\delta),p)$ with length not exceeding
$2\delta$. On the other hand, considering the induced action of $\gamma$ on
$(\widetilde{M},\pi^{\ast}g)$ --- here $\pi:\widetilde{M}\to M$ is the
universal covering and $\pi^{\ast}g$ is the covering metric --- we clearly see
that $d_{\pi^{\ast}g}(\gamma.\tilde{p},\tilde{p})\le 2\delta$, with
$\tilde{p}\in \pi^{-1}(p)$ denoting a lift of $p$ to the universal covering
space $\widetilde{M}$. From this point of view, $\tilde{\Gamma}_{\delta}(p)$
could be characterized as a subgroup of $Isom(\widetilde{M},\pi^{\ast}g)$, by
\begin{align*}
\tilde{\Gamma}_{\delta}(p)\ \cong\ \widetilde{G}_{\delta}(p)\ :=\ \left\langle
\gamma\in \pi_1(M,p):\ d_{\pi^{\ast}g}(\gamma.\tilde{p},\tilde{p})\le
2\delta \right\rangle,
\end{align*}
for any given lift $\tilde{p}\in \pi^{-1}(p)$ of $p\in M$. 

If for some $(N,h)\in \mathcal{M}_{Rm}(k,D,v)$ with $k\le m$ and
$\iota\in (0,10^{-2}D)$, we have $d_{GH}(M,N)<10^{-1}\delta$, then we could
see that whenever $\delta>0$ is sufficiently small,
$\tilde{\Gamma}_{\delta}(p)$ is almost nilpotent with nilpotency rank not
exceeding $m-k$, for any $p\in M$.

\begin{lemma}\label{lem: nil_rank}
In the setting above, there is a constant $\delta_{Nil}>0$ determined by
$\iota$ and $m$, such that if $d_{GH}(M,N)<\delta$ for some
$\delta\le 10^{-1}\delta_{Nil}$, then $\rank\
\tilde{\Gamma}_{\delta_{Nil}}(p)\le \dim M-\dim N$ for any $p\in M$.
\end{lemma}

\begin{proof}
For any $(N,h)\in \mathcal{M}_{Rm}(k,D,v)$, by \cite{JK82} we understand
that there are uniform constants $C_{hr}(k,D,v)>0$ and 
$\iota_{hr}(k,D,v)\in (0,1)$ such that the $C^{1,\frac{1}{2}}$ harmonic radii
at all points in $N$ is bounded below by $\iota_{hr}$. We let
$\bar{\iota}_{hr}(m,D,v):=\min_{0\le k\le m}\iota_{hr}(k,D,v)$. On the
other hand, there is a constant $\varepsilon_{NZ}(m) :=\min_{0\le l\le m}
\varepsilon_0(m, \mathbb{B}^l(1))$, with each
$\varepsilon_0(m,\mathbb{B}^l(1))\in (0,1)$ denoting the uniform constant
obtained in \cite[Theorem 4.25]{NaberZhang}. Now by the $C^{1,\frac{1}{2}}$
harmonic radius lower bound, we have a uniform radius
$\iota_0\left(m,\max_{0\le k\le m}C_{hr}(k,D,v)\right)\in (0,\bar{\iota}_{hr})$
such that
\begin{align}
\forall \bar{p}\in N,\quad
d_{GH}\left(B_h(\bar{p},\iota_0),\mathbb{B}^k(\iota_0) \right)\ <\
10^{-1}\varepsilon_{NZ}(m)\iota_0.
\end{align}
We now set 
$\delta_{Nil}(m,D,v):=2^{-1} \varepsilon_{NZ}(m)\iota_0$, and 
assume that $\delta\le 10^{-1} \delta_{Nil}$.

If $d_{GH}(M,N)<\delta$, let $\Phi:M\to N$ denote a
$\delta$-Gromov-Hausdorff approximation and we have
\begin{align*}
 d_{GH}\left(B_g(p,\iota_0),\mathbb{B}^k(\iota_0) \right)\
\le\ d_{GH}\left(M,N\right)
+d_{GH}\left(B_h(\Phi(p),\iota_0), \mathbb{B}^k(\iota_0) \right)
\le\ 5^{-1}\epsilon_{NZ}\iota_0
\end{align*}
for every $p \in M$.

Now performing the rescaling $g\mapsto 4\iota_0^{-2}g=:\bar{g}$ and $h\mapsto
4\iota_0^{-2}h=:\bar{h}$, the above estimate becomes 
\begin{align}
d_{GH}\left(B_{\bar{g}}(p,2),\mathbb{B}^k(2)\right)\ <\ \varepsilon_{NZ}.
 \end{align} 
 
On the other hand, we notice  that the universal covering $\pi:
\widetilde{M}\to M$ is a normal covering with deck transformation group
$\pi_1(M,p)$, and the same conditions hold for its restriction to the local
covering $\pi_p:\pi^{-1}(B_{\pi^{\ast}\bar{g}}(p,2))\to B_{\bar{g}}(p,2)$ ---
the rescaled metric $\bar{g}$ is pulled back to the universal covering space
$\widetilde{M}$. We then see that
\begin{align*}
\widetilde{G}_{\delta_{Nil}}(p)\ =\ &\left\langle \gamma\in \pi_1(M,p):\
d_{\pi^{\ast}g}(\gamma.\tilde{p},\tilde{p})\le 2\delta_{Nil}\right\rangle\\
=\ &\left\langle \gamma\in \pi_1(M,p):\ d_{\pi^{\ast}\bar{g}}(\gamma.\tilde{p},
\tilde{p})\le 2\varepsilon_{NZ}\right\rangle.
\end{align*}
Consequently, we appeal to \cite[Theorem 4.25]{NaberZhang} to see that
$\widetilde{G}_{\delta_{Nil}}(p)$ is almost nilpotent with nilpotency rank
bounded above by $m-k$. But as we have already
seen that $\tilde{\Gamma}_{\delta_{Nil}}(p)\cong
\widetilde{G}_{\delta_{Nil}}(p)$ for any $p\in M$, we know
$\tilde{\Gamma}_{\delta_{Nil}}(p)$ is almost nilpotent, with $\rank\
\tilde{\Gamma}_{\delta_{Nil}}(p)\le \dim M-\dim N$.
\end{proof}
For any $\delta<\delta_{Nil}$, once we have bounded $\rank\
\tilde{\Gamma}_{\delta_{Nil}}(p)$ by the dimensional difference, our goal would
be to show that $b_1(M)-b_1(N)\le \rank\ \tilde{\Gamma}_{\delta_{Nil}}(p)$ for
any $p\in M$.

To extract the homological information and get the desired control on the
pseudo-local fundamental group, we collect the first homology classes in $M$
generated by short loops in the group
\begin{align*}
H_1^{\delta}(M;\mathbb{Z})\ :=\ \left\langle [\![\gamma]\!]:\ |\gamma|\le
10\delta \right\rangle.
\end{align*}
Clearly, $H_1^{\delta}(M;\mathbb{Z})$ is an abelian subgroup of
$H_1(M;\mathbb{Z})$, and in Proposition~\ref{prop: rank} we will show, under the
assumption of Theorem~\ref{thm: main1}, that 
\begin{align}\label{eqn: rank_b1}
\rank\ H_1^{\delta}(M;\mathbb{Z})\ =\ b_1(M)-b_1(N).
\end{align}

So our discussion will now be to compare $\rank\ H_1^{\delta}(M;\mathbb{Z})$ and $\rank\ \tilde{\Gamma}_{\delta_{Nil}}(p)$ for any $p\in M$.
While the Hurewicz theorem tells that 
\begin{align*}
\forall p\in M,\quad H_1(M;\mathbb{Z})\ \cong\ \pi_1(M,p)\slash [\pi_1(M,p),
\pi_1(M,p)],
\end{align*}
the same reasoning may not directly lead to the realization of
$H_1^{\delta}(M;\mathbb{Z})$ as (a sub-group of) 
\begin{align*}
\tilde{\Gamma}_{\delta_{Nil}}(p)\slash
\left([\pi_1(M,p),\pi_1(M,p)]\cap \tilde{\Gamma}_{\delta_{Nil}}(p)\right)
\end{align*} 
for every $p\in M$.
This is because the definition of
$\tilde{\Gamma}_{\delta_{Nil}}(p)$ not just requires the generating loops in
consideration to be very short, but also to be based at the given point
$p\in M$. A $\delta$-small generator in $H_1^{\delta}(M;\mathbb{Z})$
may, however, be located anywhere in $M$, not necessarily passing through the
given point $p\in M$. In contrast, the Hurewicz theorem holds because in
$H_1(M;\mathbb{Z})$ the size of the generators are allowed to be arbitrarily
large --- though not exceeding $2\diam (M,g)$.

To remedy the situation, we would start from the $\delta$-small generators of
$H_1^{\delta}(M;\mathbb{Z})$, and estimate its size when slided to other points.
More specifically, denoting $\rank\ H_1^{\delta}(M;\mathbb{Z})=:l_M$, we could
find geodesic loops $\gamma_1',\ldots,\gamma_{l_M}'$ of length not exceeding
$10\delta$, such that $[\![\gamma_1']\!],\ldots,[\![\gamma_{l_M}']\!]$ generate
the torsion-free part of $H_1^{\delta}(M;\mathbb{Z})$, which is a rank $l_M$
free $\mathbb{Z}$-module. For each $i=1,\ldots,l_M$, we may then perturb
$\gamma'_i$ within its free homotopy class to some $\gamma_i$, achieving
the minimal possible length. Then each $\gamma_i$ becomes a closed geodesic
with $[\![\gamma_i]\!]=[\![\gamma_i']\!]$ and $|\gamma_i|\le 10\delta$.
Notice that each
$\gamma_i\in \tilde{\Gamma}_{5\delta}(\gamma_i(0))\cong
\widetilde{G}_{5\delta}(\gamma_i(0))$, and we will examine the effect of the
action $\gamma_i\in Isom(\widetilde{M},\pi^{\ast}g)$ on $\pi^{-1}(p)$, for any
$p\not\in \gamma_i([0,1])$.

Fix some $\gamma_i$ ($i=1,\ldots,l_M$), by straightforward volume comparison we
get an estimate of the form
$d_{\pi^{\ast}g}(\gamma_i.\tilde{p},\tilde{p}) \le
Cd_{g}(p,\gamma_i)^m\delta^{-m}$, 
for any $\tilde{p}\in \pi^{-1}(p)$ with $p\not\in \gamma_i([0,1])$; compare the
constants in \cite[Lemma 5.2]{NaberZhang}. While this estimate may be useful
when $p$ and $\gamma_i([0,1])$ are within a distance comparable to $o(\delta)$, 
it clearly provides insufficient information to recognize $\gamma_i$
as an element of $\tilde{\Gamma}_{C\delta}(p)$, when $d_g(p,\gamma_i([0,1]))$
is comparable to $\diam (M,g)$. A more reasonable attempt would rely on Colding
and Naber's H\"older continuity theorem (\cite[Theorem 1.1]{ColdingNaber}),
where, say, for a minimal geodesic $\sigma: [0,1+\varepsilon]\to M$ such that
$p=\sigma(1)$, $\gamma_i(t_0)=\sigma(\varepsilon)$ and 
$\left|\sigma|_{[\varepsilon,1]}\right|=d_g(p,\gamma_i([0,1])$, we could lift
it to a minimal geodesic $\tilde{\sigma}$ in $\widetilde{M}$ with
$\tilde{\sigma}(\varepsilon)=\tilde{\gamma}_i(t_0)$ and see
\begin{align}\label{eqn: CoNa_GH}
d_{GH}\left(B_{\pi^{\ast}g}(\tilde{p},r),
B_{\pi^{\ast}g}(\tilde{\sigma}(\varepsilon),r) \right)\ \le\
C(m,D)\varepsilon^{-1}r,
\end{align}
for  $\varepsilon, r>0$ sufficiently small, with
$\tilde{p}=\tilde{\sigma}(1)\in \pi^{-1}(p)$. Let $\tilde{\Phi}:
B_{\pi^{\ast}g}(\tilde{\sigma}(\varepsilon),r)\to B_{\pi^{\ast}g}(\tilde{p},r)$
denote the Gromov-Hausdorff approximation obtained from the proof of Colding and
Naber's theorem (\cite[Theorem 1.1]{ColdingNaber}). While (\ref{eqn: CoNa_GH})
provides certain control on
$d_{\pi^{\ast}g}\left(\tilde{\Phi}(\gamma_i.\tilde{\sigma}(\varepsilon)),
\tilde{\Phi}(\tilde{\sigma}(\varepsilon)) \right)$ in terms of
$d_{\pi^{\ast}g}\left(\gamma_i.\tilde{\sigma}(\varepsilon),
\tilde{\sigma}(\varepsilon)\right)$, the problem is that $\tilde{\Phi}$ is
\emph{not} almost equivariant with respect to the action of $\gamma_i$ --- in
general we have \emph{no} comparison between
$d_{\pi^{\ast}g}\left(\tilde{\Phi}(\gamma_i.\tilde{\sigma}(\varepsilon)),
\tilde{\Phi}(\tilde{\sigma}(\varepsilon)) \right)$ and
$d_{\pi^{\ast}g}(\gamma_i.\tilde{\sigma}(1),\tilde{\sigma}(1))$. See
Figure~\ref{fig:geodesicflow} for an illustration.

\begin{figure}
 \begin{center}
 \psfrag{A}[c][c]{$\tilde{\sigma}(1+\epsilon)$}
 \psfrag{B}[c][c]{$\tilde{\sigma}(0)$}
 \psfrag{F1}[c][c]{$\textcolor{green}{B_g(\tilde{\sigma}(\epsilon),r)}$}
 \psfrag{F2}[c][c]{$\textcolor{green}{\tilde{\Phi}(B_g(\tilde{\sigma}(\epsilon),r))}$}
 \psfrag{p}[c][c]{$\tilde{p}=\tilde{\sigma}(1)$}
 \psfrag{p1}[c][c]{$\tilde{\gamma} . \tilde{p}$}
 \psfrag{q}[c][c]{$\tilde{\sigma}(\epsilon)$}
 \psfrag{q1}[c][c]{$\tilde{\gamma} . \tilde{q}$}
 \psfrag{g}[c][c]{$\textcolor{red}{\tilde{\gamma}}$}
 \psfrag{s}[c][c]{$\textcolor{blue}{\tilde{\sigma}}$}
 \psfrag{s1}[c][c]{$\textcolor{blue}{\tilde{\gamma} .  \tilde{\sigma}}$}
 \includegraphics[width=0.5 \columnwidth]{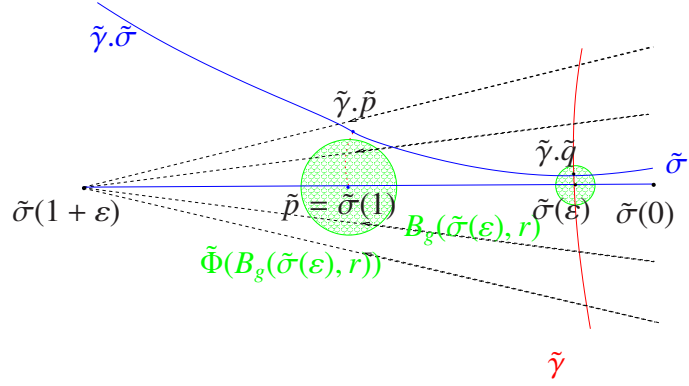}
 \caption{Lack of equivariance of the Gromov-Hausdorff approximation
 $\tilde{\Phi}$}
 \label{fig:geodesicflow}
 \end{center}
 \end{figure}
This explains the necessity of developing Theorem~\ref{thm: main3}, whose proof
in \S 5 essentially relies on Colding and Naber's original arguments and results
in \cite{ColdingNaber}. With this theorem at hand, we could slide $\gamma_i$ to
any $p\in M$ and obtain a geodesic loop of length smaller than $\delta_{Nil}$,
producing an element of $\tilde{\Gamma}_{\delta_{Nil}}(p)$ --- in fact, the
``slided loop'' at $p$ is defined as the projection under $\pi$ of any minimal
geodesic realizing $d_{\pi^{\ast}g}(\gamma.\tilde{\sigma}(1),
\tilde{\sigma}(1))$ in the setting above. The $\mathbb{Z}$-independence of the
homology classes $[\![\gamma_1]\!],\ldots, [\![\gamma_{l_M}]\!]$ then
guarantees the new loops at $p$ to define independent torsion-free elements of
$\tilde{\Gamma}_{\delta_{Nil}}(p)$, proving $\rank\
\tilde{\Gamma}_{\delta_{Nil}}(p)\ge l_M =\rank\ H_1^{\delta}(M;\mathbb{Z})$ ---
here we obviously need to assume that $\delta<<\delta_{Nil}(m)$ is sufficiently
small.

\section{First homology classes generated by short loops}
The goal of this section is to prove the equality $\rank\
H_1^{\delta}(M;\mathbb{Z})=b_1(M)-b_1(N)$ for manifolds $(M,g)$ and $(N,h)$
satisfying the assumption of Theorem~\ref{thm: main1} and $\delta\le
10^{-3}\bar{\iota}_{hr}$. In this section, we let $M\in \mathcal{M}_{Rc}(m)$
and $N\in \mathcal{M}_{Rm}(k,D,v)$, and assume that there is a
$10^{-1}\delta$-Gromov-Hausdorff approximation $\Phi: M\to N$ with $\delta\in
(0,10^{-3}\bar{\iota}_{hr})$. We put the following notations for any
$\varepsilon>0$: we say that two curves $c_0,c_1: [0,1]\to M$ are
$\varepsilon$-close to each other if $\sup_{t\in [0,1]} d(c_0(t),c_1(t))<2
\varepsilon$; also, for any curve $c:[0,1]\to M$ we let $c^{-1}$ denote inverse
curve $c^{-1}(t):=c(1-t):[0,1]\to M$. Moreover, for any curve $c:[0,1]\to M$,
we say that a curve $\bar{c}:[0,1]\to N$ is $\delta$-\emph{approximating} if
$\sup_{t\in [0,1]} d_h\left(\bar{c}(t), \Phi(c(t))\right)<\delta$ --- notice
that $\Phi$ is not necessarily continuous and so we cannot directly take
$\Phi(c)$ as a $\delta$-approximating loop, but for any curve in $M$, it is
easy to see that a $\delta$-approximating curve in $N$ always exists. 

To see this, we just let $0=t_0<t_1<\cdots<t_n=1$ be a fine enough partition of
$[0,1]$ so that $\left|c|_{[t_{i-1},t_i]}\right|\le 10^{-1}\delta$ for
each $i=1,\ldots,n$, pick $y_i\in B_h(\Phi(c(t_i)),10^{-1}\delta)$ and let
$\bar{\mu}_i$ be a minimal geodesic connecting $y_{i-1}=\bar{\mu}_i(0)$ to
$y_i=\bar{\mu}_i(1)$ (we could choose $\bar{\mu}_1(0)=\bar{\mu}_n(1)$ when $c$
is a loop); it is easy to see that
 \begin{align}\label{eqn: muba_i}
 \begin{split}
 \left|\bar{\mu}_i\right|\ \le\ &d_h\left(y_{i-1},\Phi(c(t_{i-1}))\right)
 +d_h\left(y_i,\Phi(c(t_i))\right)
 +d_h\left(\Phi(c(t_{i-1})),\Phi(c(t_i))\right)\\
 \le\ &\frac{3}{10}\delta+d_g(c(t_{i-1}),c(t_i))\ \le\
 \frac{3}{10}\delta+\left|c|_{[t_{i-1},t_i]}\right|\ \le\
 \frac{2}{5}\delta;
 \end{split}
 \end{align}
 forming the loop $\bar{c}:=\bar{\mu}_1\ast
 \bar{\mu}_2\ast\cdots\ast\bar{\mu}_n$, it is easily seen that for any $t\in
 [0,1]$, say, $t\in [t_{i-1},t_i]$,
 \begin{align}\label{eqn: muba_dist}
 \begin{split}
 d_h\left(\Phi(c(t)),\bar{c}(t)\right)\ =\
 &d_h\left(\Phi(c(t)),\bar{\mu}_i(t)\right)\ \le\
 d_h\left(\Phi(c(t)), y_{i-1}\right) +\left|\bar{\mu}_i\right|\\
 \le\ &\left|c|_{[t_{i-1},t_i]}\right|+\frac{3}{5}\delta\ \le\
 \frac{7}{10}\delta.
 \end{split}
 \end{align}
Therefore, $\bar{c}$ is the desired $\delta$-approximating curve of
$c$. Since the harmonic radii at all points of $N$ are bounded below by
$\bar{\iota}_{hr} \ge 10^3\delta$, we also notice that two $\delta$-approximating
loops for a given approximating loop determines the same homology class in
$H_1(N;\mathbb{Z})$.

We now discuss the finitely generated abelian group
$H_1^{\delta}(M;\mathbb{Z})$, consisting of homology classes generated by
geodesic loops of length not exceeding $10\delta$. As a basic property, we
notice that loops that are $\delta$-close to each other in $M$ define the same
first homology class modulo $H_1^{\delta}(M;\mathbb{Z})$:
\begin{lemma}\label{lem: loop_close}
Let $\gamma_0:[0,1]\to M$ be a loop formed by connecting geodesic segments of
lengths not exceeding $\delta$, and let $\gamma_1:[0,1]\to M$ be another loop
which is $2\delta$-close to $\gamma_0$, then
\begin{align*} 
[\![\gamma_0]\!]\ \equiv\ [\![\gamma_1]\!] \mod H_1^{\delta}(M;\mathbb{Z}).
\end{align*}
\end{lemma}
\begin{proof}
Suppose $\gamma_0=\mu_1\ast\mu_2\ast\cdots \ast\mu_l$ with $\mu_j$
being minimal geodesics in $M$, connecting $\gamma_0(s_{j-1})$ to
$\gamma_0(s_j)$ for $j=1,\ldots,l$, and $\left|\mu_j\right|
=\left|\gamma_0|_{[s_{j-1},s_j]}\right|\le \delta$; clearly
$\gamma_0(s_0)=\gamma_0(s_l)$. We also subdivide each $[s_{j-1},s_j]$
sufficiently fine by inserting $t_i$ so that
$\left|\gamma_1|_{[t_i,t_{i+1}]}\right|\le \delta$. We denote $t_{i_j}=s_j$, and
set
\begin{align*}
I_j\ :=\ \left\{0\le i\le n:\ t_i\in [s_{j-1},s_j)\right\}\ \text{for}\
j=1,\ldots,l.
\end{align*}
So our notation becomes
\begin{align*}
0=t_{i_0}=s_0<t_1<\cdots<t_{i_{j-1}}=s_{j-1}<t_{i_{j-1}+1}
<\cdots<t_{i_j-1}<t_{i_j}=s_j<\cdots<t_{i_l}=t_n=s_{l}=1,
\end{align*}
showing $I_j=\left\{i_{j-1},i_{j-1}+1,\ldots,i_j-1 \right\}$ in the middle.

For each $j=1,\ldots,l$ and $i\in I_j\cup\{i_{j-1}-1\}$ (with $i_0-1=i_l-1$),
connect $\gamma_0(s_{j-1})=:\mu_{j-1,i}(0)$ to $\gamma_1(t_i)=:\mu_{j-1,i}(1)$
by a minimal geodesic $\mu_{j-1,i}$, and we have
\begin{align*}
\left|\mu_{j-1,i}\right|\ \le\ d_g(\gamma_0(s_{j-1}),\gamma_0(t_i))
+d_g(\gamma_0(t_i),\gamma_1(t_i))\ 
\le\ \max\left\{\left|\mu_{j-1}\right|,\left|\mu_j\right|\right\}+2\delta\ \le\
3\delta.
\end{align*}
Notice that for $j=1,\ldots,l$, each $i_j-1$ is ``double-booked'' ---
 $\mu_{j-1,i_j-1}(1) =\gamma_1(t_{i_j-1}) =\mu_{j,i_j-1}(1)$. 
 
 With the convention $i_{-1}=0$, we then define a family of singular $1$-cycles
 as 
 \begin{align*}
\sigma_j\ &:=\ \mu_{j}\ast \mu_{j,i_j-1}\ast \mu_{j-1,i_j-1}^{-1} \quad
\text{for each}\ j=1,\ldots,l;\\
 \sigma_{j,i}\ &:=\ \mu_{j-1,i}\ast
\gamma_1|_{[t_{i-1},t_i]}^{-1}\ast \mu_{j-1,i-1}^{-1}\quad \text{for each}\
i\in I_j.
\end{align*}
 Clearly we have $|\sigma_j|\le 7\delta$ and $|\sigma_{i,j}|\le 7\delta$ for all
 possible $i$ and $j$. Moreover, from the construction it
 is clear that
 \begin{align*}
 \gamma_0-\gamma_1\ =\ \sum_{j=1}^l\left(\sigma_j+\sum_{i\in
 I_j}\sigma_{j,i}\right),
 \end{align*}
implying the claim of the lemma, as the right-hand side defines an element in
$H_1^{\delta}(M;\mathbb{Z})$.
\end{proof}

This lemma enables us to replace any loop in $M$ with one that we could
construct bare handedly with an error in $H_1^{\delta}(M;\mathbb{Z})$.
Actually, our basic principle predicts that for any loop $\gamma$ in $M$, if
its $\delta$-approximating loop is trivial in $H_1(N;\mathbb{Z})$, then we must
have $[\![\gamma]\!]\in H_1^{\delta}(M;\mathbb{Z})$. Intuitively speaking, we
expect deformations of $\delta$-approximating loops in $N$ to produce
corresponding deformations of the original loops in $M$ modulo loops of lengths
comparable to the Gromov-Hausdorff distance between $M$ and $N$. We now explain
a relatively simple case:
\begin{lemma}\label{lem: homotopy}
If $\gamma_0$ and $\gamma_1$ are two geodesic loops in $M$, such that the
$\delta$-approximating loops $\bar{\gamma}_0$ and $\bar{\gamma}_1$ are
homotopic to each other, then $[\![\gamma_0]\!]\equiv [\![\gamma_1]\!]\mod 
H_1^{\delta}(M;\mathbb{Z})$.
\end{lemma}

\begin{proof}
Let $\bar{\gamma}_0$ and $\bar{\gamma}_1$ be $\delta$-approximating loops
of  $\gamma_0$ and $\gamma_1$, respectively. Let $H:[0,1]^2\to N$ be a homotopy
between $\bar{\gamma}_0$ and $\bar{\gamma}_1$, with $H(0,-)=\bar{\gamma}_0$,
$H(1,-)=\bar{\gamma}_1$ and $H(-,0)=H(-,1)$. By the compactness of $[0,1]^2$,
we may let $n$ be so large that $\diam_hH([\frac{i}{n},\frac{i+1}{n}]\times
[\frac{j}{n},\frac{j+1}{n}])<10^{-1}\delta$. Let us define the paths in $N$ by
$\bar{\mu}_{i,j}(t):=H(\frac{i}{n},\frac{j+t}{n})$ and
$\bar{\nu}_{i,j}(t):=H(\frac{i+t}{n},\frac{j}{n})$ for 
$i,j=0,1,\ldots,n-1$. Now we could find $p_{i,j}\in M$ such that
$\Phi(p_{i,j})\in B_h(H(i,j),10^{-1}\delta)$; clearly, we have 
\begin{align}
\label{eqn: p_ijkl}
\begin{split}
d_g\left(p_{i,j},p_{k,l}\right)\ \le\
&d_h\left(\Phi(p_{i,j}),\Phi(p_{k,l})\right)+10^{-1}\delta\\
\le\ &d_h\left(H(i,j),H(k,l) \right) +d_h\left(H(i,j), \Phi(p_{i,j})\right)
+d_h\left(H(k,l),\Phi(p_{k,l}) \right)+10^{-1}\delta\\
 <\ &\frac{2}{5}\delta
\end{split}
\end{align} 
as long as $\max\{|i-k|,|j-l|\}\le 1$.
Since $\gamma_0$ and $\gamma_1$ already exist in $M$, we could assume that
$\{p_{0,j}\}\subset \gamma_0([0,1])$ and $\{p_{n,j}\}\subset \gamma_1([0,1])$.

We could moreover find minimal geodesics $\mu_{i,j}$ connecting 
$\mu_{i,j}(0)=p_{i,j}$ to $\mu_{i,j}(1)=p_{i,j+1}$, as well as $\nu_{i,j}$
with $\nu_{i,j}(0)=p_{i,j}$ and $\nu_{i,j}(1)=p_{i+1,j}$. We now form the loops
$\sigma_{i,j}:=\mu_{i,j}\ast \nu_{i,j+1}\ast \mu_{i+1,j}^{-1}\ast
\nu_{i,j}^{-1}$ for all $i,j=0,1,\ldots,n-1$. By (\ref{eqn: p_ijkl}) it is clear
that
\begin{align}\label{eqn: sigma_length_homotopy}
\left|\sigma_{i,j}\right|\ \le\ \left|\mu_{i,j}\right| +\left|\nu_{i,j+1}\right|
+\left|\mu_{i+1,j}\right| +\left|\nu_{i,j}\right|\ \le\ \frac{8}{5}\delta.
\end{align}
Moreover, it is easily seen that the loop $\sigma_{i,j}$ have its image
contained within $B_g(p_{i,j},\frac{4}{5}\delta)$, for all $i,j=0,1,\ldots,n-1$.

We also notice that the boundary loop $\gamma_0':=
\mu_{0,0}\ast\mu_{0,1}\ast\cdots \ast\mu_{0,n-1}$ is $2\delta$-close to the
original loop $\gamma_0$, as we now check:
 For any $t\in [0,1]$, say $t\in [\frac{j}{n},\frac{j+1}{n}]$, then by
 (\ref{eqn: p_ijkl}) and the choice that $\Phi(p_{0,j})\in
 B_h(H(0,j),10^{-1}\delta)$ we have
\begin{align}\label{eqn: 2delta_close}
\begin{split}
d_g\left(\gamma_0(t), \gamma_0'(t)\right)\ =\ 
&d_g\left(\gamma_0(t),\mu_{0,j}(t)\right)\ \le\
d_g\left(\gamma_0(t),p_{0,j}\right) +\left|\mu_{0,j}\right|\\ 
\le\ &d_h\left(\Phi(\gamma_0(t)),\Phi(p_{0,j}) \right) + \frac{1}{2}\delta\\
\le\ &d_h\left(\Phi(\gamma_0(t)),\bar{\gamma}_0(t) \right)
+d_h\left(\bar{\gamma}_0(t),\Phi(p_{0,j}) \right)+\frac{1}{2}\delta\\
 \le\ &\delta+d_h\left(\bar{\gamma}_0(t),H(0,j) \right)+\frac{3}{5}\delta\ 
<\ 2\delta.
\end{split}
\end{align}
And the same reasoning shows that $\gamma_1':= \mu_{n,0}\ast\mu_{n,1} \ast\cdots
\ast\mu_{n,n-1}$ is $2\delta$-close to $\gamma_1$. Therefore, by
Lemma~\ref{lem: loop_close}, we have $[\![\gamma_0]\!]-[\![\gamma_0']\!],
[\![\gamma_1]\!]-[\![\gamma_1']\!]\in H_1^{\delta}(M;\mathbb{Z})$.

Back in $N$, we also let $\bar{\sigma}_{i,j}=\bar{\mu}_{i,j}\ast
\bar{\nu}_{i,j+1}\ast \bar{\mu}_{i+1,j}^{-1}\ast \bar{\nu}_{i,j}^{-1}$, and
clearly $\partial H([\frac{i}{n},\frac{i+1}{n}]\times
[\frac{j}{n},\frac{j+1}{n}])=\bar{\sigma}_{i,j}$ with appropriate ordering. We
then have the following relation among singular $1$-cycles in $N$:
\begin{align}\label{eqn: homotopy}
\bar{\gamma}_1-\bar{\gamma}_0\ =\
\sum_{i,j}\bar{\sigma}_{i,j}.
\end{align}
Notice that the right-hand side of this equality vanishes because each
$\bar{\sigma}_{i,j}$ is a $1$-boundary in $N$, provided by the homotopy $H$, as
mentioned above. However, the relation (\ref{eqn: homotopy}), combinatorial in
nature, still holds for $\gamma'_{0}$, $\gamma'_1$ and $\sigma_{i,j}$ according
to our construction above, and thus 
\begin{align*}
 [\![\gamma_1']\!]-[\![\gamma_0']\!]\ =
\sum_{i,j}[\![\sigma_{i,j}]\!]\ \in\ H_1^{\delta}(M;\mathbb{Z}),
\end{align*}
since each $\left|\sigma_{i,j}\right|<10\delta$ by (\ref{eqn:
sigma_length_homotopy}). Finally, we have
\begin{align*}
[\![\gamma_1]\!]-[\![\gamma_0]\!]\ =\ \left([\![\gamma_1]\!]-[\![\gamma_1']\!]
\right)-\left([\![\gamma_0]\!]-[\![\gamma_0']\!]
\right)+\sum_{i,j}[\![\sigma_{i,j}]\!]\ \in\ H_1^{\delta}(M;\mathbb{Z}),
\end{align*}
whence the claim of the lemma.
\end{proof}

An immediate consequence concerns the case when $\pi_1(N)=0$: 
\begin{corollary}\label{cor: simply_connected}
If $N$ is simply connected, then $H_1(M;\mathbb{Z})
=H_1^{\delta}(M;\mathbb{Z})$.
\end{corollary} 
We could also see that the first homology of $M$ at scale $\bar{\iota}_{hr}$ is
determined by $N$ up to $H_1^{\delta}(M;\mathbb{Z})$:

\begin{corollary}\label{cor: homotopy_nearby}
If $\gamma_0, \gamma_1: [0,1]\to M$ are two $\frac{1}{4}\bar{\iota}_{hr}$-close
loops, then $[\![\gamma_0]\!]\equiv [\![\gamma_1]\!]\mod H_1^{\delta}(M;\mathbb{Z})$.
\end{corollary}

\begin{proof}
This is because
$d_h(\Phi(\gamma_0(t)), \Phi(\gamma_1(t)))<\delta+\frac{\bar{\iota}_{hr}}{2}$,
the assumption $\delta \le 10^{-3}\bar{\iota}_{hr}$, and that the
$\bar{\iota}_{hr}$-balls in $N$ are contractible, as the harmonic radii of all
points are at least $\bar{\iota}_{hr}$ on $N$.
\end{proof}

For a loop $c$ defined on $[0,1]$, we let $c^{\ast k}$ denote the $k$-fold
concatenation of itself, as a loop defined on $[0,k]$. Clearly, for any loop
$c$, the homology class $k[\![c]\!]$ can be represented by the loop $c^{\ast
k}$. We now upgrade Lemma~\ref{lem: homotopy} by showing the same results for
homologous loops in $N$:
\begin{lemma}\label{lem: homology_0}
If $\gamma_1,\ldots,\gamma_l$ are geodesic loops in $M$ with
$\delta$-approximating loops $\bar{\gamma}_1,\ldots,\bar{\gamma}_l$ in $N$, such
that there is a vanishing $\mathbb{Z}$-linear combination
$k_1[\![\bar{\gamma}_1]\!] +\cdots+k_l[\![\bar{\gamma}_l]\!] = 0 \in
H_1(N;\mathbb{Z})$, then we have
\begin{align*}
k_1[\![\gamma_1]\!]+\cdots+k_l[\![\gamma_l]\!]\ \in\ H_1^{\delta}(M;\mathbb{Z}).
\end{align*}
\end{lemma}
\begin{proof}
Clearly we only need to consider the case $k^2_1+\cdots+k^2_l\not=0$. By the
assumption we could find singular $2$-simplicies 
$\bar{\omega}_{abc}: \Delta^2\to N$ such that as singular $1$-cycles,
\begin{align*}
k_1\bar{\gamma}_1+\cdots+k_l\bar{\gamma}_l\ =\
\sum_{a,b,c}\partial \bar{\omega}_{abc}.
\end{align*}
By covering $\bar{\omega}_{abc}(\Delta^2)$ by $20^{-1}\delta$-balls and the
compactness of $\Delta^2$, we may perform barycentric subdivision and
guarantee that each $\diam_h\bar{\omega}_{abc}(\Delta^2)< 10^{-1}\delta$. Each
$\partial \bar{\omega}_{abc}$ is a $1$-boundary in $N$, and we denote
$\partial \bar{\omega}_{abc}=\bar{\mu}_{ab}+\bar{\mu}_{bc} -\bar{\mu}_{ac}$,
with the singular $1$-simplicies $\bar{\mu}_{ab},\bar{\mu}_{bc},
\bar{\mu}_{ac}: [0,1]\to N$ connecting the corresponding vertices
$\bar{p}_a,\bar{p}_b,\bar{p}_c\in N$ oriented in accordance with the
subscripts. Clearly, $d_h(\bar{p}_a,\bar{p}_b)<10^{-1}\delta$, and notice that
the above equation becomes
\begin{align}\label{eqn: homology_0}
k_1\bar{\gamma}_1+\cdots+k_l\bar{\gamma}_l\ =\
\sum_{a,b,c}\bar{\mu}_{ab}+\bar{\mu}_{bc}-\bar{\mu}_{ac},
\end{align}
and each $[\![\bar{\mu}_{ab}+\bar{\mu}_{bc}-\bar{\mu}_{ac}]\!]=0$ in
$H_1(N;\mathbb{Z})$ thanks to the singular $2$-simplex $\bar{\omega}_{abc}$.
Moreover, for each loop $\bar{\gamma}_i^{\ast k_i}:[0,k_i]\to N$, we may assume
it is a concatenation of some $\bar{\mu}_{ab}$, i.e. there
is a sub-collection $\left\{\bar{\mu}_{a^i_{0}a^i_{1}},\ldots,
\bar{\mu}_{a^i_{n_i-1}a^i_{n_i}}\right\}$ of the singular $1$-simplicies
appeared on the right-hand side of (\ref{eqn: homology_0}), such that
$\bar{\gamma}_i^{\ast k_i}=\bar{\mu}_{a^i_{0}a^i_{1}}\ast \cdots \ast
\bar{\mu}_{a^i_{n_i-1}a^i_{n_i}}$ with
$\bar{p}_{a^i_{0}}=\bar{p}_{a^i_{n_i}}$.

 We could then work as before to find $\{p_a\}\subset M$ such
that $d_h(\Phi(p_a),\bar{p}_a)<10^{-1}\delta$, and minimal geodesics $\mu_{ab}$
with $\mu_{ab}(0)=p_a$ and $\mu_{ab}(1)=p_b$. By the same argument leading to
(\ref{eqn: p_ijkl}), we know that $\left|\mu_{ab}\right|\le \frac{2}{5}\delta$
for all indices $a,b$. Moreover, let $\gamma_i'$ be the loop in $M$ formed by
concatenating those $\mu_{ab}$'s such that $\bar{p}_a,\bar{p}_b$ are in
$\bar{\gamma}_i^{\ast k_i}$, then by the same way leading to the estimate
(\ref{eqn: 2delta_close}), we know that each loop $\gamma_i'$ is
$2\delta$-close to the loop $\gamma_i^{\ast k_i}$. Consequently, we have
$[\![\gamma_i']\!]\equiv k_i[\![\gamma_i]\!]\mod H_1^{\delta}(M;\mathbb{Z})$
for each $i=1,\ldots,l$, thanks to Lemma~\ref{lem: loop_close}.

We now consider the geodesic triangles $\sigma_{abc} :=\mu_{ab}\ast
\mu_{bc}\ast \mu_{ac}^{-1}$ as singular $1$-cycles in $M$. Clearly each
$\left|\sigma_{abc}\right|\le 2\delta$, and has its image contained in in
$B_g(p_a,\delta)$. Moreover, the above combinatorial relation (\ref{eqn:
homology_0}) implies that
\begin{align*}
[\![\gamma'_1]\!]+\cdots+[\![\gamma'_l]\!]\ =\
\sum_{a,b,c}[\![\sigma_{abc}]\!]\ \in H_1^{\delta}(M;\mathbb{Z}),
\end{align*}
whence the claim of the lemma, as $k_i[\![\gamma_i]\!]-[\![\gamma_i']\!]\in
H_1^{\delta}(M;\mathbb{Z})$ for each $i=1,\ldots,l$.
\end{proof}
We also have a certain inverse to this lemma:
\begin{lemma}\label{lem: homology}
If $\gamma_1,\ldots,\gamma_l:[0,1]\to M$ are geodesic loops, such that there is
a $\mathbb{Z}$-linear relation
\begin{align*}
k_1[\![\gamma_1]\!]+\cdots+k_l[\![\gamma_l]\!]\ \in\ H_1^{\delta}(M;\mathbb{Z}),
\end{align*}
then their $\delta$-approximating loops $\bar{\gamma}_1,\ldots,\bar{\gamma}_l$
in $N$, as constructed at the very beginning of the subsection, satisfy
\begin{align*}
k_1[\![\bar{\gamma}_1]\!]+\cdots+k_l[\![\bar{\gamma}_l]\!]\ =\ 0\ \in
H_1(N;\mathbb{Z}).
\end{align*}
\end{lemma}
\begin{proof}
We could find a sufficiently large $n\in \mathbb{N}$ such that
$\left|\gamma_i|_{[\frac{j-1}{n},\frac{j}{n}]}\right|<10^{-1}\delta$ for each
$i=1,\ldots,l$ and $j=1,\ldots,n$. For each $i=1,\ldots,l$, we also let
$p_{i,j}:=\gamma_i^{\ast k_i}(\frac{j}{k_in})$, for $j=0,1,\ldots,k_in$. Notice
that $d_{g}(p_{i,j},p_{i,j-1})<10^{-1}\delta$ and
$p_{i,0}=p_{i,n}=p_{i,2n}=\cdots=p_{i,k_in}$. By the assumption, we know that
there are singular $1$-cycles $\sigma_{abc}$ such that
\begin{align}\label{eqn: homology_1}
\gamma_1^{\ast k_1}+\cdots +\gamma_l^{\ast k_l}\ =\ \sum_{a,b,c}\sigma_{abc}.
\end{align}
In particular, we have singular $1$-simplicies $\left\{\mu_{ab}:[0,1]\to
M\right\}$ such that $\sigma_{abc}=\mu_{ab}+\mu_{bc}-\mu_{ac}$, and let
$p_a:=\mu_{ab}(0)$ and $p_b:=\mu_{ab}(1)$ for all indicies $a$ and $b$. Without
loss of generality, we may assume that for each $i=1,\ldots,l$ and
$j=1,\ldots,k_in$, there is some $\mu_{a^i_{j-1}a^i_{j}}=\gamma_i^{\ast
k_i}|_{[\frac{j-1}{k_in}, \frac{j}{k_in}]}$. Obviously, $\mu_{a^i_{j-1}a^i_j}$
connects from $p_{a^i_{j-1}}=p_{i,j-1}$ to $p_{a^i_j}=p_{i,j}$ and
$\gamma_i=\mu_{a^i_0a^i_1}\ast\cdots \ast \mu_{a^i_{n-1}a^i_n}$.
 
We now pick $\{\bar{p}_a\}\subset N$ with $d_h(\Phi(p_a),\bar{p}_a)
<10^{-1}\delta$, and let $\bar{\mu}_{ab}$ be a minimal geodesic connecting
$\bar{p}_a=\bar{\mu}_{ab}(0)$ to $\bar{p}_b=\bar{\mu}_{ab}(1)$. Here we insist
that if $p_{a^i}=p_{b^i}\in M$, then we pick $\bar{p}_{a^i}=\bar{p}_{b^i}\in
N$. By the assumption that $\left|\mu_{ab}\right|\le 10^{-1}\delta$, we could
argue as in (\ref{eqn: muba_i}) to see that $|\bar{\mu}_{ab}|\le
\frac{2}{5}\delta$. 
For each $i=1,\ldots,l$ we define $\bar{\gamma}'_i:=\bar{\mu}_{a_0^ia_1^i}\ast
\bar{\mu}_{a_1^ia_2^i}\ast \cdots \ast \bar{\mu}_{a_{n_{i}-1}^ia_{n_i}^i}$ with
$a_{n_i}^i=a_0^i$, and (\ref{eqn: muba_dist}) shows that $\bar{\gamma}_i'$ is
a $\delta$-approximating loop of $\gamma_i^{\ast k_i}$. Notice that
$\bar{\gamma}_i:=\bar{\mu}_{a_0^ia_1^i}\ast \cdots \ast
\bar{\mu}_{a^i_{n-1}a^i_{n}}$ is also a $\delta$-approximating loop of
$\gamma_i$.

Now for any triple $(a,b,c)$, if $\sigma_{abc}$ appear on the right-hand side of
(\ref{eqn: homology_1}), we form the geodesic triangles
$\bar{\sigma}_{abc}:=\bar{\mu}_{ab}\ast \bar{\mu}_{bc}\ast\bar{\mu}_{ac}^{-1}$
in $N$, and regard each such $\bar{\sigma}_{abc}$ as a singular $1$-cycle in
$N$. According to (\ref{eqn: homology_1}), we then have the following equation
of singular $1$-cycles in $N$:
\begin{align}\label{eqn: bar_sigma_abc}
\bar{\gamma}'_1+\cdots +\bar{\gamma}'_l\ =\ \sum_{a,b,c}\bar{\sigma}_{abc}.
\end{align}
Moreover, it is clear that $|\bar{\sigma}_{abc}|\le 2\delta$, ensuring each
$\bar{\sigma}_{abc}\subset B_h(\bar{p}_a,\bar{\iota}_{hr})$ as $\delta\le
10^{-3}\bar{\iota}_{hr}$.
But since the harmonic radius at $\bar{p}_a$ is no less than $\bar{\iota}_{hr}$,
$B_h(\bar{p}_a,\bar{\iota}_{hr})$ is homeomorphic to an Euclidean ball, which is
contractible. By the Poincar\'e lemma, $\bar{\sigma}_{abc}$ is a $1$-boundary,
i.e. $\bar{\sigma}_{abc}=\partial \bar{\omega}_{abc}$ for some singular
$2$-simplex $\bar{\omega}_{abc}:\Delta^2\to B_h(\bar{p}_a,\bar{\iota}_{hr})$.
Therefore, the right-hand side of (\ref{eqn: bar_sigma_abc}) vanishes in
$H_1(N;\mathbb{Z})$. But the left-hand side represents
$\sum_ik_i[\![\bar{\gamma}_i]\!]$ and each
$\bar{\gamma}_i:=\bar{\gamma}_i'|_{[0,1]}$ is a $\delta$-approximation of
$[\![\gamma_i]\!]$.
\end{proof}
With the above understanding, we could now compare $\rank\ H_1(M;\mathbb{Z})$
and $b_1(M)-b_1(N)$. 
\begin{proposition}\label{prop: rank}
The shortest representatives of the torsion-free generators of
$H_1(M;\mathbb{Z})$ have lengths either $\le 10\delta$ or $\ge
10^{-1}\bar{\iota}_{hr}$.
Among these loops, we have a total number of $b_1(N)$ loops of length $\ge
10^{-1}\bar{\iota}_{hr}$, representing distinct torsion-free homology classes in
$H_1(M;\mathbb{Z})$, and making $\rank\ H_1(M;\mathbb{Z})\slash
H_1^{\delta}(M;\mathbb{Z}) =  b_1(N).$ Consequently, we have 
\begin{align}\label{eqn: rank}
\rank\ H_1^{\delta}(M;\mathbb{Z})\ =\ b_1(M)-b_1(N).
\end{align}
\end{proposition}

\begin{proof}
If $\gamma$ represents a generator of $H_1(M;\mathbb{Z})$ and $|\gamma|
<10^{-1}\bar{\iota}_{hr}$, then it has a $\delta$-approximating loop
$\bar{\gamma}$ in $B_h(\Phi(\gamma(0)),\bar{\iota}_{hr})$, since
$\delta+10^{-1}\bar{\iota}_{hr}<\frac{1}{2}\bar{\iota}_{hr}$.
Since the harmonic radius of $\Phi(\gamma(0))$ is at least
$\bar{\iota}_{hr}$, it means that $B_h(\Phi(\gamma(0)),\bar{\iota}_{hr})$ is
contractible, and thus $\bar{\gamma}\simeq_N \overline{\Phi(\gamma(0))}$, the
constant loop based at $\Phi(\gamma(0))$. By Lemma~\ref{lem: homotopy}, we must
have $[\![\gamma]\!]\in H_1^{\delta}(M;\mathbb{Z})$.

Since $H_1(M;\mathbb{Z})$ is a finitely generated abelian group, so are 
$H_1^{\delta}(M;\mathbb{Z})$ and their quotients. To compute the rank of the
quotient group $H_1(M;\mathbb{Z})\slash H_1^{\delta}(M;\mathbb{Z})$, we notice
that a coset $[\![\gamma]\!]+H_1^{\delta}(M;\mathbb{Z})$ defines a torsion
element in the quotient group \emph{if and only if} $k[\![\gamma]\!]\in
H_1^{\delta}(M;\mathbb{Z})$ for some $k\in \mathbb{Z}$. By the Hurewicz theorem,
we could find $\bar{\gamma}_1,\ldots,\bar{\gamma}_{b_1(N)}:[0,1]\to N$, 
representing the distinct generators of $H_1(N;\mathbb{Z})\slash Torsion$, which
is a free $\mathbb{Z}$-module of rank $b_1(N)$. Subdividing $[0,1]$ by
$0=t_0<t_1<\cdots <t_n=1$ sufficiently fine, we could ensure
$\left|\bar{\gamma}_i|_{[t_{j-1},t_{j}]}\right|\le 10^{-1}\delta$ for each
$j=1,\ldots, n$. We then find $\{p_{i,j}\}\subset M$ such that
$d_h(\Phi(p_{i,j}),\bar{\gamma}_i(t_j))<10^{-1}\delta$ for each $i$ and $j$, 
let $\mu_{i,j}$ be a minimal geodesic connecting $p_{i,j-1}=\mu_{i,j}(0)$ to
$p_{i,j}=\mu_{i,j}(1)$, with $\mu_{i,n}(1)=\mu_{i,1}(0)$, and form the loops
$\gamma_i:=\mu_{i,1}\ast \mu_{i,2}\ast \cdots \ast \mu_{i,n}$ for each
$i=1,\ldots,b_1(N)$. Just done in (\ref{eqn: p_ijkl}) we see that
$\left|\mu_{i,j}\right|\le \frac{2}{5}\delta$, and for $t\in [t_{j-1},t_j]$ we
have
\begin{align*}
d_h(\Phi(\gamma_i(t)),\bar{\gamma}_i(t) )\ \le\ \left|\mu_{i,j}\right|+
\left|\bar{\gamma}_i|_{[t_{j-1},t_j]}\right|+\frac{1}{5}\delta\ \le\
\frac{7}{10}\delta.
\end{align*}
This implies that each $\bar{\gamma}_i$ is $\delta$-approximating $\gamma_i$. We
notice that $|\gamma_i|\ge 10^{-1}\bar{\iota}_{hr}$, because otherwise we must have
$[\![\gamma_i]\!]\in H_1^{\delta}(M;\mathbb{Z})$, making
$[\![\bar{\gamma}_i]\!]=0\in H_1(N;\mathbb{Z})$ by Lemma~\ref{lem: homology},
contradicting the choice of $[\![\gamma_i]\!]$ as a generator.

In the same vein, we could show that the homology classes
$[\![\gamma_1]\!],\ldots,[\![\gamma_{b_1(N)}]\!]$ are indeed
$\mathbb{Z}$-independent modulo $H_1^{\delta}(M;\mathbb{Z})$: 
if there is some $\mathbb{Z}$-linear relation such that
\begin{align*}
k_1[\![\gamma_1]\!]+\cdots+k_{b_1(N)}[\![\gamma_{b_1(N)}]\!]\ \in\ 
H^{\delta}_1(M;\mathbb{Z}),
\end{align*} 
then by Lemma~\ref{lem: homology} we have $k_1[\![\bar{\gamma}_1]\!]+\cdots+
k_{b_1(N)}[\![\bar{\gamma}_{b_1(N)}]\!]=0 \in H_1(N;\mathbb{Z})$, and the
$\mathbb{Z}$-independence of the chosen classes in $H_1(N;\mathbb{Z})$ forces
$k_1=\cdots =k_{b_1(N)}=0$. Moreover, each $\gamma_i$ ($i=1,\ldots,b_1(N)$) is
torsion free again by Lemma~\ref{lem: homology}: if $[\![\gamma_i]\!]$ defines a
torsion element in $H_1(M;\mathbb{Z})\slash H_1^{\delta}(M;\mathbb{Z})$, then 
$n[\![\gamma_i]\!]\in H_1^{\delta}(M;\mathbb{Z})$ for some $n\in \mathbb{Z}$,
implying that  $n[\![\bar{\gamma}_i]\!]=0\in H_1(N;\mathbb{Z})$, contradicting
the choice $[\![\gamma_i]\!]$ as a generator of $H_1(N;\mathbb{Z})\slash
Torsion$. This proves $\rank\ H_1(M;\mathbb{Z})\slash
H_1^{\delta}(M;\mathbb{Z})\ge b_1(N)$.

Conversely, by the Hurewicz theorem, $H_1(M;\mathbb{Z})\slash Torsion$ also has
a collection of generators represented by geodesic loops. If $\gamma:[0,1]\to M$
is such a representing loop with $|\gamma|\ge 10^{-1}\bar{\iota}_{hr}$, we
consider a $\delta$-approximating loop $\bar{\gamma}:[0,1]\to N$, and we have  
\begin{align*}
[\![\bar{\gamma}]\!]\ =\ \sum_{i=1}^{b_1(N)}k_i[\![\bar{\gamma}_i]\!]
+\sum_j[\![\bar{\gamma}^{tor}_j]\!]\ \in\ H_1(N;\mathbb{Z}),
\end{align*}
where $[\![\bar{\gamma}_j^{tor}]\!]\in H_1(N;\mathbb{Z})$ are torsion elements,
represented by geodesic loops (by the Hurewicz theorem). Since we could
argue as before to obtain  loops $\gamma_j^{tor}:[0,1]\to M$ so that each 
$\bar{\gamma}_j^{tor}$ is $\delta$-approximating to $\gamma_j^{tor}$, by
Lemma~\ref{lem: homology_0} we know that
\begin{align*}
[\![\gamma]\!]\ \equiv\ \sum_{i=1}^{b_1(N)}k_i[\![\gamma_i]\!]+
\sum_j[\![\gamma^{tor}_j]\!] \mod H_1^{\delta}(M;\mathbb{Z}).
\end{align*}
But since certain finite multiple of $[\![\bar{\gamma}_j^{tor}]\!]$ vanishes in
$H_1(N;\mathbb{Z})$, by Lemma~\ref{lem: homology_0} we know that the coset
$\sum_j[\![\gamma^{tor}_j]\!] +H_1^{\delta}(M;\mathbb{Z})$ defines a torsion
element in $H_1(M;\mathbb{Z}) \slash H_1^{\delta}(M;\mathbb{Z})$. 
Consequently, the coset $[\![\gamma]\!] + H_1^{\delta}(M;\mathbb{Z})$ is
generated by those of $[\![\gamma_1]\!],\ldots, [\![\gamma_{b_1(N)}]\!]$.

The above discussion implies that $\rank\ H_1(M;\mathbb{Z})\slash
H_1^{\delta}(M;\mathbb{Z})\le b_1(N)$. Moreover, since $H_1(M;\mathbb{Z})$ and
$H_1^{\delta}(M;\mathbb{Z})$ are finitely generated $\mathbb{Z}$-modules, we
have
\begin{align*}
\rank\ H_1^{\delta}(M;\mathbb{Z})\ =\ \rank\ H_1(M;\mathbb{Z})-\rank\
H_1(M;\mathbb{Z})\slash H_1^{\delta}(M;\mathbb{Z})\ =\ b_1(M)-b_1(N),
\end{align*}
which is the desired equality for this section.
\end{proof}

\section{Effective distance control of initially nearby geodesics}
In this section we prove Theorem~\ref{thm: main3}, which is the major 
technical ingredient in proving the first claim Theorem~\ref{thm: main1}. Our
proof is inspired by the work of Colding and Naber \cite{ColdingNaber} --- in
fact, once we have set up the most basic estimates, i.e. the Laplacian
comparison for the distance function to a closed embedded submanifold, and the
local control of the spreading of minimal geodesics, the rest of Colding and
Naber's original argument works directly. While this may seem to be obvious to
experts, we will fill in the necessary details that bridge our considerations
to Colding and Naber's original results in \cite[\S 2 and \S 3]{ColdingNaber}.

Let $\Sigma$ be a smoothly embedded submanifold of a complete Riemannian
manifold $(M^m,g)$ such that $\Sigma=\overline{\Sigma}$, i.e. $\Sigma$ is
closed but not necessarily bounded. Let $r_{\Sigma}: M\to \mathbb{R}$ denote
the distance to $\Sigma$, i.e. $r_{\Sigma}(q):=\inf_{y\in \Sigma} d_g(q,y)$. By
the completeness of $(M,g)$ and the closedness of $\Sigma$, we know that for
any $q\in M\backslash \Sigma$, $r_{\Sigma}(q)>0$ is always realized by some
unit speed smooth geodesic $\sigma$ with $\sigma(0)=p\in \Sigma$,
$\sigma(r_{\Sigma}(q))=q$ and $|\sigma|=r_{\Sigma}(q)$.
This tells, by the triangle inequality, that $r_{\Sigma}$ is a $1$-Lipschitz
function on $M$. We will check that $r_{\Sigma}$ is in fact smooth almost
everywhere on $M$ in Lemma~\ref{lem: ae_smooth}. Consequently, the gradient
vector field $\nabla r_{\Sigma}$ is smoothly defined \emph{almost everywhere} on
$M$, and so is its gradient flow $\psi^{\Sigma}_s$ for each $s\ge 0$.

Assuming $\Rc_g\ge -(m-1)g$ and $\diam (M,g)\le D$, we will check that $\Delta
r_{\Sigma}\le C(m,D)r_{\Sigma}^{-1}$ in the distributional sense in
the first subsection, and then locally control the spreading of the flow lines
of $\nabla r_{\Sigma}$ in the second subsection. Once these are done, we could
directly appeal to the estimates in \cite[\S 2]{ColdingNaber} to obtain uniform
$C^1_{loc}$ and $H^2_{loc}$ control of the parabolic approximation of
$r_{\Sigma}$ in the third subsection, and finally, we follow the argument in
\cite[\S 3]{ColdingNaber} to effectively control the spreading of the flow
lines of $\nabla r_{\Sigma}$.

\subsection{Laplacian comparison for distance to submanifolds}
In this subsection, we will obtain an upper bound of $\Delta r_{\Sigma}$ in
(\ref{eqn: Laplace_comparison}), in the barrier sense \emph{a la} Calabi
\cite{Calabi58}. Though being a simple estimate, we surprisingly notice its
absence in the literature, and the purpose of this subsection is to fill this   
gap; also compare \cite[Lemma 2.1 and Lemma 2.2]{CF18} for the case of
non-negative Ricci curvature. 
The first order of business is to understand the regularity of $r_{\Sigma}$.
\begin{lemma}\label{lem: ae_smooth}
The function $r_{\Sigma}$ is almost everywhere smooth on $M$. 
\end{lemma} 
\begin{proof} 
We beginning with considering a point $q\in M\backslash \Sigma$ which is not a
focal point of $\Sigma$ (see \cite[\S 10.4]{doCarmo}), and which is connected
to $\Sigma$ by a \emph{unique} minimal geodesic $\sigma$ of unit speed, such
that $|\sigma|=r_{\Sigma}(q)=:l$, $\sigma(0)=p\in \Sigma$ and $\sigma(l)=q$. We
will show that $r_{\Sigma}$ is smooth in a neighborhood around $q$.

Since $q\in M$ is not a focal point of $\Sigma$, the initial data
$(\sigma(0),\dot{\sigma}(0))$ is a regular point of the normal exponential map
$\exp^{\perp}:T^{\perp}\Sigma\to M$, where $T^{\perp}\Sigma$ is the normal
bundle of $\Sigma$ within $TM$, and $\exp^{\perp}$ is nothing but the
restriction of the usual exponential map restricted to $T^{\perp}\Sigma$. Since
$\dim T^{\perp}\Sigma=\dim M$ and $q$ is not a singular point of
$\exp^{\perp}$, there is an open neighborhood $U_0\subset T^{\perp}\Sigma$ where
$\exp^{\perp}$ restricts to be a diffeomorphism onto its image
$W_0:=\exp^{\perp}(U_0)$. We now make the following

\textbf{Claim:} There is a smaller neighborhood $W\subset W_0$ of $q$ such that
for any $q'\in W$, $r_{\Sigma}(q')$ is uniquely realized by the geodesic
$t\mapsto \exp_{p'}t\vec{v}$, for some $(p',\vec{v})\in
U=(\exp^{\perp})^{-1}(W)\subset U_0$. 

\begin{proof}[Proof of the Claim]
Suppose otherwise, that there is a sequence $q_i\to q\in W_0$ with distinct
initial data $(\tau_i(0),\dot{\tau}_i(0)) \in T^{\perp}\Sigma$ and
$(\sigma_i(0),\dot{\sigma}_i(0))\in U_0$, such that
$|\tau_i|=d_g(q_i,\Sigma)\le |\sigma_i|$, $\tau_i(|\tau_i|)=q_i$, and
$\exp_{\sigma_i(0)} t_i\dot{\sigma}_i(0)=q_i$ with $t_i\ge d_g(q_i,\Sigma)$.
(Here we only work with unit tangent vectors.) Since $\exp^{\perp}|_{U_0}$ is
bijective, we have for all $i$ large enough $(\tau_i(0),\dot{\tau}_i(0))\not\in
U_0$. On the other hand, for all $i$ sufficiently large, since
\begin{align*}
d_g(q,\tau_i(0))\ \le\  d_g(q,q_i)+d_g(q_i,\tau_i(0))\ <\ 3d_g(q,\Sigma),
\end{align*}
by taking subsequence if necessary,  we know that 
$(\tau_i(0),\dot{\tau}_i(0))\to (p',\vec{v})\in
T^{\perp}\Sigma\backslash U_0$ for some $p'\in \Sigma$ and $\vec{v}\in T_{p'}M$
with unit length. However, denoting the limit geodesic $t\mapsto
\exp_{p'}t\vec{v}$ by $\tau$, we notice that it has length
\begin{align*}
|\tau|\ =\ \lim_{i\to \infty}|\tau_i|\ =\ \lim_{i\to \infty}d_g(q_i,\Sigma)\ =\
r_{\Sigma}(q),
\end{align*} 
and this contradicts our uniqueness assumption on $\sigma$, which has length 
$|\sigma|=r_{\Sigma}(q)$ and initial data $(\sigma(0),\dot{\sigma}(0))\in U_0$. 
\end{proof}

Now on $U$ denoting the unique geodesics $t\mapsto \exp_{p'}t(\vec{v})$ by
$\tau_{(p',\vec{v})}$ for all $(p',\vec{v})\in U$, we always have
$r_{\Sigma}(\exp_{p'}\vec{v})=|\tau_{(p',\vec{v})}|$, so that $r_{\Sigma}$
smoothly depends on $(p',\vec{v})\in U$. Moreover, denoting the inverse function
of $\exp^{\perp}|_{U}$ by $\log^{\perp}_U: W\to U\subset T^{\perp}\Sigma$, we
see that $r_{\Sigma} =|\tau_{\log^{\perp}_U}|$, which clearly depends on the
input smoothly. So $r_{\Sigma}$ is smooth on $W\subset M$.

Therefore, we understand that a possibly non-smooth point $q\in M$ of
$r_{\Sigma}$ must fall into one of the following two categories:
\begin{enumerate}
  \item $q$ is a focal point of $\Sigma$; or 
  \item there are multiple minimal geodesics that realizes $r_{\Sigma}(q)$.
\end{enumerate}
Notice that the focal points of $\Sigma$ are characterized by the critical
values of the normal exponential map (see \cite[\S 10.4, Proposition
4.4]{doCarmo}), which, by Sard's theorem, is a null set in $M$. On the other
hand, if (2) is the case, say, there are geodesics $\sigma_1$ and
$\sigma_2$ realizing $r_{\Sigma}(q)$, with $\sigma_1(0),\sigma_2(0)\in \Sigma$
and $\sigma_1(1)=\sigma_2(1)=q$, but 
$(\sigma_1(0),\dot{\sigma}_1(0)) \not= (\sigma_2(0),\dot{\sigma}_2(0))$, then we
must have $\dot{\sigma}_1(1) \not= \dot{\sigma}_2(1)$, and $r_{\Sigma}$
fails to be differentiable at $q$. However, since $r_{\Sigma}$ is a Lipschitz
function, its non-differentiable points form a measure $0$ subset of $M$.
Therefore, all non-smooth points of $r_{\Sigma}$ fall into the union of two
null subsets of $M$, which has to be of measure $0$.
\end{proof}

Before checking the desired Laplacian bound for $r_{\Sigma}$ at its smooth
points, we define the notation $F_m(r)$ as a function for $r>0$ with 
\begin{align*}
F_m(r)\ :=\
(m-1) r\coth r\ =\ (m-1)\frac{\sum_{n=0}^{\infty}\frac{r^{2n}}{(2n)!}}
{\sum_{n=0}^{\infty}\frac{r^{2n}}{(2n+1)!}}.
\end{align*}
Clearly $F_m(r)>m-1$ and $\lim_{r\searrow 0}F_m(r)=m-1$. We also put
$C_{F_m}(l):=\max_{r\in [0,l]}F_m(r)$. 
\begin{lemma}[Laplacian comparison at smooth points]
\label{lem: Laplace_comparison_smooth}
 Let $\Sigma\subset M$ be a closed and smoothly embedded submanifold. Suppose
 that the Riemannian manifold $(M,g)$ satisfies $\Rc_g\ge -(m-1)g$, and let
 $r_{\Sigma}$ denote the distance function to $\Sigma$. Then at the smooth
 points of $r_{\Sigma}$, we have 
\begin{align}\label{eqn: Laplace_comparison}
\Delta r_{\Sigma}\ \le\ \frac{F_m(r_{\Sigma})}{r_{\Sigma}}.
\end{align}
\end{lemma}
 The basic observation here is that the level set of $r_{\Sigma}$ is always
 less convex than the geodesic sphere touching it, as illustrated in
 Figure~\ref{fig:differentspheres}.
 
 \begin{figure}
 \begin{center}
 \psfrag{B1}[c][c]{$\textcolor{green}{\partial B_g(q, r)}$}
 \psfrag{B2}[c][c]{$r_{\Sigma}^{-1}(r)$}
 \psfrag{p}[c][c]{$p$}
 \psfrag{q}[c][c]{$q$}
 \psfrag{q1}[c][c]{$\tilde{\gamma} . \tilde{q}$}
 \psfrag{g}[c][c]{$\textcolor{red}{\Sigma}$}
 \psfrag{s}[c][c]{$\sigma$}
 \psfrag{s1}[c][c]{$\tilde{\gamma} .  \tilde{\sigma}$}
 \includegraphics[width=0.5 \columnwidth]{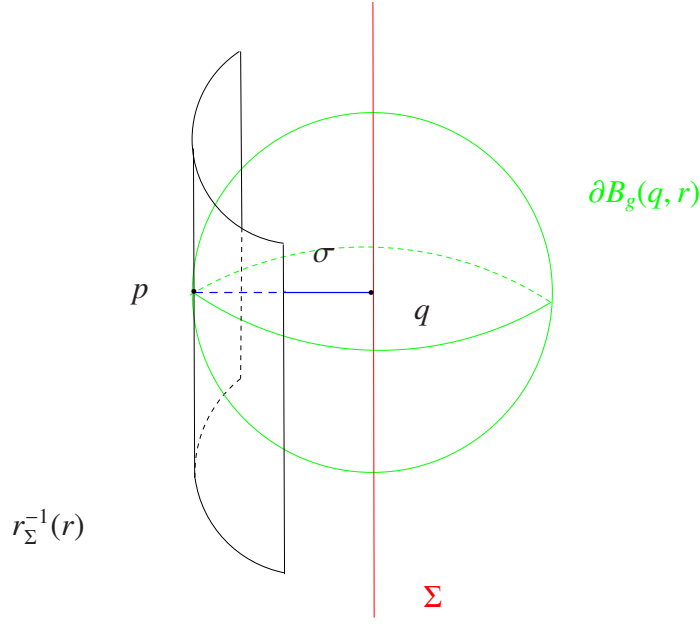}
 \caption{Hessian comparison at regular points of $r_{\Sigma}$}
 \label{fig:differentspheres}
 \end{center}
 \end{figure}
 
\begin{proof}
Suppose $r_{\Sigma}(q)>0$ and $r_{\Sigma}$ is a smooth function around $q\in
M$. We also assume that $r_{\Sigma}(q)$ is realized by a minimal geodesic
$\sigma: [0,r_{\Sigma}(q)]\to M$ of unit speed, with $\sigma(0)=p\in
\Sigma$ and $\sigma(1)=q$. Now let $d_p$ denote the distance function to
$p$, i.e. $d_p(x):=d_g(x,p)$. Then $d_p$ is smooth around $q\in M$ with
$d_p(q)=r_{\Sigma}(q)$ and $d_p(q)$ is also realized by the minimal geodesic
$\sigma$.

Now pick an orthonormal basis $E_1,\ldots,E_m$ of $T_qM$ such that
$E_1=\dot{\sigma}(1)$. Let $\gamma_i:(-\varepsilon,\varepsilon)\to M$ be
minimal geodesics with $\gamma_i(0)=q$ and $\dot{\gamma}_i(0)=E_i$
($i=2,\ldots,m$). Here we choose $\varepsilon>0$ sufficiently small (smallness
depending on $q\in M$) so that both functions $r$ and $d_p$ are smooth when
restricted to each $\gamma_i$. For each $i=2,\ldots,m$, let $f_i,g_i:
(-\varepsilon,\varepsilon)\to \mathbb{R}$ be defined as $f_i(t):=
r_{\Sigma}(\gamma_i(t))$ and $g_i(t):=d_p(\gamma_i(t))$, then we see that
\begin{align*}
f_i(0)\ =\ g_i(0)\ =\ r_{\Sigma}(q),\quad \text{and}\quad 
f'_i(0)\ =\ g_i'(0)\ =\ \langle E_i, \dot{\sigma}(1)\rangle\ =\ 0.
\end{align*}
Moreover, since for each $i=2,\ldots,m$ we have
\begin{align*}
\forall t\in (-\varepsilon, \varepsilon),\quad  
f_i(t)\ =\ \inf_{y\in \Sigma}d_g(\gamma_i(t),y)\ \le\ d_g(\gamma_i(t),p)\ =\
g_i(t),
\end{align*}
comparing the Taylor polynomials of $f_i$ and $g_i$ expanded around $t=0$,
we get the estimates
\begin{align}\label{eqn: second_derivative_comparison}
f_i''(0)\ \le\ g_i''(0)\quad \text{for}\ i=2,\ldots,m.
\end{align}
On the other hand, for each $i=2,\ldots,m$ simple computation gives
\begin{align*}
f_i''(0)\ =\ Hess_{r_{\Sigma}}(E_i,E_i)\quad \text{and}\quad g_i''(0)\ =\
Hess_{d_p}(E_i,E_i),
\end{align*}
 and thus (\ref{eqn: second_derivative_comparison}) leads to
\begin{align}\label{eqn: Laplace_at_q}
\Delta r_{\Sigma}(q)\ \le\ \Delta d_p(q).
\end{align}
Since the usual Laplace comparison for the distance function to a point gives 
\begin{align*}
\Delta d_p\ \le\ (m-1)\coth d_p
\end{align*}
whenever $d_p$ is smooth, and $d_p(q)=r_{\Sigma}(q)$, by (\ref{eqn:
Laplace_at_q}) we especially have at $q\in M$ that
\begin{align*}
\Delta r_{\Sigma}(q)\ \le\ \frac{F_m(r_{\Sigma}(q))}{r_{\Sigma}(q)}.
\end{align*}  
Since $q\in M$ is an arbitrary smooth point the function $r_{\Sigma}$, we have
derived (\ref{eqn: Laplace_comparison}) wherever $r_{\Sigma}$ is smooth.
\end{proof} 

 In a similar spirit, we could in fact show that (\ref{eqn: Laplace_comparison})
 holds everywhere on $M$ in the barrier sense:
\begin{proposition}[Global Laplacian comparison]\label{prop: Laplace_comparison}
As assumed in Lemma~\ref{lem: Laplace_comparison_smooth}, we have (\ref{eqn:
Laplace_comparison}) holding everywhere on $M$ in the barrier sense, i.e. for
any $q\in M$ and every $\varepsilon>0$ small enough, there is an open
neighborhood $U$ of $q$ and a function $h_{q,\varepsilon}\in C^2(U)$, such that 
\begin{enumerate}
  \item $r_{\Sigma}(q)= h_{q,\varepsilon}(q)$,
  \item $h_{q,\varepsilon}\ge r_{\Sigma}$ in $U$, and
  \item $\Delta h_{q,\varepsilon}\le
  \frac{F_m(r_{\Sigma}(q))}{r_{\Sigma}(q)}+\varepsilon$.
\end{enumerate}
Consequently, (\ref{eqn: Laplace_comparison}) holds everywhere on $M$ in the
distributional sense.
\end{proposition}
\begin{proof}
As in the proof of Lemma~\ref{lem: Laplace_comparison_smooth}, we let $\sigma$
denote a unit speed minimal geodesic such that $\sigma(0)=p\in \Sigma$,
$\sigma(r_{\Sigma}(q))=q$ and $|\sigma|=r_{\Sigma}(q)$, without assuming $q\in
M$ being a smooth point of $r_{\Sigma}$. Now for any positive 
\begin{align*}
\varepsilon\ \le\  \min\left\{ 10^{-1}\sqrt{r_{\Sigma}(q)},\ (m-1)^{-1}
\left(2(m-1)^{-1}C_{F_m}(r_{\Sigma}(q))^2 r_{\Sigma}(q)^{-2}-1\right)^{-1}
\right\}, 
\end{align*} 
we consider the function
\begin{align*}
h_{q,\varepsilon}(x)\ :=\ d_g\left(x,\sigma(\varepsilon^2)\right)+\varepsilon^2.
\end{align*}
 Then clearly $r_{\Sigma}(q)=h_{q,\varepsilon}(q)$. 
On the other hand, by the triangle inequality we have
\begin{align*}
\forall x\in M,\quad h_{q,\varepsilon}(x)\ \ge\ d_g(x,p)\ \ge\ r_{\Sigma}(x).
\end{align*}
Moreover, $h_{q,\varepsilon}$ is smooth in some open neighborhood around $q$,
as $q$ is not in the cut locus of $p\in M$, and thus for some $s\in
[r_{\Sigma}(q)-\varepsilon^2,r_{\Sigma}(q)]$ we have
\begin{align*}
\Delta h_{q,\varepsilon}(q)\ \le\ &(m-1)\coth
d_g\left(q,\sigma(\varepsilon^2)\right)\\
=\ &(m-1)\coth
 r_{\Sigma}(q)+(m-1)\varepsilon^2
 \left((m-1)^{-1}s^{-2}F_m(s)^2-1\right)\\
 \le\ &\frac{F_m(r_{\Sigma}(q))}{r_{\Sigma}(q)}+\varepsilon,
\end{align*}
by the constraint imposed on $\varepsilon>0$. Therefore, we have constructed an
upper barrier function $h_{q,\varepsilon}$ for $r_{\Sigma}$ satisfying all
requirements in Calabi's Laplacian comparison in the barrier sense (see
\cite{Calabi58}). It is well-known that this guarantees (\ref{eqn:
Laplace_comparison}) to hold in the distributional sense.
\end{proof}

Slightly modifying the proof of \cite[Lemma 3.2]{ColdingNaber}, which is
originated from \cite{Calabi67}, we obtain the following uniform Hessian $L^2$
estimate along the interior of a minimal geodesic, as a straightforward
consequence of the above Laplacian comparison for $r_{\Sigma}$:
\begin{lemma}\label{lem: Hess_L2_line}
Suppose $\Sigma$ is a closed embedded submanifold of a Riemannian
manifold $(M,g)$ with $\Rc_g\ge -(m-1)g$, and let $r_{\Sigma}$ denote the
distance function to $\Sigma$, as discussed above. If $q\in M\backslash \Sigma$
and $\sigma$ is a minimal geodesic of unit speed realizing the value
$l:=r_{\Sigma}(q)>0$, then 
\begin{align}\label{eqn: Hess_L2_segment}
 \int_{\beta l}^{(1-\beta)l}\left|Hess_{r_{\Sigma}}\right|^2(\sigma(t))\
 \text{d}t\ \le\ \frac{C^2_H(m,l)}{\beta l},
\end{align}
for any $\beta \in (0,10^{-2})$, with $C^2_H(m,l):=(m-1)l^2+C_{F_m}(l)$.
\end{lemma}

\begin{proof}
We put $d_q(x):= d_g(q,x)=d_g(\sigma(l),x)$, then $d_q$ is smooth away from
$q=\sigma(l)$, and around $p=\sigma(0)\in \Sigma$. Now the usual Laplace
comparison for the distance function to a point tells that 
\begin{align*}
\Delta d_q\ \le\ \frac{F_m(d_q)}{d_q}
\end{align*}
in a neighborhood around $\sigma((0,(1-\beta) l])$, where $d_q\ge \beta l$. 
Consequently, we have
\begin{align}
\sup_{0<t\le (1-\beta) l}\Delta d_q(\sigma(t))\ \le\
\frac{C_{F_m}(l)}{\beta l}.
\end{align}
By the previous Laplace comparison (\ref{eqn: Laplace_comparison}) for
$r_{\Sigma}$, we also have
\begin{align}
\sup_{\beta l\le t<l}\Delta r_{\Sigma}(\sigma(t))\ \le\
\frac{C_{F_m}(l)}{\beta l},
\end{align}
as $r_{\Sigma}\ge \beta l$ when restricted on $\sigma([\beta l,l))$.

On the other hand, since on $\sigma([\beta l,(1-\beta) l])$ both functions
$d_q$ and $r_{\Sigma}$ are smooth, and their sum $d_q+r_{\Sigma}$ achieves the
\emph{minimum} value $l$ by the triangle inequality, we must have
$\Delta(d_q+r_{\Sigma})\ge 0$ when restricted to $\sigma([\beta
l,(1-\beta)l])$. Therefore, we have for any  $t\in [\beta l,(1-\beta)l]$,
\begin{align}
-\frac{C_{F_m}(l)}{\beta l}\ \le\ -\Delta d_q (\sigma(t))\ \le\
\Delta r_{\Sigma}(\sigma(t))\ \le\ \frac{C_{F_m}(l)}{\beta l}.
\end{align}

We could now apply the Weitzenb\"ock formula to $r_{\Sigma}$ to see 
\begin{align*}
\partial_t\Delta r_{\Sigma}(\sigma(t))+|Hess_{r_{\Sigma}}|^2(\sigma(t))\ \le\
m-1,
\end{align*}
and integrating along $\sigma$ from $\beta l$ to $(1-\beta)l$ we have
\begin{align*}
\int_{\beta l}^{(1-\beta)l}\left|Hess_{r_{\Sigma}}\right|^2(\sigma(t))\
\text{d}t\ \le\ (m-1)(1-2\beta)l+\Delta r_{\Sigma} (\sigma(\beta l))-\Delta
r_{\Sigma}(\sigma((1-\beta)l)),
\end{align*}
which leads to the desired estimate (\ref{eqn: Hess_L2_segment}) if we put
$C^2_H(m,l)=(m-1)l^2+C_{F_m}(l)$.
\end{proof}

\subsection{Local control of the geodesic spreading}
Recall that we would like to control the spreading of the flow lines of $\nabla
r_{\Sigma}$. In this subsection, we do this locally around a smooth flow line of
$\nabla r_{\Sigma}$ that connects a smooth point back to $\Sigma$.

Now let $q\in M\backslash \Sigma$ with the minimal geodesic $\sigma$ realizing
$r_{\Sigma}(q)=:l$ (with $\sigma(0)=p\in \Sigma$), then for any $t\in (0,l)$,
$\sigma(t)$ is a smooth point of $r_{\Sigma}$. We fix some $\beta\in
(0,10^{-2})$, and cover $\sigma([\frac{\beta}{2}l,(1-\frac{\beta}{2})l])$ by
finitely many open sets $W_i\subset M$ as obtained by the Claim in the proof of
\ref{lem: ae_smooth}, and let $U_i\subset T^{\perp}\Sigma$ be the corresponding
open subsets of initial values. By the compactness of
$\sigma([\frac{\beta}{2}l, (1-\frac{\beta}{2})l])$, $\{W_i\}$ can be reduced
to a finite covering and we could let $U:=\cup U_i$ which is an open
neighborhood of $\left\{(p,t\dot{\sigma}(0)):\ t\in
[0,(1-\frac{\beta}{2})l)]\right\}$ in $T^{\perp}\Sigma$. We also let
$W:=\exp^{\perp}U$ which is an open subset of $M$. Notice that the geodesics
$\sigma_{(p',\vec{v})}: t\mapsto \exp_{p'}t\vec{v}$ uniquely realizes
$r_{\Sigma}(\sigma_{(p',\vec{v})}(t))$ for any $t\in [0,1]$, whence being an 
integral curve of $\nabla r_{\Sigma}$ with initial value $(p',\vec{v})\in U$.

With the previous Hessian $L^2$ estimate along a minimal geodesic in
Lemma~\ref{lem: Hess_L2_line}, we now aim to control the spreading of the
integral curves of $\nabla r_{\Sigma}$ more effectively in a small tubular
neighborhood around $\sigma$. To be precise, for any $t\in [\beta l,l-\beta l]$
fixed, and for each $r\in [0, \beta \slash 10]$, we consider the following core
neighborhood of $\sigma(t)$:
\begin{align*}
H_r^t(\sigma)\ :=\ \left\{y\in B_g(\sigma(t),r):\ \forall s\in
[0,(1-\beta) l-t],\
\frac{d_g\left(\psi^{\Sigma}_s(y),\sigma(t+s)\right)}{d_g(y,\sigma(t))} \le
\exp\left (\frac{2C_H(m,l)}{\sqrt{\beta l}}\sqrt{s}\right)\right\}.
\end{align*}
Intuitively speaking, such a neighborhood of $\sigma(t)$ consists of
points in $B_g(\sigma(t),r)$ that are carried by the gradient
flow $\psi^{\Sigma}_s$ up to a controllable distance for all $s\le (1-\beta)
l-t$. When the ambient manifold $M$ has a uniform Ricci curvature lower bound,
we could in fact conclude that \emph{almost every} point of $B_g(\sigma(t),r)$
are in $H^t_r(\sigma)$, provided that $r>0$ is sufficiently small:
\begin{lemma}\label{lem: core_nbhd}
With the same assumptions as in Lemma~\ref{lem: Hess_L2_line}, we fix $\beta\in
(0,10^{-2})$ and $t\in [\beta l,l-\beta l]$. For some $r>0$ sufficiently small,
we have $H_r^t(\sigma)=B_g(\sigma(t),r)$.
\end{lemma}

\begin{proof}
Let $W\subset M$ denote the open neighborhood of $\sigma([\beta
l,(1-\beta)l])$ where $r_{\Sigma}$ is smooth and let $r$ satisfy 
\begin{align*}
r\ \le\ \frac{1}{10}
\min\left\{\beta l, \min_{s\in [\beta l,(1-\beta)l]} d_g(\sigma(s),M\backslash
W), inj_g(\sigma(t)) \right\},
\end{align*}
where $inj_g(\sigma(t))$ denotes the injectivity radius at $\sigma(t)$. Further
shrinking $r$ if necessary, we also required that
$\inf_{B_g(\sigma(t),2r)}r_{\Sigma}\ge \frac{\beta l}{2}$.
By the compactness of $\sigma([\beta l,(1-\beta)l])$ and the Lipschitz
continuity of $r_{\Sigma}$, it is clear that $r>0$.

For any smooth point $y\in B_g(\sigma(t),r)$, there
is a unique $\vec{v}\in T_{\sigma(t)}M$ such that
$\exp_{\sigma(t)}\vec{v}_0=y$. We let
$\tau(s,u):=\psi_s^{\Sigma}(\exp_{\sigma(t)}u\vec{v}):[0,(1-\beta)l-t]\times
[-2,2]\to M$ be a parametrized $2$-dimensional submanifold in $M$. Notice that
for any $u\in [-2,2]$, $s\mapsto \tau(s,u)$ is an integral curve of $\nabla
r_{\Sigma}$ and thus a smooth geodesic. This implies that the variation $\tau$
is by geodesics and thus $J(s,u):=\partial_u\tau(s,u)$ is a Jacobi field along
the geodesic $s\mapsto \tau(s,u)$. Since for each $s$ fixed,
$u\mapsto \tau(s,u)$ furnishes a curve connecting $\sigma(t+s)=\tau(s,0)$ and
$\psi_s^{\Sigma}(y)=\tau(s,1)$, we have 
\begin{align*}
\forall s\in [0,(1-\beta)l-t],\quad 
d_g\left(\psi_s^{\Sigma}(y),\sigma(t+s)\right)\ \le\
\left|\tau(s,-)|_{[0,1]}\right|.
\end{align*}
Since $\left|\tau(s,-)|_{[0,1]}\right|=\int_0^1|J(s,u)|\ \text{d}u$, we would
like to compare $|J(0,u)|$ and $|J(s,u)|$. As $\mathcal{L}_{\nabla
r_{\Sigma}}J=0$, we have
\begin{align*}
\forall s\in [0,(1-\beta)l-t],\ \forall u\in [-2,2]\quad 
\partial_s|J(s,u)|^2\ =\ 2Hess_{r_{\Sigma}}(J(s,u),J(s,u)),
\end{align*} 
and thus 
\begin{align*}
\left|\partial_s\log |J(s,u)|^2\right|\ \le\
2|Hess_{r_{\Sigma}}|(\tau(s,u)).
\end{align*}
Integrating with respect to $s$ we see for any $s_1\in [0,(1-\beta)l-t]$ that 
\begin{align*}
\left|\log\frac{|J(t+s_1,u)|^2}{|J(t,u)|^2} \right|\ \le\
2\int_{0}^{s_1}|Hess_{r_{\Sigma}}|(\tau(t+s,u))\ \text{d}s\ \le\
2\left(\int_{\frac{\beta l}{4}}^{(1-\beta)l}
\left|Hess_{r_{\Sigma}}\right|^2(\tau)\right)^{\frac{1}{2}}\sqrt{s_1},
\end{align*}
since the geodesic $u\mapsto \tau(0,u)=\exp_{\sigma(t)}(u\vec{v})$ is at
least $\frac{\beta l}{4}$ away from $\Sigma$. Consequently, we have
\begin{align*}
\exp\left(-2\left(\int_{\frac{\beta l}{4}}^{(1-\beta)l}
\left|Hess_{r_{\Sigma}}\right|^2(\sigma)\right)^{\frac{1}{2}} \sqrt{s_1}\right)\
\le\ \frac{|J(t+s_1,u)|^2}{|J(t,u)|^2}\ \le\
\exp\left(2\left(\int_{\frac{\beta l}{4}}^{(1-\beta)l}
\left|Hess_{r_{\Sigma}}\right|^2(\sigma)\right)^{\frac{1}{2}} \sqrt{s_1}\right),
\end{align*} 
and by Lemma~\ref{lem: Hess_L2_line} we have
\begin{align*}
\forall s_1\in [0,(1-\beta)l-t], \forall u\in [0,1],\quad |J(t+s_1,u)|\ 
\le\ e^{\frac{2C_H(m,l)}{\sqrt{\beta l}}\sqrt{s_1}}|J(t,u)|.
\end{align*}
Integrating $u$ from $0$ to $1$ we get
\begin{align*}
\left|\tau(s_1,-)|_{[0,1]}\right|\ \le\ e^{\frac{2C_H(m,l)}{\sqrt{\beta
l}}\sqrt{s_1}}|\vec{v}|,
\end{align*}
and as $s_1$ is arbitrary in $[0,(1-\beta)l-t]$, we have
\begin{align}\label{eqn: core_distance}
\forall s\in [0,(1-\beta)l-t],\quad 
d_g\left(\psi^{\Sigma}_{s}(y),\sigma(t+s)\right)\ \le\
e^{\frac{2C_H(m,l)}{\sqrt{\beta l}}\sqrt{s}}d_g(y,\sigma(t)).
\end{align} 
By the definition of $H_r^t(\sigma)$ and the choice of $r$, we see that
$H_r^t(\sigma)= B_g(\sigma(t),r)$.
\end{proof}

From the proof of this lemma, we could clearly see that $H^t_r(\sigma)$ depends
on the specific manifold $M$ and geodesic $\sigma$, rather than being a uniform
neighborhood that we wish to find. In fact, it is impossible to get such a 
neighborhood in a uniform way; however, we notice that once the good
neighborhood $W$ is specified, the actual distance estimate (\ref{eqn:
core_distance}) only depends on the $L^2$ Hessian control of $r_{\Sigma}$ along
the interior of $\sigma$. Based on this observation, we will see in the sequel
that there is a subset $T_{\eta}^r(\sigma)\subset B_g(\sigma(t),r)$ of
sufficiently large measure, that resembles the key property of $H_r^t(\sigma)$:
the gradient flow lines of $\nabla r_{\Sigma}$ with initial data in
$T_{\eta}^r(\sigma)$ does \emph{not} spread too far away from $\sigma$.
Moreover, $T_{\eta,\varepsilon}^r(\sigma)$ is defined analytically and its
properties depend on the estimates uniformly.

\subsection{Parabolic approximation and effective estimates}
In order to define the desired subset that stays close to a given flow line of
$\nabla r_{\Sigma}$, we need to uniformly estimate the behavior of $r_{\Sigma}$,
especially bounding $Hess_{r_{\Sigma}}$. While impossible to control
$Hess_{r_{\Sigma}}$ in the $C^0$ sense, Colding and Naber observed in
\cite{ColdingNaber} that by parabolically smoothing $r_{\Sigma}$ (with $\Sigma$
being a single point in their setting), an $L^2_{loc}$ estimate around a given
flow line of $\nabla r_{\Sigma}$ is indeed possible, and is sufficient for the
purpose. For a general closed embedded submanifold $\Sigma$, we notice that once
the Laplacian comparison (\ref{eqn: Laplace_comparison}) for $r_{\Sigma}$ is
ready at hand, then all the estimates in \cite[\S 2]{ColdingNaber} go through
without any change, for the parabolic approximation of $r_{\Sigma}$. In this
subsection we summarize the relevant estimates and refer directly to the
corresponding ones in \cite[\S 2]{ColdingNaber}.

 Fix a minimal geodesic $\sigma:[0,l]\to M$ such that
$d_g(\sigma(t),\Sigma)=t$ for all $t\in [0,l)$, we let $p:=\sigma(0)\in
\Sigma$, $q:=\sigma(l)$ and denote $d^+(x):=l-d_g(q,x)$ for all $x\in M$.
We also put the notation 
\begin{align*}
M_{r,s}\ :=\ \left\{x\in M:\ r< l^{-1}r_{\Sigma}(x)<s\ \text{and}\ r< l^{-1}
d_g(q,x)<s \right\}.
\end{align*}
Now we consider the excess function $e^{\Sigma}:M\to \mathbb{R}$ defined as
$e^{\Sigma}:=r_{\Sigma}-d^+$. Since $r_{\Sigma}(x)\le d_g(x,p)$ for any $x\in
M$, we always have 
\begin{align}\label{eqn: compare_excess}
e^{\Sigma}(x)\le e_{p,q}(x),
\end{align}
 where $e_{p,q}(x)=d_g(p,x)-d^+(x)$ is the original excess function defined for
 a minimal geodesic connecting the two end points. By the excess function
 estimate due to Abresch and Gromoll \cite{AG90}, we have for any $t\in
 (0,l-r)$ (with $r>0$ sufficiently small),
\begin{align*}
\sup_{B_g(\sigma(t),r)} e^{\Sigma}\ \le\ \sup_{B_g(\sigma(t),r)} e_{p,q}\ \le\
C_{AG}(m)r^{1+\alpha_{AG}(m)},
\end{align*}
where $C_{AG}(m)>1$ and $\alpha_{AG}(m)\in (0,1)$ are dimensional constants. By
Proposition~\ref{prop: Laplace_comparison} we see that 
\begin{align}
\Delta e^{\Sigma}\ \le\ \frac{2C_{F_m}(l)}{r_{\Sigma}}.
\end{align} 
Consequently, we could invoke \cite[Corollary 2.4]{ColdingNaber} to obtain the
following estimate, which is a version of \cite[Theorem 2.8]{ColdingNaber}:
\begin{lemma}[Average excess estimate]\label{lem: excess}
For any $\beta\in (0,10^{-2})$ and $D\ge l$, there are constants
$C_{Ex}(m,D,\beta)>1$ and $\varepsilon_{Ex}(m,D,\beta)\in (0,1)$ such that if
$x\in M_{\beta,2}$ satisfies $e(x)\le \varepsilon^2l\le \varepsilon_{Ex}^2l$,
then
\begin{align}
\fint_{B_g(x,\varepsilon l)}e^{\Sigma}\ \le\ C_{Ex}\varepsilon^2l.
\end{align}
\end{lemma}

Now we let $\psi^{\pm}:M\to \mathbb{R}$ be the cut-off function given by
\cite[Lemma 2.6]{ColdingNaber} such that for some $\beta \in (0,10^{-2})$ we
have
\begin{align*}
\psi^-(x)=\begin{cases}
1\quad \text{if}\ \frac{\beta l}{4}<r_{\Sigma}(x)<8 l,\\
0\quad \text{if}\ r_{\Sigma}(x)\le \frac{\beta l}{16}\ \text{or}\
r_{\Sigma}(x)>16l;
\end{cases}
\psi^+(x)=\begin{cases}
1\quad \text{if}\ \frac{\beta l}{4}<d_g(q,x)<8 l,\\
0\quad \text{if}\ d_g(q,x)\le \frac{\beta l}{16}\ \text{or}\
d_g(q,x)>16l.
\end{cases}
\end{align*}
We now put $\psi:=\psi^+\psi^-$ and evolve $\psi r_{\Sigma}$, $\psi d^+$ and
$\psi e^{\Sigma}$ by the heat equation to obtain smooth functions $h_t$,
$d^+_t$ and $e^{\Sigma}_t$ on $M$, i.e.
we have 
\begin{align*}
(\partial_t-\Delta)h_t= 0\ \text{with}\ h_0 = \psi r_{\Sigma},\ 
(\partial_t-\Delta)d^+_t = 0\ \text{with}\ d^+_0 = \psi d^+,\ 
\text{and}\ (\partial_t-\Delta)e^{\Sigma}_t = 0\ \text{with}\ e^{\Sigma}_0=
\psi e^{\Sigma}.
\end{align*}
By uniqueness we clearly have $e^{\Sigma}_t=h_t-d^+_t$.

Now by (\ref{eqn: Laplace_comparison}) and \cite[Lemma 2.6]{ColdingNaber}, we
could estimate
\begin{align}
\Delta h_0\ =\ r_{\Sigma}\Delta \psi+2\langle \nabla \psi,\nabla
r_{\Sigma}\rangle+\psi \Delta r_{\Sigma}\ \le\ C(m,D,\beta)l^{-1},
\end{align}
where $C(m,D,\beta)$ depends on $C_{F_m}(D)$. Similarly, $\Delta d_0^+\ge 
-C(m,D,\beta)l^{-1}$ and $\Delta e^{\Sigma}_0\le C(m,D,\beta)l^{-1}$.
Moreover, $\Delta h_0$, $\Delta d_0^+$ and $\Delta e_0^{\Sigma}$ are supported
in $M_{\frac{\beta }{16},16}$. Consequently, by the proof of \cite[Lemma
2.10]{ColdingNaber} we see that for some positive constant $C(m,D,\beta)>0$,
\begin{align*}
\max\left\{\Delta h_t,-\Delta d_t^+,\ \Delta e^{\Sigma}_t\right\}\ \le\
C(m,D,\beta)l^{-1}.
\end{align*}
We could then plug this estimate into \cite[Lemma 2.11 and Lemma
2.13]{ColdingNaber} to obtain some new constants $C_{C^0}(m,D,\beta)>0$ and
$\bar{\varepsilon}_{C^0}(m,D,\beta)>0$ such that for any $\varepsilon\in
(0,\bar{\varepsilon}_{C^0})$,
\begin{align}
\sup_{M_{\frac{\beta}2,4}} \left|h_{\varepsilon^2l^2}-r_{\Sigma}\right|\le
C_{C^0}\left(\varepsilon^2l+e^{\Sigma}\right).
\end{align}
Moreover, since for any $\varepsilon$-geodesic $\sigma'$ connecting $p$ and $q$,
it holds $e(\sigma(t))<\varepsilon^2l$, and so does
$e_{p,q}^{\Sigma}(\sigma(t))$ by (\ref{eqn: compare_excess}), we have, by
\cite[Corollary 2.16]{ColdingNaber} that 
\begin{align}\label{eqn: h_C0}
\sup_{\sigma'\cap
M_{\frac{\beta}{2},4}}\left|h_{\varepsilon^2l^2}-r_{\Sigma}\right|\ \le\
C_{C^0}\varepsilon^2 l,\quad \text{and}\quad \sup_{t\in
(\frac{\beta}{2}l,(1-\frac{\beta}{2})l)}
\left|h_{\varepsilon^2l^2}(\sigma'(t))-t\right|\ \le\
C_{C^0}\varepsilon^2 l.
\end{align}
The gradient upper bound of $h_t$ could be obtained by the Bochner formula and
Li-Yau heat kernel upper bound (see \cite{LiYau}), as done in \cite[Lemma
2.17]{ColdingNaber}:
\begin{align}
\sup_{M_{\frac{\beta}2,4}} \left|\nabla h_{\varepsilon^2l^2}\right|\ \le\
1+C_{C^1}(m,D,\beta)\varepsilon^2l^2.
\end{align}
This estimate, together with (\ref{eqn: h_C0}) and \cite[Lemma
2.1]{ColdingNaber}, then implies an $H^1_{loc}$ estimate of
$h_{\varepsilon^2l^2}$ around the interior of the geodesic curve $\sigma$ as in
\cite[Theorem 2.18]{ColdingNaber}. Integration by parts along $\sigma$ and in
time, we could then obtain an $H^2_{loc}$ estimate of $h_{\varepsilon^2l^2}$,
as done in \cite[Theorem 2.19 and Lemma 2.20]{ColdingNaber} --- the proofs are
identical since $h_t$ satisfies exactly the same estimate as $h^-_t$ in
\cite[\S 2]{ColdingNaber}, and we only record the needed estimates:
\begin{proposition}\label{prop: approx_est} 
For each $\beta\in (0,10^{-2})$ and $D\ge l$, there are positive constants
$C_{Ap}(m,D,\beta)>1$ and $r_{Ap}(m,D,\beta)\le \varepsilon_{Ex}$
to the following effects: $\forall \varepsilon\in (0,r_{Ap}]$,
$\exists c^2\in [\frac{1}{2},2]$, such that
\begin{align}\label{eqn: Hess_sigma}
\int_{\beta l}^{(1-\beta)l}\left(\fint_{B_g(\sigma(s),\varepsilon l)}
\left|Hess_{h_{c^2\varepsilon^2l^2}}\right|^2\ \dvol_g\right)\ \text{d}s\ \le\
C_{Ap}l^{-2};
\end{align}
moreover, for any smooth point $x\in M_{\beta,2}$ with $e^{\Sigma}(x)\le
\varepsilon^2l$, let $\sigma_x$ denote the integral curve of $\nabla
r_{\Sigma}$ passing through $x$, then 
\begin{align}\label{eqn: H1_x}
\forall \beta r_{\Sigma}(x)\le s<t\le r_{\Sigma}(x), \quad
\int_{s}^t\left|\nabla h_{\varepsilon^2l^2}-\nabla
r_{\Sigma}\right|(\sigma_x(u))\ \text{d}u\ \le\ C_{Ap}\varepsilon
\sqrt{\frac{t-s}{l}}.
\end{align}
\end{proposition}

\subsection{Effective control of the geodesic spreading}
In this subsection, we prove the desired estimate of the distance between two
integral curves of $\nabla r_{\Sigma}$ in Theorem~\ref{thm: main3}.
This relies on the existence of some subset that remains (up to certain time)
close to a given flow line of $\nabla r_{\Sigma}$. The definition of such a set
is due to Colding and Naber \cite{ColdingNaber}. While our argument mimics the
original one in \cite[\S 3]{ColdingNaber}, it is much simplified thanks to
\cite[Proposition 3.6 and Corollary 3.7]{ColdingNaber}. In fact,
\cite[Proposition 3.6]{ColdingNaber} is the major technical input in Colding
and Naber's work, utilizing all estimates obtained from the parabolic
approximation in \cite[\S 2]{ColdingNaber} --- the proof of Theorem~\ref{thm:
main3} not just borrows from Colding and Naber's arguments, but also relies on
their results.

As in the last subsection, we consider a closed embedded 
submanifold $\Sigma\subset M$. Fixing any $q\in M\backslash \Sigma$, we let
$\sigma$ be a unit-speed minimal geodesic realizing $r_{\Sigma}(q)=:l$, and let
$q:=\sigma(l)$ and $p:=\sigma(0)\in \Sigma$. Since $\sigma$ is a minimal
geodesic connecting its two end points $p$ and $q$, we could apply
\cite[Proposition 3.6 and Corollary 3.7]{ColdingNaber} to $\sigma$ and see
\begin{lemma}[Interior volume comparison]\label{lem: volume_Harnack}
Suppose $\frac{1}{4}\le l\le D$, then there exist positive constants 
$\varepsilon_{CN}(m,D,\beta)<1$ and $r_{CN}(m,D,\beta)< 1$ such
that if $s,t\in [\beta l, (1-\beta)l]$ satisfy $|s-t|<\varepsilon_{CN}l$, then
for any $r\in (0,r_{CN}l]$,
\begin{align}\label{eqn: volume_estimate}
\frac{1}{2}\ \le\ \frac{B_g(\sigma(s),r)}{B_g(\sigma(t),r)}\ \le\ 2.
\end{align}
\end{lemma}
\begin{remark}\label{rmk: l_range}
In \cite[\S 3]{ColdingNaber}, this result is proven for $l=1$, and the
constants there only depend on the dimension. When $l\approx D>1$, the
constants $r_{CN}$ and $\varepsilon_{CN}$ are affected by $C_{F_m}(D)$ in
the Laplacian comparison, while the lower bound $l\ge \frac{1}{4}$ is
 required essentially due to (\ref{eqn: Hess_sigma}).
\end{remark}
We now make the notation for the scale 
$\bar{r}_0(m,D,\beta):=\frac{1}{4}\min\left\{10^{-2}\beta,
r_{Ap}, r_{CN}\right\}$, and define for any $r\in (0,\bar{r}_0]$ the subset 
\begin{align}\label{eqn: defn_big_A} 
\mathcal{A}_s^t(\sigma,r)\ :=\ \left\{z\in B_g(\sigma(t),r):\ \forall u\in 
[0,sl],\ \psi^{\Sigma}_u(z)\in B_g(\sigma(t+u),2r)\right\}.
\end{align}
Clearly, $\mathcal{A}_0^t(\sigma,r)=B_g(\sigma(t),r)$, since $\psi^{\Sigma}_0$
is the identity map; also notice that when $r,s>0$ are very small,
$H_r^t(\sigma)\subset \mathcal{A}_{s}^t(\sigma,r)$ by (\ref{eqn:
core_distance}).

We also let $\chi_{\sigma}^{s,t}$ be the characteristic function of
$\mathcal{A}_s^t(\sigma, r)\times \mathcal{A}_s^t(\sigma, r)$ in
$B_g(\sigma(t),r)\times B_g(\sigma(t),r)$, then for any $s\in [0,l-\beta 
l-t]$ and $\eta\in (0,10^{-2})$, we define quantities 
\begin{align}
\mathcal{F}^r_{\sigma}(x,y;s)\ :=\ &\int_0^{s}\chi_{\sigma}^{u,t}(x,y)
\left(\int_{\gamma_{\psi^{\Sigma}_u(x),
\psi^{\Sigma}_u(y)}}\left|Hess_{h_{c^2r^2}}\right|\right)\ \text{d}u, 
\label{eqn: Fxy}\\
\text{and}\quad I_{s}^t(\sigma,r)\ :=\ &\fint_{B_g(\sigma(t),r)\times
B_g(\sigma(t),r)}\mathcal{F}_{\sigma}^r(x,y;s)\
\dvol_g(x)\dvol_g(y),
\label{eqn: int_I}
\end{align}
where the constant $c^2\in [\frac{1}{2},2]$ depends on $r\le r_{Ap}l$ and is
guaranteed to exist by Proposition~\ref{prop: approx_est}. We also define the
subsets (notice that we omit writing the dependence on $t\in [\beta
l,(1-\beta)l]$)
\begin{align}
T_{\eta,s}^r(\sigma):= \left\{x\in B_g(\sigma(t),r):\
e^{\Sigma}(x)\le \frac{C_{Ex} r^2}{\eta l},\
\fint_{B_g(\sigma(t),r)}\mathcal{F}_{\sigma}^r(x,y;s)\ \dvol_g(y)\
\le \frac{I_{s}^t(\sigma,r)}{\eta}\right\},
\end{align}
and for each $x\in T_{\eta,s}^r(\sigma)$ we define
\begin{align}
T_{\eta,s}^r(\sigma,x):=\left\{y\in B_g(\sigma(t),r):\
e^{\Sigma}(y)\le \frac{C_{Ex} r^2}{\eta l},\
\mathcal{F}_{\sigma}^r(x,y;s)\ \le\
\frac{I_{s}^t(\sigma,r)}{\eta^2}\right\}.
\end{align}
By the average excess function estimate in Lemma~\ref{lem: excess} applied to
$B_g(\sigma(t),r)$ and Chebyshev's inequality, we have
\begin{align}
\frac{\left|T_{\eta,s}^r
(\sigma)\right|}{\left|B_g(\sigma(t),r)\right|}\ \ge\
1-2\eta,\quad\text{and}\quad \forall x\in T_{\eta,s}^r(\sigma),\
\frac{\left|T_{\eta,s}^r
(\sigma,x)\right|}{\left|B_g(\sigma(t),r)\right|}\ \ge\ 1-2\eta.
\label{eqn: Chebyshev}
\end{align}
Notice that these estimates are \emph{uniform}, and we would like to first
understand how the analytic conditions defining $T_{\eta,s}^r(\sigma)$
could affect the spreading of the flow lines of $\nabla r_{\Sigma}$:
\begin{lemma}[Effective distance estimate]\label{lem: distance_estimate}	
Fix $\eta \in (0,10^{-2})$ and $D\ge 1$ such that $\frac{1}{4}\le l\le D$, then 
there are constants $C_0(m,D,\beta)>0$ and $\varepsilon_0(m,D,\beta,\eta)\in
 (0,1)$ such that for any $s\in [0,\varepsilon_0 l]$ and any $r\in
 (0,\bar{r}_0]$, every pair of smooth points $x_1\in  T_{\eta,s}^r
 (\sigma)$ and $ x_2\in T_{\eta,s}^r(\sigma,x)\cap
 \mathcal{A}_s^t(\sigma,\xi r)$ (where we set the notation $\xi:=10^{-2}$), 
\begin{align}
\left|d_g\left(\psi^{\Sigma}_s(x_1),
\psi^{\Sigma}_s(x_2)\right)-d_g(x_1,x_2)\right|\ \le\ C_0\eta^{-2}r
\sqrt{s\slash l}.
\label{eqn: distance_estimate}
\end{align}
Especially, $\left|T_{\eta,s}^r(\sigma)\backslash
\mathcal{A}_{\varepsilon_0l}^t(\sigma,r)\right|=0$ for all $s\in
[0,\varepsilon_0 l]$.
\end{lemma}
\begin{proof}
Recalling that $\bar{r}_0(m,D,\beta)=\frac{1}{4}\min\left\{10^{-2}\beta,
r_{Ap}, r_{CN}\right\}$, we fix for any $r\in (0,\bar{r}_0]$ a smooth point
$x_1\in T_{\eta,s}^r(\sigma)$ and denote
\begin{align*}
\varepsilon(x_1)\ :=\ \sup\left\{s\le l-\beta
l -t:\ \forall u\in [0,s],\ \psi_{u}^{\Sigma}(x_1)\in
B_g(\sigma(t+u),2r)\right\}.
\end{align*} 
Without loss of generality, we may assume that $\varepsilon(x_1)\le
\varepsilon_{CN}(m,D,\beta)$. Clearly, when $s\le \varepsilon(x_1)$, $x_1\in
\mathcal{A}_s^t(\sigma,r)$; moreover, $\mathcal{A}_s^t(\sigma,\xi r)\subset
\mathcal{A}_s^t(\sigma,r)$. We want to understand how
$d_g\left(\psi^{\Sigma}_s(x_1),\sigma(t+s)\right)$ is controlled by the
properties of $T_{\eta,s}^r(\sigma)$. By the continuity of the mapping
$u\mapsto \psi^{\Sigma}_u(x_1)$ and the maximality of $\varepsilon(x_1)$, we
see that
\begin{align}
\psi^{\Sigma}_{\varepsilon(x_1)}(x_1)\not\in
B_g(\sigma(t+\varepsilon(x_1)),\frac{3}{2}r).
\label{eqn: away}
\end{align}
In fact, we will show that $\varepsilon(x_1)\ge \varepsilon_0l$ for suitably
chosen $\varepsilon_0$. Fix any $x_2\in T_{\eta,s}^r(\sigma,x_1)\cap
\mathcal{A}_s^t(\sigma,r)$ which is also a smooth point of $r_{\Sigma}$,
clearly $\chi_{\sigma}^{s,t}(x_1,x_2)=1$ for $s\le \varepsilon(x_1)$.
  We let $\sigma_1$ and $\sigma_2$ denote the integral curves of $\nabla
  r_{\Sigma}$ starting from $x_1$ and $x_2$, respectively. These are smooth
  geodesics. Since $r\le \frac{1}{4}r_{Ap}\le r_{Ap}l$,
  there is some $c^2\in [\frac{1}{2},2]$ so that (\ref{eqn: Hess_sigma}) holds.
  Now integrating (3.6) in \cite[Lemma 3.4]{ColdingNaber} for $s\le
  \varepsilon(x_1)$, we have
\begin{align}
\begin{split}
\left|d_g\left(\psi^{\Sigma}_s(x_1),\psi^{\Sigma}_s(x_2)\right)
-d_g(x_1,x_2)\right|\ \le\ &\int_0^s\left|\nabla h_{c^2r^2}-\nabla
r_{\Sigma}\right|(\sigma_1(u))\ \text{d}u\\
&+\int_0^s\left|\nabla h_{c^2r^2}-\nabla r_{\Sigma}\right|(\sigma_2(u))\
\text{d}u+\mathcal{F}_{\sigma}^{r}(x_1,x_2;s).
\end{split}
\label{eqn: distance_1}
\end{align}
We now estimate each term in the right-hand side of (\ref{eqn: distance_1}). 
By the bound on $r$, the estimate (\ref{eqn: H1_x}) in 
Proposition~\ref{prop: approx_est}, and the choice of $x_1$ and $x_2$, we see
for $i=1,2$,
\begin{align}
\begin{split}
\forall s\in [0,\varepsilon(x_1)],\quad \int_0^s \left|\nabla h_{c^2r^2}-\nabla
r_{\Sigma}\right|(\sigma_i(u))\ \text{d}u\ \le\ \sqrt{2}C_{Ap}rl^{-1}
\sqrt{s\slash l}.
\end{split}
\label{eqn: distance_2}
\end{align}
The last term on the right-hand side of (\ref{eqn: distance_1}) is by
definition bounded by $\eta^{-2}I_s^t(\sigma,r)$. By the segment inequality in 
\cite[Lemma 3.5]{ColdingNaber} and the definition of
$\mathcal{A}_s^t(\sigma,r)$, for any $s\in [0,\varepsilon(x_1)]$ we could
estimate $I_s^t(\sigma,r)$ as:
\begin{align*}
I_{s}^t(\sigma,r)\ 
&\le\ \int_0^{s}\left(
\frac{1}{\left|B_g(\sigma(t),r)\right|^2}
\int_{\psi^{\Sigma}_u(\mathcal{A}_u^t(\sigma, r))\times
\psi^{\Sigma}_u(\mathcal{A}_u^t(\sigma,r))}
\left(\int_{\gamma_{x,y}}\left|Hess_{h_{c^2r^2}}\right|\right)\
\dvol_g^2\right)\ \text{d}u\\
&\le\
\int_0^{s}\left(10r\ C_{Seg}(m)
\frac{\left|\psi^{\Sigma}_u(\mathcal{A}_u^t(\sigma,r))\right|}
{\left|B_g(\sigma(t),r)\right|^2}
\int_{B_g(\sigma(t+u),5r)}\left|Hess_{h_{c^2r^2}}\right|\ \dvol_g\right)\
\text{d}u\\
&\le\ \int_0^{s}\left(10r\ C_{Seg}(m)
\frac{\left|B_g(\sigma(t+u),2r)\right|}{\left|B_g(\sigma(t),r)\right|^2}
\int_{B_g(\sigma(t+u),5r)}\left|Hess_{h_{c^2r^2}}\right|\ \dvol_g\right)\
\text{d}u.
\end{align*}
We rely on the Bishop-Gromov volume comparison and Lemma~\ref{lem:
volume_Harnack} to compare $\left|B_g(\sigma(t),r)\right|$ and
$\left|B_g(\sigma(t+u),2r)\right|$ for $u\le \varepsilon(x_1)$: by (\ref{eqn:
volume_estimate}) we have
\begin{align}
\begin{split}
 I_s^t(\sigma,r)\ &\le\ 
\int_0^{s}\left(10r\ C(m,\bar{r}_0) \left(
\frac{\left|B_g(\sigma(t+u),r)\right|}{\left|B_g(\sigma(t),r)\right|}\right)^2
\fint_{B_g(\sigma(t+u),5r)}\left|Hess_{h_{c^2r^2}}\right|\ \dvol_g\right)\
\text{d}u\\
&\le\ 40r\ C(m,\bar{r}_0)\left(\int_{\beta l}^{l-\beta l}
\fint_{B_g(\sigma(u),5r)}\left|Hess_{h_{c^2r^2}}\right|^2\
\dvol_g\ \text{d}u\right)^{\frac{1}{2}}\sqrt{s}\\
&\le\ 40 C(m,\bar{r}_0)\sqrt{C_{Ap}} rl^{-1}\sqrt{s},
\end{split}
\label{eqn: distance_3}
\end{align}
where $C(m,\bar{r}_0)$ is the multiple of $C_{Seg}(m)$ by the doubling constant
on the space form of sectional curvature $-1$, up to scale $\bar{r}_0$, so
$C(m,\bar{r}_0)$ is ultimately determined by $m$, $D$ and $\beta$. 

Now (\ref{eqn: distance_1}), (\ref{eqn: distance_2}) and (\ref{eqn:
distance_3}) together imply that for every pair of smooth points $x_1\in
T_{\eta,s}^r(\sigma)$ and $x_2\in
T_{\eta,s}^r(\sigma,x_1)\cap \mathcal{A}_s^t(\sigma,r)$,
\begin{align}
\forall s\in [0,\varepsilon(x_1)],\quad
\left|d_g\left(\psi^{\Sigma}_s(x_1),\psi^{\Sigma}_s(x_2)\right)
-d_g(x_1,x_2)\right|\ \le\ C_0 \eta^{-2}r \sqrt{s\slash l},
\label{eqn: distance_4}
\end{align}
where $C_0:=8\sqrt{2}C_{Ap}+80C(m,\bar{r}_0)\sqrt{C_{Ap}}$ only depends
on $m$, $D$ and $\beta$; compare Remark~\ref{rmk: l_range}. 
In proving this estimate we only needed $x_1\in
T_{\eta,s}^r(\sigma)\cap \mathcal{A}_{s}^t(\sigma,r)$
and $x_2\in T_{\eta,s}^r(\sigma,x_1)\cap
\mathcal{A}_{s}^t(\sigma,r)$, and we emphasize that the stronger assumption 
$x_2 \in \mathcal{A}_{s}^t(\sigma,\xi r)$ is only used later to bound
$\varepsilon_0$.

Now we put $\varepsilon_0:=\min\left\{\varepsilon_{CN},\eta^4\slash
(16C_0^2)\right\}$ --- notice that $\varepsilon_0$ only depends on
$m$, $D$, $\beta$ and $\eta$. Suppose, for the purpose of a contradiction
argument, that the inequalities $\varepsilon(x_1)<\varepsilon_0 l$ hold, then
since actually $x_2\in \mathcal{A}_s^t(\sigma,\xi r)$, we have
\begin{align*}
d_g\left(\psi^{\Sigma}_s(x_2),\sigma(t+s)\right)\ \le\ 2\xi r\ \le\ 
\frac{r}{10} 
\end{align*} 
whenever $s\in [0,\varepsilon(x_1)]$, and the triangle inequality implies that
\begin{align}\label{eqn: distance_inclusion}
\forall s\in [0,\varepsilon(x_1)],\quad 
d_g\left(\psi^{\Sigma}_{s}(x_1),
\sigma(t+s)\right)\ \le\ \frac{7}{5}r,
\end{align}
contradicting (\ref{eqn: away}) at $s=\varepsilon(x_1)$. Therefore it must hold
that $\varepsilon(x_1) \ge \varepsilon_0 l$, and (\ref{eqn: distance_2}) is
valid for all $s\in [0,\varepsilon_0 l]$. Moreover, (\ref{eqn:
distance_inclusion}) tells that $x_1\in \mathcal{A}_s^t(\sigma,r)$ whenever
$s\le \varepsilon_0 l$.
\end{proof}

We are now ready to effectively control the spreading, under the diffeomorphisms
$\psi_s^{\Sigma}$, of the set $T_{\eta,s}^r(\sigma)$, for any integral
curve $\sigma$ of $\nabla r_{\Sigma}$, and for uniformly controlled
$\varepsilon>0$ and $r>0$.
\begin{lemma}[Controlling $T_{\eta,s}^r(\sigma)$ under $\psi_s^{\Sigma}$]
For the closed embedded submanifold $\Sigma\subset M$ and for any $l$ with 
$\frac{1}{4}\le l\le D$, there is a constant $\bar{\varepsilon}_0(m,D,\beta)
\in (0,1)$ such that $\left|T_{\eta,\bar{\varepsilon}_0l}^r(\sigma)\backslash
\mathcal{A}_{\bar{\varepsilon}_0 l}^t(\sigma,r)\right|=0$ for any $r\in
[0,r_0]$.
\label{lem: big_A}
\end{lemma}

\begin{proof}
We begin with recalling that by Lemma~\ref{lem: core_nbhd}, there is a small
$r'=r'(M,\sigma)>0$ and a core neighborhood specified by $H_{r'}^t(\sigma)$, a
full measure subset of $B_g(\sigma(t),r')$, such that flow lines of $\nabla
r_{\Sigma}$ initiating from it stay uniformly close to $\sigma$. Let us now fix
this neighborhood of $\sigma(t)$, which depends on the specific $M$ and $\sigma$.
Notice that if we set $\varepsilon_1:=(\ln 2)^2\beta\slash(4 C_H(m,D))^2$, then
by the definition of $H_{r'}^t(\sigma)$ and the proof of Lemma~\ref{lem:
core_nbhd}, we have
\begin{align}
\forall s\in [0,\varepsilon_1 l],\ \forall x\in H_{r'}^t(\sigma),\quad
d_g\left(\psi_s^{\Sigma}(x),\sigma(t+s)\right)\ \le\ 2d_g(x,\sigma(t)).
\label{eqn: core_inclusion}
\end{align} 

For any $r\in (0,\bar{r}_0]$, we set $r_i:=\xi^{i}r$ for $i=0,1,2,\ldots,I$,
where $I:=\left\lceil \log_{\xi}\frac{r'}{2r}\right\rceil$ is defined to be the
first natural number such that $r_I\le r'\slash 2$.

We then put $\bar{\varepsilon}_0=\min\left\{\varepsilon_0,
\varepsilon_1\right\}$ and pick a smooth point $x_0\in
T_{\eta,\bar{\varepsilon}_0l}^r(\sigma)$ with $r\le r_0$. We ``connect''
it to $H_{r'}^t(\sigma)$ by selecting $\{x_i\}_{i=0}^{I}$ inductively:
suppose $x_i$ is chosen, then pick any smooth point $x_{i+1}\in
T_{\eta,\bar{\varepsilon}_0l}^{r_i}(\sigma,x_i)\cap
T_{\eta,\bar{\varepsilon}_0l}^{r_{i+1}}(\sigma)$.
This is doable because (\ref{eqn: Chebyshev}) is independent of $r$ --- as long
as we choose $\eta:=\min\left\{10^{-2},
\frac{C(m,D,\beta)}{4(1+C(m,D,\beta))}\right\}$, where
\begin{align*}
C(m,D,\beta)\ :=\ \min\left\{\sup_{r\in [0,\bar{r}]}\frac{\Lambda_{-1}^m(\xi
r)}{\Lambda_{-1}^m(r)},\ 1\right\}
\end{align*}
is determined by $m$, $D$ and $\beta$ --- whence the sole dependence of $\eta$
on $m$, $D$ and $\beta$. We now have
\begin{align*}
\left|T_{\eta,\bar{\varepsilon}_0l}^{r_i}(\sigma,x_i)\right|
+\left|T_{\eta,\bar{\varepsilon}_0l}^{r_{i+1}}(\sigma)\right|\ &\ge\
(1-2\eta)\left(\left|B_g(\sigma(t),r_i)\right|
+\left|B_g(\sigma(t),r_{i+1})\right|\right)\\
&\ge\ (1-2\eta)(1+C(m,D,\beta))\left|B_g(\sigma(t),r_i)\right|\\
&>\ \left|B_g(\sigma(t),r_i)\right|,
\end{align*}
i.e. $T_{\eta,\bar{\varepsilon}_0l}^{r_i}(\sigma,x_i)\cap
T_{\eta,\bar{\varepsilon}_0l}^{r_{i+1}}(\sigma)$ has positive measure, and
especially there are smooth points of $r_{\Sigma}$ in the intersection. We
denote by $\sigma_i$ the integral curve of $\nabla r_{\Sigma}$ with initial
value $x_i$. Clearly we could select $x_i$ so that each $\sigma_i$ is a smooth
geodesics on $[0,\bar{\varepsilon}_0l]$.

According to (\ref{eqn: core_inclusion}), $x_I\in
\mathcal{A}_{s}^t(\sigma,r_I)
=\mathcal{A}_{s}^t(\sigma,\xi r_{I-1})$ whenever $s\in
[0,\varepsilon_1l]$.
Therefore, applying Lemma~\ref{lem: distance_estimate} to $x_{I-1}\in
T_{\eta,\bar{\varepsilon}_0}^{r_{I-1}}(\sigma)$ and $x_{I}\in
T_{\eta,\bar{\varepsilon}_0}^{r_{I-1}}(\sigma,x_{I-1})\cap
\mathcal{A}_{\bar{\varepsilon}_0l}^t(\sigma,r_I)$, we could obtain
\begin{align*}
\forall s\in \left[0,\bar{\varepsilon}_0 l\right],\quad
d_g\left(\psi^{\Sigma}_s(x_{I}),\psi^{\Sigma}_s(x_{I-1})\right)\ \le\
\left(\frac{5}{4}+\xi\right)r_{I-1}.
\end{align*}
This further implies that for any $s\le \bar{\varepsilon}_0l$,
\begin{align}
\begin{split}
d_g\left(\psi^{\Sigma}_s(x_{I-1}),\sigma(t+s)\right)\ &\le\
d_g\left(\psi^{\Sigma}_s(x_I),\sigma(t+s)\right)
+\left(\frac{5}{4}+\xi\right)r_{I-1}\\
&\le\ \left(\frac{5}{4}+3\xi\right)r_{I-1}.
\end{split}
\label{eqn: distance_I}
\end{align}
Especially, $x_{I-1}\in \mathcal{A}_{\bar{\varepsilon}_0l}^t(\sigma,r_{I-1})$.

We could then apply Lemma~\ref{lem: distance_estimate} to the pair of
smooth points $x_{I-2}$ and $x_{I-1}$, and conclude that $x_{I-2}\in
\mathcal{A}_{\bar{\varepsilon}_0l}^t(\sigma,r_{I-2})$. Repeating the same
argument another $I-2$ steps and by the choice of $\xi= 10^{-2}$, we get for
any $s\le \bar{\varepsilon}_0 l$,
\begin{align*}
d_g\left(\psi^{\Sigma}_s(x_0),\sigma(t+s)\right)\ \le\
&d_g\left(\psi^{\Sigma}_s(x_I),\sigma(t+s)\right)
+\left(\frac{5}{4}+\xi\right)r\sum_{i=0}^{I-1}\xi^i\\
<\ &2r.
\end{align*}
Especially, this implies that $x_0\in
\mathcal{A}_{\bar{\varepsilon}_0l}^t(\sigma,r)$. But since the collection of
smooth points of $r_{\Sigma}$ is a full measure subset of $M$, we have
$\left|T_{\eta,\bar{\varepsilon}_0l}^r(\sigma) \backslash
\mathcal{A}_{\bar{\varepsilon}_0l}^t(\sigma,r)\right|=0$ for any $r\in
(0,\bar{r}_0]$.
\end{proof}

We could now control the distance of two minimal geodesics emanating from
closeby parallel initial data along $\Sigma$, as promised in Theorem~\ref{thm:
main3}: 
\begin{proof}[Proof of Theorem~\ref{thm: main3}]
 Lets fix some $\theta'\in (0,1)$ as the largest number such that 
 \begin{align*}
 \forall r\in [0,\bar{r}_0],\quad
 \frac{\Lambda_{-1}^m((1-\theta')r)}{\Lambda_{-1}^m((1+\theta')r)}\ \ge\
 \frac{3}{4}.
 \end{align*}
 Notice that by the continuity of $\Lambda_{-1}^m(s)$ in $s$, such $\theta'$
 exists, and is determined by $m$, $D$ and $\beta$, due to the dependence of 
 $\bar{r}_0$ on these parameters.
 
 Now for any pair of flow lines $\sigma_0$ and $\sigma_1$ of $\nabla
 r_{\Sigma}$, with parallel initial data along $\Sigma$ such that
 $d_g(\sigma_0(t),\sigma_1(t))\le 2\theta'\bar{r}_0$, let us put
 $r:=(2\theta')^{-1}d_g(\sigma_0(t),\sigma_1(t))$, and let 
 $A:=B_g(\sigma_0(t),r)\cap B_g(\sigma_1(t),r)$ denote the intersection. For
 $i=0,1$, the volume comparison tells that 
 \begin{align*}
 \left|A\right|\ \ge\
 \frac{3}{4} \left|B_g(\sigma_i(t),r)\right|.
  \end{align*} 
 On the other hand, by (\ref{eqn: Chebyshev}) we have for each $i=0,1$, 
 \begin{align*}
 \left|T_{\eta,\bar{\varepsilon}_0 l}^r(\sigma_i)\cap A\right|\
 \ge\ (\frac{3}{4}-2\eta)\left|B_g(\sigma_i(t),r)\right|.
 \end{align*}
Consequently, by the assumption $\eta\le  10^{-2}$ we have
 \begin{align}\label{eqn: positive_intersection}
 \left|T_{\eta,\bar{\varepsilon}_0l}^r(\sigma_0)\cap
 T_{\eta,\bar{\varepsilon}_0l}^r(\sigma_1)\right|\ \ge\ \frac{1}{4}|A|\ >\ 0.
 \end{align}
By Lemma~\ref{lem: big_A} we know that for $i=0,1$, 
$\left|T_{\eta,\bar{\varepsilon}_0l}^r(\sigma_i)\backslash
\mathcal{A}_{\bar{\varepsilon}_0 l}^{t}(\sigma_i,r)\right|=0$, and thus by the
definition (\ref{eqn: defn_big_A}) of
$\mathcal{A}_{\bar{\varepsilon}_0l}^{t}(\sigma_i,r)$, we have
  \begin{align}\label{eqn: volume_after_flow}
  \forall s\in [0,\bar{\varepsilon}_0 l],\quad \left|\psi^{\Sigma}_{s}
  \left(T_{\eta,\bar{\varepsilon}_0l}^r(\sigma_i)\right)\backslash
  B_g(\sigma_i(t+s),2r)\right|\ =\ 0\quad \text{for}\quad i\ =\
  0,1. \end{align} 
 Since both $\sigma_0$ and $\sigma_1$ are integral curves of $\nabla 
 r_{\Sigma}$, which is smoothly defined almost everywhere on
 $B_g(\sigma_0(t),r)\cup B_g(\sigma_1(t),r)$ --- the very reason that we
 reworked Colding and Naber's original proof to suit $\nabla r_{\Sigma}$ --- we
 have
  \begin{align*}
  \forall s\in [0,\bar{\varepsilon}_0 l],\quad 
  \psi_{s}^{\Sigma}\left(T_{\eta,\bar{\varepsilon}_0 l}^r (\sigma_0)\cap
  T_{\eta,\bar{\varepsilon}_0 l}^r(\sigma_1)\right)\ \subset\
  \psi^{\Sigma}_{s}\left(T_{\eta,\bar{\varepsilon}_0 l}^r
  (\sigma_0)\right)\cap \psi^{\Sigma}_{s}
  \left(T_{\eta,\bar{\varepsilon}_0 l}^r(\sigma_1)\right).
  \end{align*}
 Especially, by (\ref{eqn: positive_intersection}) and (\ref{eqn:
 volume_after_flow}) we clearly see that
  \begin{align*}
  \forall s\in [0,\bar{\varepsilon}_0 l],\quad 
  \left|B_g(\sigma_0(t+s),2r)\cap
  B_g(\sigma_1(t+s),2r)\right|\ \ge\
  \left|\psi_s^{\Sigma}\left(T_{\eta,\bar{\varepsilon}_0l}^r(\sigma_0)\cap
  T_{\eta,\bar{\varepsilon}_0l}^r(\sigma_1)\right)\right|\ >\ 0.
  \end{align*} 
  Consequently, we have the distance bound for the geodesics  
  \begin{align}
  \forall s\in [0,\bar{\varepsilon}_0l],\quad 
  d_g(\sigma_0(t+s),\sigma_1(t+s))\ \le\
  4(2\theta')^{-1}d_g(\sigma_0(t),\sigma_1(t)).
  \end{align}
  Therefore, we could start from $t=\beta l$ and iterate the above estimate
  along $\sigma_0$ and $\sigma_1$, to see that as long as $d_g(\sigma_0(\beta
  l),\sigma_1(\beta l))\le 
\left(2^{-1}\theta'\right)^{1+2\bar{\varepsilon}_0^{-1}}\bar{r}_0=:\bar{r}$,
then
 \begin{align}
 \forall t\in [\beta l,(1-\beta)l],\quad 
 d_g(\sigma_0(t),\sigma_1(t))\ \le\ \bar{C}d_g(\sigma_0(\beta l),\sigma_1(\beta
 l)),
 \end{align}
 where $\bar{C}:= \left(\frac{2}{\theta'}\right)^{\lceil
 2\bar{\varepsilon}_0^{-1}\rceil}$ and $\bar{r}\in (0,1)$ clearly only depend 
 on $m$, $D$ and $\beta$.
\end{proof}
\begin{remark}
With suitable controls on the sectional curvature and the second fundamental
form of $\Sigma$ around $\sigma_0(0)$ and $\sigma_1(0)$, we may find some
uniform $\bar{C}'>0$ depending on these data, such that
\begin{align*}
d_g\left(\sigma_0(t),\sigma_1(t)\right)\ \le\
\bar{C}'d_{g}\left(\sigma_0(0),\sigma_1(0)\right).
\end{align*}
\end{remark}

\section{The first Betti numbers and dimensional difference}
With the heuristic discussion in \S 2 and technical preparation in \S 3 and
\S4, we are now ready to fill in the details in proving the first claim of 
Theorem~\ref{thm: main1}. We assume that $(M,g)\in \mathcal{M}_{Rc}(m)$ and
$(N,h)\in \mathcal{M}_{Rm}(k,D,v)$ (with $k\le m$) satisfy $d_{GH}(M,N)\le
10^{-1}\delta$ for some $\delta\in (0,1)$, and our task would be to determine
the range of $\delta$ uniformly according to $m$, $D$ and $v$, so that the
first claim of Theorem~\ref{thm: main1} holds. We also let $\Phi:M\to N$ denote
a $10^{-1}\delta$-Gromov-Hausdorff approximation.

Since we have known that $b_1(M)-b_1(N)=\rank\ H_1^{\delta}(M;\mathbb{Z})
=:l_M$ whenever $\delta<10^{-3}\bar{\iota}_{hr}(k,D,v)$, as mentioned before, we
would like to ``localize'' the torsion-free generators of
$H_1^{\delta}(M;\mathbb{Z})$ to each point $p\in M$, as torsion-free generators
of $\tilde{\Gamma}_{Nil}(p)$. By the Hurewicz theorem, we could find a total
number of $b_1(M)$ geodesic loops whose homology classes generate
$H_1(M;\mathbb{Z})\slash Torsion$. By Proposition~\ref{prop: rank}, we know
that exactly $b_1(N)$ of these loops are of lengths at least
$10^{-1}\bar{\iota}_{hr}$, and the rest of these loops have lengths not
exceeding $10\delta$. Since $H_1^{\delta}(M;\mathbb{Z})$ is a subgroup of
$H_1(M;\mathbb{Z})$ and $\rank\ H_1^{\delta}(M;\mathbb{Z})=b_1(M)-b_1(N)$, we
know that there are geodesic loops $\gamma_1',\ldots,\gamma_{l_M}'$ in $M$,
with lengths not exceeding $10\delta$, and that $[\![\gamma'_1]\!],\ldots,
[\![\gamma'_{l_M}]\!]$ generate $H_1^{\delta}(M;\mathbb{Z})\slash Torsion$.

For each $i=1,\ldots,l_M$, we now let $\gamma_i$ be a length minimizer in the
free homotopy class of $\gamma'_i$. Clearly,
$[\![\gamma_i]\!]=[\![\gamma_i']\!]\in H_1(M;\mathbb{Z})$, and $|\gamma_i|\le
10\delta$. Moreover,  each $\gamma_i:[0,1]\to M$ is a closed geodesic. Letting
$\pi:\widetilde{M}\to M$ denote the universal covering and equipping
$\widetilde{M}$ with the covering metric $\pi^{\ast}g$, we see that each
$\gamma_i$ lifts to a complete geodesic $\tilde{\gamma}_i$ in $\widetilde{M}$. 
We denote $\Sigma_i:=\tilde{\gamma}_i(\mathbb{R})$, which is clearly a closed embedded
smooth submanifold of $\widetilde{M}$. Also each $\gamma_i$ acts on
$(\widetilde{M},\pi^{\ast}g)$ isometrically, while restricting to a translation
along $\Sigma_i$ by distance $|\gamma_i|\le 10\delta$.

We now fix an arbitrary $p\in M$, and for each $i=1,\ldots,l_M$, let
$\sigma_i:[0,d_i]\to M$ be a unit speed minimal geodesic that realizes 
$d_g(p,\gamma_i([0,1]))=:d_i$, with $\sigma_i(0)= \gamma_i(t_i)$ for some
$t_i\in [0,1)$ and $\sigma_i(d_i)=p$. Clearly $\dot{\sigma}_i(0)\perp
\dot{\gamma}_i(t_i)$. Fixing some $\tilde{p}\in \pi^{-1}(p)$ in the universal
covering space $\widetilde{M}$, we could uniquely lift each $\sigma_i$
($i=1,\ldots,l_M$) to a minimal geodesic $\tilde{\sigma}_i:[0,d_i]\to
\widetilde{M}$ of unit speed with $\tilde{\sigma}(d_i)=\tilde{p}$. Clearly
$\tilde{\sigma}_i(0)\in \Sigma_i$ and we could parametrize
$\tilde{\gamma}_i:\mathbb{R}\to \widetilde{M}$ so that
$\tilde{\gamma}_i(t_i)=\tilde{\sigma}_i(0)=:\tilde{q}_i$, and that
$\dot{\tilde{\sigma}}_i(0)\perp \Sigma_i$.
 
 Notice that for each $i=1,\ldots,l_M$, the isometric action $\gamma_i$ sends
 $\tilde{\sigma}_i$ to another minimal geodesic $\gamma_i.\tilde{\sigma}_i$ in
 $\widetilde{M}$, which realizes the distance $d_{\pi^{\ast}g}(\gamma_i.\tilde{p},
 \Sigma_i) =d_{\pi^{\ast}g}(\gamma_i.\tilde{p},\gamma_i.\tilde{q}_i)=d_i$.
 Moreover, we have $(\gamma_i^n.\tilde{\sigma}_i)'(0)\perp 
 \dot{\tilde{\gamma}}_i(t_i+n|\gamma_i|)\in T_{\gamma_i^n.\tilde{q}_i}\Sigma_i$
 for any $n\in \mathbb{Z}$. Now we would like to estimate
 $d_{\pi^{\ast}g}(\gamma_i^n.\tilde{p}, \tilde{p})$ for a suitable positive
 power $n$.

We assume $\delta<\delta_1$ with  
$\delta_1 := \min\left\{10^{-1}\Psi_{NZ}\left(\min \left\{\bar{C}^{-1}
\Psi_{NZ}(\delta_{Nil},1,m),\bar{r}\right\},1,m \right),
10^{-3}\bar{\iota}_{hr}(k,D,v)\right\}$,
 where the uniform constants $\bar{C}(m,2D,\beta)>1$ and $\bar{r}(m,2D,\beta)\in
(0,1)$ are obtained from Theorem~\ref{thm: main3} by setting 
$\beta:=\min\left\{10^{-3}, (4D)^{-1}\right\}$, the uniform constant 
$\Psi_{NZ}(\varepsilon,1,m)\in (0,\varepsilon)$ is obtained
from \cite[Lemma 5.2]{NaberZhang} for any $\varepsilon\in (0,1)$ given, and the
uniform constants $\bar{\iota}_{hr}(m,D,v)$ and $\delta_{Nil}(m,D,v)\in (0,1)$
are determined in Lemma~\ref{lem: nil_rank}.

If $d_i< 1$, since $|\gamma_i|\le 10\delta<\Psi_{NZ}\left(\min 
\left\{\bar{C}^{-1} \Psi_{NZ}(\delta_{Nil},1,m),\bar{r}\right\},1,m \right)$,
by \cite[Lemma 5.2]{NaberZhang} we know that there is some uniform
$N_{NZ}(m,D,\iota)\in \mathbb{N}$ such that for some $k_i\le N_{NZ}$,
\begin{align}\label{eqn: dle1}
d_{\pi^{\ast}g}\left(\gamma_i^{k_i}.\tilde{\sigma}_i(d_i),
\tilde{\sigma}_i(d_i)\right)\ \le\ \bar{C}^{-1}
\Psi_{NZ}(\delta_{Nil},1,m)\ <\ \delta_{Nil}.
\end{align}
Then obviously $\gamma_i\in \widetilde{G}_{\delta_{Nil}}(p)$ by 
definition. 

If $d_i\ge 1$ instead, we will apply Theorem~\ref{thm: main3} to
$(\widetilde{M},\pi^{\ast}g)$, with $\Sigma_i:=\tilde{\gamma}_i$. Since 
$\beta d_i\le \beta D<\frac{1}{2}$, from the previous case we have some $k_i\le
N_{NZ}$ so that 
\begin{align*}
d_{\pi^{\ast}g}\left(\gamma_i^{k_i}.\tilde{\sigma}_i(\beta d_i),
\tilde{\sigma}_i(\beta d_i)\right)\ \le\
\min\left\{\bar{C}^{-1}\Psi_{NZ}(\delta_{Nil},1,m),\bar{r}\right\}.
\end{align*}
Then Theorem~\ref{thm: main3}, applied to the minimal geodesics
$\tilde{\sigma}_i$ and $\gamma_i^{k_i}.\tilde{\sigma}_i$, gives
\begin{align}\label{eqn: dge1'}
d_{\pi^{\ast}g}\left(\gamma_i^{k_i}.\tilde{\sigma}_i((1-\beta)d_i),
\tilde{\sigma}_i((1-\beta)d_i)\right)\ \le\ \Psi_{NZ}(\delta_{Nil},1,m).
\end{align}
Now since $d_{\pi^{\ast}g}\left(\tilde{p},
\tilde{\sigma}_i((1-\beta)d_i)\right)=
\left|\tilde{\sigma}_i|_{[d_i,(1-\beta)d_i]}\right|=\beta d_i<1$, we could
apply \cite[Lemma 5.2]{NaberZhang} again to see that for some $k_i'\le N_{NZ}$,
\begin{align}\label{eqn: dge1''}
d_{\pi^{\ast}g}\left(\gamma_i^{k_i k_i'}.\tilde{p},\tilde{p} \right)\ =\ 
d_{\pi^{\ast}g}\left( (\gamma_i^{k_i})^{k_i'}.
\tilde{\sigma}_i(d_i), \tilde{\sigma}_i(d_i)\right)\ \le\
\delta_{Nil}.
\end{align} 
This shows that $\gamma_i^{k_ik_i'}\in
\widetilde{\Gamma}_{\delta_{Nil}}(p)$ for each $i=1,\ldots,l_M$.

We now connect $\tilde{p}$ to each
$\gamma_i^{k_ik_i'}.\tilde{p}=\gamma_i^{k_ik_i'}.
\tilde{\sigma}_i(d_i)$ by a minimal geodesic $\tilde{\gamma}_{i;p}$. Since
the curve
\begin{align*}
\tilde{\sigma}_i^{-1}\ast
\tilde{\gamma}_i|_{\left[t_i,t_i+k_ik_i'\right]}\ast
\left(\gamma_i^{k_i k_i'}.\tilde{\sigma}_i\right)\ast
\left(\tilde{\gamma}_{i;p}\right)^{-1}
\end{align*}
is actually a loop based at $\tilde{p}\in \widetilde{M}$, it is null homotopic
in $\widetilde{M}$ as $\pi_1(\widetilde{M},\tilde{p})=0$. Consequently, the
loop $\gamma_{i;p}:=\pi\circ \tilde{\gamma}_{i;p}$ in $M$ is based at $p\in M$
and is free homotopic to the loop $\gamma_i^{k_ik_i'}=\pi\circ
\tilde{\gamma}_i|_{\left[t_i,t_i+k_ik_i'|\gamma_i|\right]}$, along the curve
$\sigma_i$. Especially, $[\![\gamma_{i;p}]\!]=k_ik_i'[\![\gamma_i]\!]\in
H_1(M;\mathbb{Z})$.
On the other hand, by the estimate $|\tilde{\gamma}_{i;p}|
=d_{\pi^{\ast}g}\left( \gamma_i^{k_i k_i'}.\tilde{p}, \tilde{p} \right) \le
\delta_{Nil}$, we know that $[\gamma_{i;p}]\in
\widetilde{G}_{\delta_{Nil}}(p)\le \pi_1(M,p)$. The loop $\gamma_{i;p}$ is the
``slided'' loop of $\gamma_i^{k_i k_i'}$ to $p\in M$, as mentioned in the
introduction.

Notice that $\gamma_{1;p},\ldots,\gamma_{l_M;p}$ actually define independent
torsion-free elements in $\widetilde{G}_{\delta_{Nil}}(p)$: otherwise, there is
a non-trivial relation 
\begin{align*}
\gamma_{i_1;p}\ast \gamma_{i_2;p}\ast \cdots \ast \gamma_{i_l;p}\ \simeq_M\
\bar{p},
\end{align*}
with each $\gamma_{i_j;p}$ either representing a torsion element in $\pi_1(M,p)$
or belonging to $\left\{\gamma_{1;p},\ldots,\gamma_{l_M;p}\right\}$, for
$j=1,\ldots,l$; but for $i=1,\ldots,l_M$, letting $\tilde{k}_i\in \mathbb{Z}$
denote the (oriented) number of copies $\gamma_{i;p}$ appeared in the above
vanishing homotopic equation, we have the corresponding homological relation 
\begin{align*}
\tilde{k}_1[\![\gamma_{1;p}]\!]+\cdots+\tilde{k}_{l_M}[\![\gamma_{l_M;p}]\!]\
=\ \tilde{k}_1(k_1 k_1')[\![\gamma_1]\!]+\cdots +
\tilde{k}_{l_M}(k_{l_M} k_{l_M}')[\![\gamma_{l_M}]\!] \in
Tor(H_1(M;\mathbb{Z})),
\end{align*}
and this contradicts our choice of the homology classes
$[\![\gamma_1]\!],\ldots,[\![\gamma_{l_M}]\!]$ as a minimal set of generators
of $H_1^{\delta_1}(M;\mathbb{Z})\slash Torsion$, unless $\tilde{k}_1=\cdots
=\tilde{k}_{l_M}=0$.
Therefore, when $\delta<\delta_1$ we have shown that $\rank\
\widetilde{G}_{\delta_{Nil}}(p)\ge l_M= \rank\ H_1^{\delta}(M;\mathbb{Z})$, and 
consequently, by Lemma~\ref{lem: nil_rank} and Proposition~\ref{prop: rank} we
have shown the first claim of Theorem~\ref{thm: main1}: $m-n\ge b_1(M)-b_1(N)$.

\begin{remark}\label{rmk: homomorphism}
Since $\widetilde{G}_{\delta_{Nil}}(p)\le \pi_1(M,\mathbb{Z})$, we could also
abelianize $\widetilde{G}_{\delta_{Nil}}(p)$ to obtain a subgroup of
$H_1(M;\mathbb{Z})$, according to the Hurewicz theorem. 
In fact, the above argument defines an injective group homomorphism
\begin{align*}
\varphi_p: \bigoplus_{i=1}^{l_M}\mathbb{Z}[\![\gamma_i]\!] \to
\widetilde{G}_{\delta_{Nil}}(p)\slash \left([\pi_1(M,p),\pi_1(M,p)]\cap
\widetilde{G}_{\delta_{Nil}}(p)\right),
\end{align*}
sending $[\![\gamma_i]\!]$ to $[\gamma_i]^{k_ik_i'}\cdot 
\left([\pi_1(M,p), \pi_1(M,p)]\cap \widetilde{G}_{\delta_{Nil}}(p)\right)$. Here
$\oplus_{i=1}^{l_M}\mathbb{Z}[\![\gamma_i]\!]\le H_1^{\delta_1}(M;\mathbb{Z})$
is a free $\mathbb{Z}$-module of rank $l_M$. Consequently, we see that the
abelianization of $\widetilde{G}_{\delta_{Nil}}(p)$, as a finitely generated
abelian group, has rank at least $l_M$.
\end{remark}

\section{Ricci flow smoothing and the rigidity case} 
In this section we prove Theorem~\ref{thm: main2}, and consequently the
equality case of Theorem~\ref{thm: main1}. Throughout this section we assume
that $(M,g)\in \mathcal{M}_{Rc}(m)$ and $(N,h)\in \mathcal{M}_{Rm}(k,D,v)$ with
$k\le m$, and that $d_{GH}(M,N)\le \delta$ for some $\delta\in
(0,10^{-1}\delta_1)$. We have shown that $b_1(M)-b_1(N)\le m-k$, and in this
section we further assume that $b_1(M)-b_1(N)=m-k$. We would like to find some
positive $\delta_{RF}\le 10^{-1}\delta_{1}$ so that the conclusion of
Theorem~\ref{thm: main2} holds, and another positive $\delta_B\le \delta_{RF}$
so that when $\delta\le \delta_{B}$, we could see $M$ as a
$\mathbb{T}^{m-k}$-bundle over $N$.

\subsection{Starting Ricci flows with collapsing initial data}
Since $(M,g)$ is a complete manifold and $\diam (M,g)\le D+2\delta<\infty$, we
know that $M$ has to be a closed manifold. Therefore, we could appeal to
Hamilton's short-time existence result in \cite{Hamilton} to see that the
initial value problem to the Ricci flow equation
\begin{align*}
\begin{cases}
\partial_t g(t)\ &=\ -2\Rc_{g(t)}\quad \text{when}\ t\ge 0,\\
g(0)\ &=\ g,
\end{cases}
\end{align*}
is solvable up to \emph{some} positive time depending on specific $(M,g)$. By
Shi's estimate \cite{Shi}, the evolved metric $g(t)$ has much improved
regularity for any $t>0$ fixed:
\begin{align*}
\forall l\in \mathbb{N},\quad \sup_{M}\left|\nabla^l \Rm_{g(t)}\right|_{g(t)}\
\le\ C_S(l)t^{-1-l}.
\end{align*}
We could therefore regard $g(t)$ as a ``smoothing metric'' to the original
metric $g=g(0)$. However, Shi's estimate blows out of control when $t\searrow
0$, and if one would like to find a smoothing metric with uniform regularity
control, a uniform lower bound on the existence time of the Ricci flow is
desired. Here the uniformity refers to the dependence on the data $m$, $D$
and $v$.

A typical approach in obtaining a uniform lower bound on the maximal existence
time of a Ricci flow solution is to rely on Perelman's pseudo-locality theorem
(see \cite[Theorem 10.1]{Perelman}), whose proof in the complete non-compact
setting could be found in \cite[\S 8]{CTY11}. Since the initial data we consider
have a uniform Ricci curvature lower bound, we will invoke the version of the
pseudo-locality theorem due to Tian and the second named author
(\cite[Proposition 3.3]{TianWang}):
\begin{proposition}[Pseudo-locality for Ricci flows]\label{prop: CTY11}
For any $\alpha\in (0,10^{-2}m^{-1})$, there are positive constants
$\delta_{P}=\delta_{P}(m,\alpha)<1$ and $\varepsilon_P
=\varepsilon_P(m,\alpha)<1$, such that for any $m$-dimensional complete
non-compact Ricci flow solution $(M,g(t))$ defined on $t\in [0,T)$, if each
time slice has bounded sectional curvature, then for any $x\in M$ satisfying
\begin{align}
 \Rc_{g(0)}\ \ge\ -\delta_P^2(m-1)g(0)\quad
 \text{on}\ B_{g(0)}(x,\delta_P^{- 1}) \quad
\text{and}\quad \delta_P^m\left|B_{g(0)}(x,\delta_P^{-1})\right|_{g(0)}\ \ge\
(1-\delta_P)\omega_m,
\label{eqn: TianWang}
\end{align}
we have the following curvature bound for any $t\in [0,T)\cap
(0,\epsilon_P^2]$:
\begin{align}\label{eqn: pseudolocality_Rm_bound}
 \left|\Rm_{g(t)}\right|_{g(t)}(x)\ \le\ \alpha t^{-1}+\varepsilon_P^{-2}.
\end{align}
\end{proposition}
\begin{proof}
Notice that the assumed initial Ricci curvature lower bound in (\ref{eqn:
TianWang}) implies an initial scalar curvature lower bound. Therefore, if the
theorem were to fail, then the proof of \cite[Theorem 8.1]{CTY11} provides a
contradicting sequence which validates \cite[(15) and (16)]{TianWang}.
Notice that up to this stage, the assumption on the isoperimetric constant in
\cite[Theorem 8.1]{CTY11} has never been used. Starting from \cite[(15) and
(16)]{TianWang}, the rest of the proof of \cite[Proposition 3.1]{TianWang} goes
through verbatim, producing a contradiction and concluding the proof.
\end{proof}
For those initial data satisfying the assumptions of Proposition~\ref{prop:
CTY11} at every point, we could obtain the uniform existence time lower bound by
a contradiction argument: were the existence time $T$ of the Ricci
flow shorter than $\varepsilon_P^2$, then for some sequence
$t_i\nearrow T$ we could observe points $x_i\in M$ such that
$\lim_{t_i\to T}\left|\Rm_{g(t_i)}\right|_{g(t_i)}(x_i)\to \infty$; especially,
we will get $\left|\Rm_{g(t_i)}\right|_{g(t_i)}(x_i)>2\alpha
T^{-1}+\varepsilon_P^{-2}$ for all $i$ large enough, contradicting the
conclusion (\ref{eqn: pseudolocality_Rm_bound}) since $T>0$ is fixed.

In the setting of Theorem~\ref{thm: main2}, however, we could not
directly apply the pseudo-locality theorem to the Ricci flow obtained from 
Hamilton's short-time existence result \cite{Hamilton}, as the almost Euclidean
volume ratio assumption in (\ref{eqn: TianWang}) may fail drastically for the
initial data $(M,g)$ in our consideration. In order to overcome this
difficulty, we will pull the initial metric back to the universal covering
space, which could be shown to be non-collapsing under the assumption
$b_1(M)-b_1(N)=\dim M-\dim N$. By Colding's volume continuity theorem
\cite{Colding97}, we then expect to improve the lower bound for the volume
ratio of the universal covering space: 
\begin{lemma}[Almost Euclidean condition for the universal covering space]
\label{lem: almost_Euclidean}
For any $\varepsilon\in (0,1)$ fixed, there are $\delta_{AE}\in (0,1)$ and
$r_{AE}\in (0,1)$, solely determined by $\varepsilon$, $m$, $D$ and $v$, to
the following effect: if $(M,g)\in \mathcal{M}_{Rc}(m)$ and $(N,h)\in
\mathcal{M}_{Rm}(k,D,v)$ with $k\le m$ satisfy
\begin{enumerate}
  \item $d_{GH}(M,N)<\delta $ for some $\delta\le \delta_{AE}$, and
  \item $b_1(M)-b_1(N)=m-k$,
 \end{enumerate}
then for any $r\in (0,r_{AE}]$ and $\tilde{p}\in \widetilde{M}$ we
have
\begin{align}\label{eqn: almost_max_vol}
\left|B_{\pi^{\ast}g}(\tilde{p},r)\right|_{\pi^{\ast}g}\ &\ge\
(1-\varepsilon)\omega_mr^m,
\end{align}
where $\pi:\widetilde{M}\to M$ is the universal covering and we equip
$\widetilde{M}$ with the covering metric $\pi^{\ast}g$.
\end{lemma}

\begin{proof}
Fixing $\varepsilon\in (0,10^{-1})$, we let $r_1=r_1(\varepsilon)\in
(0,1)$ be the constant such that
\begin{align}\label{eqn: choosing_r_0}
\forall r\in (0,r_1],\quad (1-10^{-2}\varepsilon)\omega_mr^m\ \le\
V_{-1}^m(r)\ \le\ (1+10^{-2}\varepsilon)\omega_mr^m,
\end{align}
where $V_{-1}^m(r)$ is the volume of geodesic $r$-ball in the space form of
sectional curvature equal to $-1$. 

By Colding's volume continuity theorem, \cite[Main Lemma 2.1]{Colding97}, we
obtain  the corresponding positive constants
$\delta_C=\delta_C(10^{-2}\varepsilon)<1$,
$\Lambda_C=\Lambda_C(10^{-2}\varepsilon)<1$ and
$R_C=R_C(10^{-2}\varepsilon)>1$ for $10^{-2}\varepsilon$.
We then put $\varepsilon':= r_1\Lambda_C\delta_C R_C^{-1}$ in \cite[Proposition
5.4]{NaberZhang} to obtain some uniform positive constant
$\delta_{NZ}(\varepsilon')<1$ and $r':=r_{NZ}(\varepsilon')\in
(\delta_{NZ}(\varepsilon'),1)$. On the other hand, by the uniform
$C^{1,\frac{1}{2}}$ harmonic radius lower bound for manifold $(N,h)\in
\mathcal{M}_{Rm}(k,D,v)$, there is some uniform constant
$\bar{\iota}_1\left(m,\varepsilon', \max_{0\le k\le m}C_{hr}(k,D,v)\right)\in
(0,\bar{\iota}_0]$ such that
\begin{align*}
\forall \bar{p}\in N,\quad d_{GH}\left(B_h(\bar{p},\bar{\iota}_1),
\mathbb{B}^k(\bar{\iota}_1) \right)\ <\
10^{-1}\lambda_1 \bar{\iota}_1.
\end{align*}
where $\lambda_1:=2^{-1}\min \left\{\delta_{NZ}(\varepsilon'),
\varepsilon_{NZ}(m)\right\}$ --- notice that $\lambda_1 \bar{\iota}_1\le
\delta_{Nil}$. Following the proof of Claim (1) of Theorem~\ref{thm: main1} in
the last section, we define (compare the definition of $\delta_1$ there)
 \begin{align*}
\delta_{AE}\ :=\  10^{-1}\min\left\{ 10^{-1}
\Psi_{NZ}\left(\min\left\{\bar{C}^{-1}
\Psi_{NZ}\left(2\lambda_1\bar{\iota}_1,1,m\right),\bar{r}\right\},1,m \right),
\lambda_1 \bar{\iota}_1,10^{-3}\bar{\iota}_{hr} \right\}
 \end{align*}
 where $\bar{C}(m,D,\beta)>1$ and $\bar{r}(m,D,\beta) \in (0,1)$ are the
 uniform constants obtained from Theorem~\ref{thm: main3} by setting
 $\beta=\min\left\{10^{-3}, (4D)^{-1}\right\}$, and as before, $\Psi_{NZ}$ is
 obtained from \cite[Lemma 5.2]{NaberZhang}.
 
 If $(M,g)$ and $(N,h)$ satisfy (1) and (2) in the assumption with the
 $\delta_{AE}$ just defined, we first understand the implication of the (2) on
 the nilpotency rank of the pseudo-local fundamental group at each point of $M$. 
 It is easily seen that the estimates (\ref{eqn: dle1}), (\ref{eqn: dge1'}) and
 (\ref{eqn: dge1''}) hold with $2\lambda_1 \bar{\iota}_1$ in place of
 $2\delta_{Nil}$ --- for any $p\in M$ and any $\tilde{p}\in \pi^{-1}(p)$, we
 have
  \begin{align*}
  d_{\pi^{\ast}g}\left(\gamma_i^{k_i''}.\tilde{p},\tilde{p}\right)\ \le\
  2\lambda_1\bar{\iota}_1\ \le\ 2\delta_{Nil}
  \end{align*}
 for each $\gamma_i$ ($i=1,\ldots,b_1(M)-b_1(N)$) obtained there as a
  torsion-free short generator of $H_1^{\delta}(M;\mathbb{Z})$, with some
 $k_i''\le 2N_{NZ}(m,D,\iota)$. In particluar, $\gamma_i^{k_i''}\in
 \widetilde{G}_{\lambda_1\bar{\iota}_1}(p)$ for each $i=1,\ldots,b_1(M)-b_1(N)$.
 Lemma~\ref{lem: nil_rank}, Proposition~\ref{prop: rank} and the proof of Claim
 (1) in Theorem~\ref{thm: main1} then lead to
\begin{align*}
b_1(M)-b_1(N)\ = H_1^{\delta}(M;\mathbb{Z})\ \le\ \rank\
\widetilde{G}_{\lambda_1\bar{\iota}_1}(p)\ \le\ \rank\
\widetilde{G}_{\delta_{Nil}}(p)\ \le\ m-k,
\end{align*} 
while assumption (2) forces 
\begin{align}\label{eqn: G_max_rank}
\rank\ \widetilde{G}_{\lambda_1 \bar{\iota}_1}(p)\ =\ m-k.
\end{align}

We now examine the effect of further assuming (1). By the choice of
$\delta_{AE}$ and $\bar{\iota}_1$, we have
\begin{align*}
\forall p\in M,\quad 
d_{GH}\left(B_g(p,\bar{\iota}_1),\mathbb{B}^k(\bar{\iota}_1)\right)\ \le\
d_{GH}\left(M,N\right) +d_{GH}\left(B_h(\Phi(p),\bar{\iota}_1),
\mathbb{B}^k(\bar{\iota}_1)\right)\ <\ \frac{2}{5}\lambda_1 \bar{\iota}_1.
\end{align*}
 Moreover, performing the rescaling $g\mapsto 4\bar{\iota}_1^{-2}g=:\bar{g}_1$,
 we see that 
\begin{align}\label{eqn: dGH1}
\forall p\in M,\quad d_{GH}\left(B_{\bar{g}_1}(p,2), \mathbb{B}^k(2) \right)\ <\
2\lambda_1.
\end{align}
On the other hand, fixing any lift $\tilde{p} \in \pi^{-1}(p)$, we could see as
in the proof of Lemma~\ref{lem: nil_rank} that 
\begin{align}\label{eqn: Gdelta}
\begin{split}
\widetilde{G}_{\lambda\bar{\iota}_1}(p)\ =\ &\left\langle
\gamma\in \pi_1(M,p):\ d_{\pi^{\ast}g}(\gamma.\tilde{p},\tilde{p})\le
2\lambda_1 \bar{\iota}_1\right\rangle\\
=\ &\left\langle \gamma\in \pi_1(M,p):\
d_{\pi^{\ast}\bar{g}_1}(\gamma.\tilde{p}, \tilde{p})\le\ 4\lambda_1
\right\rangle.
\end{split}
\end{align}
Since $\pi:(\widetilde{M},\tilde{p})\to (M,p)$ is a normal covering with deck
transformation group being $\pi_1(M,p)$, the same holds for the restriction
$\pi_p:\pi^{-1}(B_{\bar{g}_1}(p,2))\to B_{\bar{g}_1}(p,2)$. 
 
Applying \cite[Proposition 5.4]{NaberZhang} to the normal covering
$\pi_p: \pi^{-1}(B_{\bar{g}_1}(p,2))\to B_{\bar{g}_1}(p,2)$ and the subgroup
$\widetilde{G}_{\lambda_1 \bar{\iota}_1}(p)$ of the deck transformation group
$\pi_1(M,p)$, we conclude, thanks to (\ref{eqn: G_max_rank}), (\ref{eqn:
dGH1}), (\ref{eqn: Gdelta}) and the choice of
$\lambda<\delta_{NZ}(\varepsilon')$, that
\begin{align*}
d_{GH}\left(B_{\pi^{\ast}\bar{g}_1}(\tilde{p},r'), \mathbb{B}^m(r')\right)\le
\varepsilon' r'.
\end{align*}

We now further rescale the metric $\bar{g}_2:= \lambda_2^{-2}
\pi^{\ast}\bar{g}_1$ with
\begin{align}\label{eqn: choosing_lambda}
 \lambda_2(\varepsilon)\ :=\ \min\left\{r_1, \Lambda_C, r'R_C^{-1}\right\},
\end{align}
then for any $p\in M$ and any $\tilde{p}\in \pi^{-1}(p)$, we have
\begin{align*}
d_{GH}\left(B_{\bar{g}_2}(\tilde{p},R_C), \mathbb{B}^m(R_C)\right)\ <\
\delta_C,
\end{align*}
and we have the Ricci curvature lower bound 
\begin{align*}
\Rc_{\bar{g}_2}\ \ge\
-(m-1)\Lambda_C^2\bar{g}_2.
\end{align*} 
Consequently, applying \cite[Main Lemma 2.1]{Colding97} we have
\begin{align*}
\tilde{p}\in \widetilde{M},\quad 
\left|B_{\bar{g}_2}(\tilde{p},1)\right|_{\bar{g}_2}\ \ge\
(1-10^{-2}\varepsilon)\omega_m.
\end{align*}
By the volume ratio comparison (\ref{eqn: choosing_r_0}) we have
\begin{align}\label{eqn: gbar_vr}
\forall \tilde{p}\in \widetilde{M},\ \forall r\in (0,1],\quad
\left|B_{\bar{g}_2}(\tilde{p},r)\right|_{\bar{g}_2}\ \ge\
(1-\varepsilon)\omega_mr^m.
\end{align}
Notice the scaling invariance of the estimate.

Now we scale back to the original metric $\pi^{\ast}g$ and the estimate
(\ref{eqn: gbar_vr})
remain valid for geodesic balls centered anywhere in $\widetilde{M}$, with
radii not exceeding $\frac{1}{2}\bar{\iota}_1\lambda_2$.
By (\ref{eqn: choosing_lambda}) and the bound of $r'\ge \delta_{NZ}'$ in
\cite[Proposition 5.8]{NaberZhang}, we have
$\frac{1}{2}\bar{\iota}_1\lambda_2$ always bounded below by
\begin{align}\label{eqn: r_AE_lb}
 r_{AE}\ :=\ \frac{1}{2}\bar{\iota}_1(\varepsilon')
\min\left\{r_1, \Lambda_C, \delta_{NZ}(\varepsilon')R_C^{-1}\right\},
\end{align}
with $\Lambda_C$, $R_C$ and $\varepsilon'$ determined by
$10^{-2}\varepsilon$ via Colding's volume continuity theorem, and
$\varepsilon_{NZ}(m)$ described in the proof of Lemma~\ref{lem: nil_rank}.
Clearly, $\delta_{AE}$ and $r_{AE}$ are determined by $\varepsilon$, $m$, $D$
and $v$.
\end{proof}

With the help of this lemma, we could then apply Proposition~\ref{prop: CTY11}
to the rescaled covering flow $(\widetilde{M},\pi^{\ast}g(t))$ to bound the
existence time uniformly from below.
\begin{proof}[Proof of Theorem~\ref{thm: main2}]
Given $\alpha\in (0,10^{-2}m^{-1})$, let $\delta_P(\alpha)\in (0,1)$ be the
almost Euclidean threshold required in (\ref{eqn: TianWang}). Given $(M,g)$ and
$(N,h)$ as in the assumption, we know that $(M,g)$ is a closed Riemannian
manifold as it is complete with finite diameter. Therefore, Hamilton's
short-time existence result applies and there is a Ricci flow solution
$(M,g(t))$ for $t\in [0,T)$ with $g(0)=g$. For
$\delta<\delta_{RF}:=\delta_{AE}(\delta_P)$ (omitting the dependence on $m$,
$D$ and $\iota$), we consider the covering Ricci flow
$(\widetilde{M},\tilde{g}(t))$ with initial data $(\widetilde{M},\pi^{\ast}g)$.
Notice that the time slices of the covering flow are complete, and satisfy 
$\|\Rm_{\tilde{g}(t)}\|_{L^{\infty}(\widetilde{M},\tilde{g}(t))} = 
\|\Rm_{g(t)}\|_{L^{\infty}(M,g(t))} <\infty$ for any $t<T$, and by
Lemma~\ref{lem: almost_Euclidean} we have 
\begin{align*}
 \forall \tilde{p}\in \widetilde{M},\quad 
\left|B_{\pi^{\ast}g}(\tilde{p},r_{AE})\right|_{\pi^{\ast}g}\ \ge\ 
(1-\delta_P)\omega_mr_{AE}^m.
\end{align*}
Rescaling $g\mapsto r_{AE}^{-2}\delta_P^{-2}g=:\bar{g}$ and $t\mapsto
r_{AE}^{-2}\delta_P^{-2}t=:\bar{t}$, we could apply Proposition~\ref{prop:
CTY11} to the Ricci flow $(\widetilde{M},\bar{g}(t))$ and conclude that the
 flow exists at least up to $\bar{t}=\varepsilon_P^2(\alpha)$. Now scaling
 back, we see that the original Ricci flow exists up to
 $T>\varepsilon_{RF}^2:=\varepsilon_P^2r_{AE}^2\delta_P^2$, and (\ref{eqn:
main2}) follows directly from (\ref{eqn: pseudolocality_Rm_bound}). We notice
that both $\delta_{RF}$ and $\varepsilon_{RF}$ are solely determined by
$m$, $D$ and $v$, besides $\alpha$.  
\end{proof}

 In order to apply Theorem~\ref{thm: main2} as a smoothing tool, we need to
 keep track of the distance change by running the Ricci flow. We have
 the following distance distortion estimate, which is a rewording of
 \cite[Lemma 1.11]{HKRX18}:
\begin{lemma}[Distance distortion]\label{lem: dis_dis}
For any $\alpha\in (0,10^{-2}m^{-1})$, under the assumptions of
Theorem~\ref{thm: main2}, there is some $\Psi_D(\alpha|m)\in(0,1)$ with
$\lim_{\alpha\to 0}\Psi_D(\alpha|m)=0$, such that for any $t\in
(0,\varepsilon_P^2]$, and for any $x,y\in M$ with $d_g(x,y)\le
\sqrt{t}$, we have
\begin{align}\label{eqn: distance_distortion}
\left|d_{g(t)}(x,y)- d_g(x,y)\right|\ \le\
\Psi_D(\alpha|m)\sqrt{t}.
\end{align}
\end{lemma}
\begin{proof}
As in the proof of Theorem~\ref{thm: main2}, we consider the universal covering
$\pi:\widetilde{M}\to M$ and we have a Ricci flow solution $\pi^{\ast}g(t)$ on
$\widetilde{M}$. Notice that for each $t\in[0,\varepsilon_P^2]$, the fundamental
group $\pi_1(M)$ acts on $(\widetilde{M},\pi^{\ast}g(t))$ by free and totally
discontinuous isometries and the Ricci flow $g(t)$ on $M$ is the quotient flow
$(\widetilde{M},\pi^{\ast}g(t))\slash \pi_1(M)$. Recall that
$(\widetilde{M},\pi^{\ast}g)$ satisfies the assumption of
Proposition~\ref{prop: CTY11} at every point, after suitable rescaling (making
$r_{AE}\mapsto 1$). By the scaling invariance, the original estimate in
\cite[Lemma 1.11]{HKRX18} descends to the flow $(M,g(t))$ and proves (\ref{eqn:
distance_distortion}).
\end{proof}

\subsection{Rigidity of the first Betti number}
With the help of Theorem~\ref{thm: main2} and Lemma~\ref{lem: dis_dis}, we now
prove Claim (2), the equality case, of Theorem~\ref{thm: main1}. Recalling that
by \cite[Theorem 2.6]{CFG92} and \cite[Theorem 2.2]{HKRX18}, we have some
dimensional constants $\varepsilon_{F}(m)\in (0,1)$ and $C_{F}(m)\ge 1$, such
that if an $m$-dimensional manifold $X$ with sectional curvature bounded by $1$
in absolute value is $(1,\varepsilon_{F})$-Gromov-Hausdorff close to a
$k$-dimensional ($k\le m$) manifold $Y$ with the same sectional curvature bound
and unit injectivity radius lower bound, then there is a fibration $F:X\to Y$,
which is also a $(2^{-1},C_{F}\varepsilon_{F})$-Gromov-Hausdorff approximation.
Here an $(r,\delta)$\emph{-Gromov-Hausdorff approximation} is a $\delta$-dense
map whose restriction to each geodesic $r$-ball is a $\delta$-Gromov-Hausdorff
approximation. We now pick the largest $\alpha_{B}(m,D,v)\in (0,10^{-2}m^{-1})$
so that $\Psi_{D}(\alpha_B|m)\le 4^{-1}\min\{\varepsilon_F,
C_F^{-1}\delta_{Nil}\}$ --- here the dependence of $\alpha_B$ on $D$ and $v$ is
due to $\delta_{Nil}(m,D,v)$, obtained in Lemma~\ref{lem: nil_rank}.

If $d_{GH}(M,N)<\delta_{RF}(\alpha_B)$ and $b_1(M)-b_1(N)=m-k$, then we could
run the Ricci flow with initial data $(M,g)$ by Theorem~\ref{thm: main2}, to
obtain a smoothing metric $g(T_B)$ with
$T_B:=\min\{\varepsilon_{RF}(\alpha_B)^2,\bar{\iota}_{hr}^2\}$,
satisfying $\left\|\Rm_{g(T_B)}\right\|_{L^{\infty}(M,g(T_B))} \le 2T_B^{-1}$. 
On the other hand, by Lemma~\ref{lem: dis_dis} we know that
$(M,2T_B^{-1}g(T_B))$ and $(M,2T_B^{-1}g)$ are
$(1,\Psi_D(\alpha_B))$-Gromov-Hausdorff close, meaning that the identity map
restricts to a $\Psi_D(\alpha_B)$-Gromov-Hausdorff approximation on any geodesic
unit ball in $(M,2T_B^{-1}g(T_B))$ or $(M,2T_B^{-1}g)$. Therefore, setting 
\begin{align*}
\delta_B(m,D,v)\ :=\ \frac{1}{10}\min \left\{\delta_{RF}(\alpha_B),
\varepsilon_{F} T_B^{\frac{1}{2}}, C_F^{-1}\delta_{Nil} \right\},
\end{align*}
we know that $(M,2T_B^{-1}g(T_B))$ and $(N,2T_B^{-1}h)$ are
$(1,\frac{1}{2}\min\{\varepsilon_F,2C_F^{-1}\delta_{Nil}\})$-Gromov-Hausdorff
close to each other, whenever $d_{GH}(M,N)=\delta<\delta_B$. Moreover, both
$(M,2T_B^{-1}g(T_B))$ and $(N,2T_B^{-1}h)$ have sectional curvature uniformly
bounded by $1$ in absolute value, and $(N,2T_B^{-1}h)$ has injectivity radius
everywhere bounded below by $1$. Now applying \cite[Theorem 2.2]{HKRX18}, we
obtain an infranil fibration $F:M\to N$, which is also a
$(2^{-1}, 2^{-1}\min\{C_F\varepsilon_F,\delta_{Nil}\})$-Gromov-Hausdorff
approximation. Especially, for any $p\in M$, the fiber $F_p$ has extrinsic
diameter $\diam_{2T_B^{-1}g(T_B)} F_p\le 2^{-1}\delta_{Nil}$. Consequently,
$\diam_{2T_B^{-1}g} F_p\le \delta_{Nil}$, and as $T_B<1$, we have $\diam_g
F_p\le \delta_{Nil}$ for any $p\in M$.

Notice that each $F$-fiber is diffeomorphic to an infranil manifold of dimension
$m-k$, and we are yet to check that the fibers are actually toral. We now fix an
arbitrary $p\in M$. By the fiber bundle structure $F:M\to N$ and the assumption
$\delta< \delta_B$, we know that $\widetilde{G}_{\delta_{Nil}}(p)\cong
\tilde{\Gamma}_{\delta_{Nil}}(p) \cong \pi_1(F_p,p)$, since $\diam_{g}F_p\le
\delta_{Nil}$ and the base $N$ is homotopically trivial at the scale
$\bar{\iota}_{hr} > 10 \delta_{Nil}$. On the other hand, by Remark~\ref{rmk:
homomorphism}, we know that the abelianization of
$\widetilde{G}_{\delta_{Nil}}(p)$ has rank at least $b_1(M)-b_1(N)$. We have
$\rank\ \widetilde{G}_{\delta_{Nil}}(p)\le m-k$ by Lemma~\ref{lem: nil_rank},
and by the structure of the finitely generated almost nilpotent groups we have
\begin{align*}
b_1(M)-b_1(N)\ \le\ \rank\ \widetilde{G}_{\delta_{Nil}}(p)\slash
\left([\pi_1(M,p),\pi_1(M,p)]\cap \widetilde{G}_{\delta_{Nil}}(p) \right)\ \le\
\rank\ \widetilde{G}_{\delta_{Nil}}(p)\ \le\ m-k.
\end{align*}
Now the assumption $b_1(M)-b_1(N)=m-k$ forces the almost nilpotent group
$\widetilde{G}_{\delta_{Nil}}(p)$ to have the same rank as its abelianization, 
which is the case only when $\widetilde{G}_{\delta_{Nil}}(p)$ is a finitely
generated abelian group. Therefore, we have shown that each fiber $F_p$ has
abelian fundamental group. As it is an infranil manifold, we thus know that $F_p$ is diffeomorphic to an $(m-k)$-torus $\mathbb{T}^{m-k}$.
Therefore, the proof of Claim (2) of Theorem~\ref{thm: main1} is complete. 

\begin{remarkin}\label{rmk: Hongzhi20}
	After finishing this work, a fiber bundle theorem for collapsing manifolds with the so-called ``local bounded Ricci covering geometry" appears in \cite{Hongzhi20}. While the collapsing manifolds in Theorem~\ref{thm: main1} with maximal first Betti number differences are shown through Sections 3 and 4 to satisfy the assumptions in \cite[Theorem 0.3]{Hongzhi20}, this theorem fails to provide the structure of the fibers. The proof of \cite[Theorem 0.3]{Hongzhi20} could neither be applied to the case when the collapsing limit is a singular orbifold, as done in \cite{HW20b}. We recently learn that an upcoming work of Rong \cite{Rong21} will provide an alternative proof of Theorem~\ref{thm: main1}, purely relying on techniques from metric Riemannian geometry and independent of the Ricci flow smoothing technique.
\end{remarkin}

\section{Further discussions}


The torus fibration structure in Theorem~\ref{thm: main1} is not the end of journey. 
Under what extra conditions can we simplify the topological structure of the collapsing manifolds? 
In fact, if  $\pi_1(N)=0$ and $b_1(M)=\dim M-\dim N$,  it is easy to see that $M\cong N\times \mathbb{T}^{m-k}$
 as smooth manifolds.  
 This can be done purely by a topological argument.   The general discussion of product structure under the assumptions in  Theorem~\ref{thm: main1}
 will appear elsewhere. 
 

On the other hand, it is natural to extend 
Theorem~\ref{thm: main1} for generic collapsing limit spaces --- we notice
that if $X$ is a compact Ricci limit space, the generalized first Betti number
$b_1(X)$ is well-defined; see \cite[Remark 7.22]{Honda}. In this direction, the
study of local Ricci bounded covering geometry pioneered by Rong
\cite{HKRX18} should provide useful tools, and the localization of the
short first homology group, as well as Theorem~\ref{thm: main2}, will be
inevitable; see also \cite{HW20b} for a local Ricci flow smoothing result for 
collapsing manifolds near lower-dimensional orbifold limits.

\subsection*{Acknowledgements.} 

We would like to thank Xiaochun Rong for enlightening discussions and his warm encouragement. 
The first named author thanks Song Sun for several valuable comments. 
The second named author is grateful to Xin Peng for reading the early version of the manuscript and pointing out a mistake in ``further discussions" and many typos. 
The second named author is partially supported by YSBR-001,  the General Program
of the National Natural Science Foundation of China (Grant No. 11971452) and a research fund of USTC.

\vspace{1 in}

\end{document}